\documentclass[review]{elsarticle}

\usepackage{hyperref}

\journal{}

\usepackage{geometry}
\usepackage{amssymb,amsmath,amssymb,mathrsfs}
\usepackage{array}
\newtheorem{theorem}{Theorem}[section]

\newtheorem{lemma}{Lemma}[section]

\newtheorem{assumption}{Assumption}[section]
\newtheorem{remark}{Remark}[section]
\newproof{proof}{Proof}
\numberwithin{equation}{section}

\usepackage{color}
\definecolor{red}{rgb}{1,0,0} 
\definecolor{blue}{rgb}{0,0,1} 

\begin{document}
	
	\begin{frontmatter}
		
		\title{Error estimates of a bi-fidelity method for a multi-phase Navier-Stokes-Vlasov-Fokker-Planck system with random inputs}

		\author[mymainaddress,mysecondaryaddress]{Yiwen Lin 
		}
		\ead{linyiwen@sjtu.edu.cn}
		
		\author[mymainaddress,mysecondaryaddress,mythirdaddress]{Shi Jin}
		\ead{shijin-m@sjtu.edu.cn}

		
		\address[mymainaddress]{School of Mathematical Sciences, Shanghai Jiao Tong University, Shanghai 200240, China}
		\address[mysecondaryaddress]{Institute of Natural Sciences, Shanghai Jiao Tong University, Shanghai 200240, China}
		\address[mythirdaddress]{Ministry of Education, Key Laboratory in Scientific and Engineering Computing, Shanghai Jiao Tong University, Shanghai 200240, China}

		\begin{abstract}
			
			Uniform error estimates of a bi-fidelity method for a kinetic-fluid coupled model with random initial inputs in the fine particle regime are proved in this paper. Such a model is a system coupling the incompressible Navier–Stokes equations to the Vlasov–Fokker–Planck equations for a mixture of the flows with \textit{distinct} particle sizes. 
			The main analytic tool is the hypocoercivity analysis for the multi-phase Navier-Stokes-Vlasov-Fokker-Planck system with uncertainties, considering solutions in a perturbative setting near the global equilibrium. This allows us to obtain the error estimates in both kinetic and hydrodynamic regimes.

		\end{abstract}
		
		
	\end{frontmatter}
	
	
	\section{Introduction}

	The study of  kinetic-fluid models for a mixture of flows, where particles represent the dispersed phase evolving in a dense fluid, is motivated by applications such as smoke or dust dispersion \cite{Friedlander1977}, biomedical spray modeling \cite{Baranger2005} and combustion theory \cite{Williams1985}. 
	For more details on the modeling of such multi-phase flows, see \cite{Prosperetti2007}.
	Mathematically, these problems are described by	partial differential equations, where the evolution of the particle distribution function is driven by kinetic equations with a combination of particle transport, Stokes drag force exerted by the surrounding fluid and Brownian motion, while the evolution of the fluid is driven by Euler or Navier-Stokes equations.
	This can lead to the Navier-Stokes-Vlasov-Fokker-Planck system, first proposed in \cite{Goudon2004a, Goudon2004b}. 
	
	So far most of the references do not take into account the impact of particle size variations, as they assume that all particles are of the same size. However, this paper focuses on developing a kinetic-fluid model for a mixture of particles with {\it {distinct}} sizes. Such multi-size particle systems have various applications in engineering, particularly in complex meteorological simulations of large aircraft icing processes. It is essential to consider the distribution of droplets in the air, as the influence of larger droplets in the distribution cannot be overlooked \cite{Cober2006,Hauf2006,Potapczuk2019}.
	Notice that simulating multi-size particles through experimental means can be very challenging, yet, the behavior of fluid-particle systems with particles of distinct sizes and uncertainties can be studied mathematically. 
	It is known that solving deterministic kinetic-fluid coupled equations is time expensive due to their high-dimensional nature in the phase space. When considering random uncertainties described by kinetic equations with random parameters in initial or boundary data, and the modeling of drag force and particle diffusion \cite{Xiu2009, Xiu2010, Gunzburger2014, JinLin2022AP}, the computational dimension becomes even higher, making it computationally challenging. Indeed, it is significant to quantify these uncertainties because such quantification helps us understand how uncertainties affect the solution, leading to reliable predictions.
	While the standard stochastic collocation method, commonly employed for high-dimensional parametric PDEs \cite{Nobile2008, Xiu2005}, can be used, it is not ideal to repeatedly solve the kinetic-fluid equation numerically, given the complex coupling and nonlinearity involved in multi-phase kinetic-fluid systems.

	Often less complex low-fidelity models are used to approximate systems with simplified physics or on a coarser physical mesh to reduce computational costs. Although they may sacrifice accuracy, the low-fidelity models are designed to capture essential features of the system and provide reliable predictions. In fluid dynamics, hydrodynamic limit equations can be used as low-fidelity models for kinetic equations.
	In a recent study \cite{LiuZhu2020}, Liu and Zhu used a bi-fidelity method from \cite{Narayan2014,Zhu2014} to efficiently compute high-fidelity solutions of the Boltzmann equation with multidimensional random parameters and multiple scales. This approach can effectively reduce the computational cost while maintaining reasonable accuracy.
	Gamba, Jin, and Liu \cite{Gamba2021} have conducted uniform error estimates of the bi-fidelity method, focusing on important examples such as the Boltzmann and linear transport equations. However, for kinetic-fluid coupled equations, there is currently limited research on bi-fidelity methods.

	We will consider our problems with uncertainty, characterized by random inputs in the initial data.  To model the uncertainty here, let the velocity of the fluid and the distribution functions of particles with $N$ different sizes depend on the random variable $z$ (i.e., $u = u(t,x, z)$ and $F_i = F_i(t,x,v, z), i=1,2,\ldots,N$), which lives in the random space $\mathbb{Z}$ with probability distribution $\pi(z) dz$. The uncertainty from the initial data is then described by letting the initial data $u_0$ and $\{F_{i,0}\}_{i=1}^N$ depend on $x$ and $z$. 
	
	Let $\varepsilon$ be the varying Stokes number, which describes the ratio of the Stokes settling time over a certain time unit of observation. The goal of this work is to provide a uniform-in-$\varepsilon$  error estimate of the bi-fidelity method for solving multi-phase kinetic-fluid equations with high-dimensional uncertainties. 
	The error estimate is obtained by splitting the error between the high- and bi-fidelity solutions into three parts: the error between the high- and low- fidelity solutions, the projection for the low-fidelity solutions and the remainder. 
	Our analysis is based not only on the hypocoercivity analysis for multi-phase kinetic-fluid equations with uncertainties \cite{JinLin2022} which shows that for near-equilibrium initial data, the solution always  preserves the  regularity of the initial data and is insensitive to random perturbations of the initial data for large time, where the solution is close to the global equilibrium, 	
	but also on the best-$N$ approximation theory and \cite{Cohen2015} which provides an upper bound for the Kolmogorov $N$-width. 
	 It is known that the multi-fidelity estimate error bounds are rarely sharp \cite{Narayan2014, LiuZhu2020} if the general high- and low-fidelity solutions are treated as vector data instead of including good properties on specific models. In this paper, we adopt the theoretical framework developed in \cite{Gamba2021}, use the good properties of the limiting system and give a sharper error estimate than \cite{Narayan2014, LiuZhu2020}. Boltzmann equation and  linear transport equations are considered in [6] while kinetic-fluid coupled system with distinct particle sizes which is a more complex model will be studied in this paper.

	One of the major difficulties in the analysis of the bi-fidelity method is that the hydrodynamic limit is not valid in the kinetic regime, yet, in the perturbative setting, one can still establish an error estimate between the kinetic-fluid equation and its hydrodynamic limit, as their velocity moments are close. In contrast to the scaling used in previous studies such as \cite{Goudon2004a,Goudon2004b,ShuJin2018,JinLin2022}, where  the scaling coefficient of the perturbative part are not  explicitly tracked, we need to introduce and track a different perturbative parameter in our analysis to obtain the error in the kinetic regime. This enables us to obtain uniform error estimates for the bi-fidelity method from both the kinetic and fluid regimes in the perturbative setting, even when the hydrodynamic limit (the low-fidelity) is not a good approximation of the kinetic-fluid equation.
	
	Our analysis is based on the property that the moments of the kinetic-fluid equations and its hydrodynamic limit are close in the perturbative setting. This implies that the bi-fidelity model does not necessarily have to be the hydrodynamic limit equation. Instead, alternative moment models in velocity, such as moment closure models, that capture the moments as accurate as the kinetic-fluid systems in the perturbative setting, may be employed as the low-fidelity model. In our analysis, we use multi-phase kinetic-fluid systems with distinct particle sizes as illustrative examples together with necessary remarks on two-phase kinetic-fluid systems. 
	
	This paper is organized as follows. Section \ref{sec:bifidelity} gives an introduction to the bi-fidelity stochastic collocation method for solving PDEs with random parameters. In Section \ref{sec:model}, we present a detailed description of the multi-phase kinetic-fluid system of interest, including the hydrodynamic limit system and the system near the global equilibrium.  Section \ref{sec:estimate} is dedicated to establishing uniform-in-$\varepsilon$ error estimates between the high-fidelity and bi-fidelity solutions, by choosing hydrodynamic limit equations of the kinetic-fluid (high-fidelity) equations as their low-fidelity models and the proofs are given in Section \ref{sec:proof}.	In Section \ref{sec:other}, other possible choices of low-fidelity models are discussed, in particular, choosing the same kinetic-fluid equation but solving by using a coarser physical mesh. The paper is concluded in Section \ref{sec:conclusion}.

	\section{A bi-fidelity stochastic collocation method}
	\label{sec:bifidelity}
	
	In this section, we review the efficient bi-fidelity stochastic collocation method studied in \cite{Narayan2014,Zhu2014}. Denote by $\textbf{u}^H, \textbf{u}^L$ the high- and low-fidelity solutions respectively, while $\textbf{u}^B$ is the bi-fidelity solution, which is an approximation of $\textbf{u}^H$. The bi-fidelity algorithm consists of two stages. In the offline stage, one uses the inexpensive low fidelity model to explore the random parameter space. Then one selects the most important parameter points $\gamma_K$ by using the greedy procedure, similar to the popular reduced-basis method for solving parameterized partial differential equations \cite{Peterson1989}.
	
	During the online stage, the bi-fidelity approximation is realized by applying exactly the same approximation rule learned from the low-fidelity model for any given $z$. For any given sample point $z \in I_z$, we project the low-fidelity solution $\textbf{u}^L(z)$ onto the low-fidelity approximation space $U^L\left(\gamma_K\right)$ :
	\begin{eqnarray}\label{defuL}
		\textbf{u}^L(z) \approx \mathcal{P}_{U^L(\gamma_K)}\left[\textbf{u}^L(z)\right]=\sum_{k=1}^K c_k(z) \textbf{u}^L\left(z_k\right),
	\end{eqnarray}
	where $\mathcal{P}_{U^L\left(\gamma_K\right)}$ is the projection operator onto the space $U^L\left(\gamma_K\right)$ with corresponding projection coefficients $\left\{c_k\right\}$, which are computed using the Galerkin approach:
	\begin{eqnarray}\label{GLsystem}
		\mathbf{G}^L \mathbf{c}=\mathbf{f}, \quad \mathbf{f}=\left(f_k\right)_{1 \leq k \leq K}, \quad f_k=\left\langle \textbf{u}^L(z), \textbf{u}^L\left(z_k\right)\right\rangle^L .
	\end{eqnarray}
	Here $\mathbf{G}^L$ is the Gramian matrix of $\textbf{u}^L\left(\gamma_K\right)$,
	\begin{eqnarray}\label{Gramianmatrix}
		\left(\mathbf{G}^L\right)_{i j}=\left\langle \textbf{u}^L\left(z_i\right), \textbf{u}^L\left(z_j\right)\right\rangle^L, \quad 1 \leq i, j \leq K,
	\end{eqnarray}
	where $\langle\cdot, \cdot\rangle^L$ is the inner product associated with the approximation space $U^L\left(\gamma_K\right)$. These low-fidelity coefficients $\left\{c_k\right\}$ act as surrogates of the corresponding high-fidelity coefficients of $\textbf{u}^H(z)$. The bi-fidelity approximation of $\textbf{u}^H$ can then be constructed as follows:
	\begin{eqnarray}\label{uB}
		\textbf{u}^B(z)=\sum_{k=1}^K c_k(z) \textbf{u}^H\left(z_k\right) .
	\end{eqnarray}
	The key idea of this method is to first represent the solution $\textbf{u}^L(z)$ as coordinates in the basis $U^L(\gamma)$, which is assumed to be a collection of linearly independent solutions, followed by using exactly the same coordinate coefficients $c_n$ to reconstruct the bi-fidelity solution \eqref{uB}. In practice, the number of low-fidelity basis is typically small, and the cost of computing the projection coefficients $c$ by solving the linear system \eqref{GLsystem} is negligible. Hence the dominant cost of the online step is just one low-fidelity simulation run. Notably, the low-fidelity coefficients can be a good approximation of the corresponding high-fidelity coefficients for a given sample $z$ if the low-fidelity model can mimic the variations of the high-fidelity model in the parameter space.

	To obtain the error estimate of $\textbf{u}^H-\textbf{u}^B$ in general, we split the total error by inserting the information of $\textbf{u}^L$. Specifically, we use the following approach:
	\begin{eqnarray}\label{eqsThreeParts}
		\begin{aligned}
			& \textbf{u}^H(z)-\textbf{u}^B(z)\\
			= & \textbf{u}^H(z)-\sum_{k=1}^K c_k(z) \textbf{u}^H\left(z_k\right) \\
			= & \textbf{u}^H(z)-\textbf{u}^L(z)+\left(\textbf{u}^L(z)-\sum_{k=1}^K c_k(z) \textbf{u}^L\left(z_k\right)\right) 
			+\sum_{k=1}^K c_k(z)\left(\textbf{u}^L\left(z_k\right)-\textbf{u}^H\left(z_k\right)\right),
		\end{aligned}
	\end{eqnarray}
	where the second term denotes the projection error of the greedy algorithm, and it remains to estimate $\textbf{u}^H(z)-\textbf{u}^L(z)$ in appropriate norms.

	\section{Multi-phase flow model in the fine particle regime}
	\label{sec:model}
	
	In this paper, our main focus is on models that describe a large number of particles, {\it with distinct but fixed sizes}, interacting with fluids. External potentials such as gravity, electrostatic force, and centrifugal force, as well as coagulation-fragmentation events that can lead to changes in particle sizes are ignored in our study.	Simiar to fluid dynamics, the dense fluid phase is modeled as a liquid or dense gas characterized by macroscopic quantities like mass density, velocity and temperature, and is governed by the incompressible Navier-Stokes equations, which depend on both time and space variables. On the other hand, particles such as droplets and bubbles dispersed in the fluid are modeled using distribution functions in phase space and are described by the Vlasov-Fokker-Planck kinetic equations that depend on time, space and microscopic particle velocity. Note that the unknowns for different phases may not depend on the same set of variables, and particles and fluid systems are coupled through nonlinear forcing terms. Such coupling and nonlinearities pose new  difficulties in both mathematical analysis and numerical computations compared to uncoupled problems.  
	
	In the fine particle regime, the suitably scaled PDE systems for a mixture of flows with uncertainties are given by:
	\begin{equation}\label{ModelEquation}
		\left\{\begin{aligned}
			&\begin{aligned}
				(F_i)_{t}+v \cdot \nabla_{x} (F_i)&=\dfrac{1}{\varepsilon}\dfrac{1}{i^{2/3}}\operatorname{div}_{v}\left((v-u) F_i+\dfrac{\bar{\theta}}{i}\nabla_{v} F_i\right), \\
				&\hspace{6em} (t, x, v, z) \in \mathbb{R}^{+} \times \mathbb{T}^{3} \times \mathbb{R}^{3} \times \mathbb{Z}^{d_z}, i=1,2,...,N, 
			\end{aligned}\\
			&u_{t}+(u \cdot \nabla_{x}) u+\nabla_{x} p-\Delta_{x} u=\dfrac{\kappa}{\varepsilon} \sum_{i=1}^N\int_{\mathbb{R}^{3}}(v-u) F_i i^{1/3} \mbox{d} v, \quad(t, x,z) \in \mathbb{R}^{+} \times \mathbb{T}^{3}, \\
			&\nabla_{x} \cdot u=0,
		\end{aligned}\right.
	\end{equation}	 
	with the initial condition that depends on $z$:
	$$\left.u\right|_{t=0}=u_{0}, \quad \nabla_{x} \cdot u_{0}=0,\left.\quad F_i\right|_{t=0}=F_{i,0},$$
	where $u_{0}=u(0,x,z)$, $F_{i,0}=F_i(0,x,v,z), i=1,2,\ldots,N$.
	Here $t\geq 0$ is  time, $x\in \mathbb{T}^{3}$ is the space variable, and $v\in \mathbb{R}^{3}$ is the particle velocity.
	For simplicity the  periodic boundary condition in the space domain is assumed. The fluid is described by its velocity field $u=u(t,x,z)\in \mathbb{R}^{3}$ and its pressure $p(t,x,z)$. The particles are described by their distribution function $F_i = F_i(t,x,v,z), i=1,2,\ldots,N$ in phase space. $N$ is the number of sizes of particles. $\bar{\theta}$ is the reference temperature.
	$\varepsilon$ is the Stokes number, which satisfies, without loss of generality, $0<\varepsilon\leq 1$. $\varepsilon=O(1)$ corresponds to the kinetic regime, while $\varepsilon\rightarrow 0$ corresponds to the fluid regime. $\kappa>0$ is the coupling constant, which equals the ratio between the particle density and fluid density.

	\subsection{The hydrodynamic limit}
	
	For the multi-size particle-fluid systems \eqref{ModelEquation}, we associate to $F_i(t,x,v,z), i=1,2,\ldots,N$ the following macroscopic quantities:
	$$
	\begin{gathered}
		n_{i}(t, x,z)=\int_{\mathbb{R}^{3}} F_i(t, x, v,z) \mathrm{d} v, \quad \rho_i(t,x) = i n_i(t,x,z), 
		\quad J_{i}(t, x,z)= i \int_{\mathbb{R}^{3}} v F_i(t, x, v,z) \mathrm{d} v, \\
		\mathbb{P}_{i}(t, x,z)=i \int_{\mathbb{R}^{3}} v \otimes v F_i(t, x, v,z) \mathrm{d} v ,
	\end{gathered}
	$$
	where $\rho_i, J_i$ and $\mathbb{P}_{i}$ are the mass, momentum and stress tensors, respectively, of particles of size $i$.
	Integrating the first equation in \eqref{ModelEquation} with respect to $i \mathrm{~d} v$ and $i v \mathrm{~d} v$ respectively, one obtains
	$$
	i \partial_{t} n_{i}+\nabla_{x} \cdot J_{i}=0,
	$$
	and
	$$\partial_{t} J_{i}+\operatorname{Div}_{x} \mathbb{P}_{i}=-\frac{1}{\varepsilon} \frac{1}{i^{2 / 3}}\left(J_{i}-i n_{i} u\right).$$
	Note that the  system \eqref{ModelEquation} conserves the total momentum since
	\begin{equation}\label{eq:kappaJ}
		\partial_{t}\left(u+ \kappa \sum_{i=1}^{N} J_{i}\right)+\operatorname{Div}_{x}\left(u \otimes u+\kappa \sum_{i=1}^{N} \mathbb{P}_{i}\right)+ \nabla_{x} p-\Delta_{x} u=0.
	\end{equation}
	Accordingly, for $\varepsilon\ll 1$,  $J_i$ and $\mathbb{P}_i$ are approximated by the moments of the Maxwellian, i.e., 
	$$J_i \simeq i n_i u, \quad \mathbb{P}_i \simeq i n_i u \otimes u+ i n_i  \mathbb{I}.$$
	Inserting this ansatz into \eqref{eq:kappaJ}, one arrives at
	\begin{equation}
		\partial_{t}\left(1\left(1+ \kappa \sum_{i=1}^{N}  i n_i\right)u\right)+\operatorname{Div}_{x}\left(\left(1+ \kappa \sum_{i=1}^{N}  i n_i\right)u \otimes u\right)+ \nabla_{x} \left(p+ \kappa \sum_{i=1}^{N}  i n_i\right)-\Delta_{x} u=0. 
	\end{equation}
	
	Thus, as $\varepsilon\rightarrow 0$,  the system \eqref{ModelEquation} formally  has a hydrodynamic limit
	\begin{equation}\label{eq:iniu}
		\left\{\begin{aligned}
			& \partial_{t} n_i +\nabla_{x} \cdot(n_i u)=0, \\
			&\partial_{t}\left(\left(1+ \kappa \sum_{i=1}^{N}  i n_i\right)u\right)+\operatorname{Div}_{x}\left(\left(1+ \kappa \sum_{i=1}^{N}  i n_i\right)u \otimes u\right)+ \nabla_{x} \left(p+ \kappa \sum_{i=1}^{N}  i n_i\right)-\Delta_{x} u=0, \\
			&\nabla_{x} \cdot u=0.
		\end{aligned}\right.
	\end{equation}
	Denote $\rho = \displaystyle\sum_{i=1}^{N}  i n_i$.
	Then system \eqref{ModelEquation} becomes
	\begin{equation}\label{eq:nu}
		\left\{\begin{aligned}
			& \partial_{t} \rho +\nabla_{x} \cdot(\rho u)=0, \\
			&\partial_{t}\left(\left(1+ \kappa \rho\right)u\right)+\operatorname{Div}_{x}\left(\left(1+ \kappa \rho\right)u \otimes u\right)+ \nabla_{x} \left(p+ \kappa \rho \right)-\Delta_{x} u=0, \\
			&\nabla_{x} \cdot u=0,
		\end{aligned}\right.
	\end{equation}
	which is the incompressible Navier-Stokes system for the composite and inhomogeneous density $(1+\kappa \rho)$.

	\subsection{Solutions close to the  equilibrium}
	
	Consider the local normalized Maxwellian
	$$\mu_i(v)=\frac{1}{(\frac{2 \pi \bar{\theta}}{i})^{3 / 2}|\mathbb{T}|^{3}} e^{-\frac{i v^{2}}{2\bar{\theta}} },$$
	and look at solutions of \eqref{ModelEquation} in the form
	\begin{equation}\label{eq:F_i0}
		F_i = \mu_i + \delta \sqrt{\mu_i} f_i.
	\end{equation}
	Plugging \eqref{eq:F_i0} into \eqref{ModelEquation}, one obtains the following system for the perturbation $(u,\{f_i\}_{i=1}^N)$:
	\begin{equation}\label{eq:uf}
		\left\{\begin{aligned}
			&(f_i)_{t}+v \cdot \nabla_{x} (f_i)+\frac{1}{i^{2/3}\varepsilon}u \cdot\left(\nabla_{v}- \frac{ i v}{2 \bar{\theta}} \right)f_i-\frac{i^{1/3}}{\bar{\theta}\varepsilon\delta}u \cdot v \sqrt{\mu_i}=\frac{1}{i^{2/3}\varepsilon}\left(-\frac{i}{\bar{\theta}}\frac{|v|^{2}}{4}+\frac{3}{2} +\frac{\bar{\theta}}{i}\Delta_{v} \right)f_i, \\
			&u_{t}+u \cdot \nabla_{x} u+\nabla_{x} p-\Delta_{x} u+\frac{\kappa}{\varepsilon}u\sum_{i=1}^{N} i^{1/3} +\frac{\kappa\delta}{\varepsilon} u \sum_{i=1}^{N}i^{1/3}  \int_{\mathbb{R}^{3}} \sqrt{\mu_i} f_i  \mathrm{d} v-\frac{\kappa\delta}{\varepsilon}\sum_{i=1}^{N} i^{1/3}  \int_{\mathbb{R}^{3}} v \sqrt{\mu_i} f_i \mathrm{d} v=0, \\
			&\nabla_{x} \cdot u=0,
		\end{aligned}	
		\right.
	\end{equation}
	with the initial data
	\begin{equation}
		\left.u\right|_{t=0}=u_{0},\left.\quad f_i\right|_{t=0}=f_{i,0},
	\end{equation}
	\begin{equation}\label{initialu0}
		\int_{\mathbb{T}^{3}} u_{0} \mathrm{d} x+\sum_{i=1}^{N}\int_{\mathbb{T}^{3}} \int_{\mathbb{R}^{3}}i\delta v \sqrt{\mu_i} f_{i,0} \,\mathrm{d} v \mathrm{d} x=0, \quad \quad \nabla_x \cdot u_{0}=0,
	\end{equation}
	and
	\begin{equation}\label{initialu0_2}
		\int_{\mathbb{T}^{3}}\int_{\mathbb{R}^{3}} \sqrt{\mu_i} f_{i,0} \,\mathrm{d} v \mathrm{d} x=0.
	\end{equation}
	
	Define the mean fluid velocity
	$$
	\bar{u}(t, z) \stackrel{\text { def }}{=} \frac{1}{|\mathbb{T}|^{3}} \int_{\mathbb{T}^{3}} u(t, x, z) \mathrm{d} x .
	$$
	Averaging the second equation in \eqref{eq:uf} yields 
	\begin{equation}\label{eq:ubar}
		\bar{u}_{t} +\frac{\kappa}{\varepsilon} \sum_{i=1}^{N} i^{1/3} \bar{u}+\frac{\kappa}{\varepsilon}\frac{1}{|\mathbb{T}|^{3}}\sum_{i=1}^{N} \int_{\mathbb{T}^{3}}  \int_{\mathbb{R}^{3}} \delta \sqrt{\mu_i} u f_i i^{1/3}\mathrm{d} v \mathrm{d} x = \frac{\kappa}{\varepsilon}\frac{1}{|\mathbb{T}|^{3}}\sum_{i=1}^{N} i^{1/3}  \int_{{\mathbb{T}}^{3}}\int_{\mathbb{R}^{3}} \delta v \sqrt{\mu_i} f_i \mathrm{d} v \mathrm{d} x .
	\end{equation}
	On the other hand, the momentum conservation 
	together with the condition \eqref{initialu0} implies
	\begin{equation}\label{eq:ubar2}
		-\frac{1}{|\mathbb{T}|^{3}}\sum_{i=1}^{N}\int_{{\mathbb{T}}^{3}}\int_{{\mathbb{R}}^{3}}i \delta v \sqrt{\mu_i} f_i \mathrm{d} v \mathrm{d} x=\frac{1}{|\mathbb{T}|^{3}}\int_{\mathbb{T}^{3}} u \mathrm{d} x=\bar{u} .
	\end{equation}

	\section{Error estimates for solutions close to the equilibrium}
	\label{sec:estimate}
	In this section, we conduct the error estimates of the bi-fidelity method for the multi-component Vlasov-Fokker-Planck-Navier-Stokes equations under different scalings. This analysis is based upon the framework established in \cite{Gamba2021}. In the kinetic-fluid coupled equations, there exist small parameters known as the Stokes number, defined as the ratio of the Stokes settling time over a certain time unit of observation. As the Stokes number $\varepsilon$ goes to $0$, the kinetic-fluid equations tend to their hydrodynamic limit, which governs a macroscopic behavior at a lower computational cost and can be chosen as an efficient low-fidelity model. However, when $\varepsilon=O(1)$, the macroscopic models may no longer be close to  the kinetic ones, making it challenging to obtain error estimates. Nonetheless, in the perturbative setting where solutions are close to global equilibrium, it is still possible to obtain error estimates using the hypocoercivity argument \cite{JinLin2022}.
	
	Let $\textbf{u}_i^H$ be the high-fidelity solution of $F_i$, which is defined by the macroscopic moments of density and momentum that are obtained from the distribution $F_i$ solved by the Vlasov-Fokker-Planck-Navier-Stokes equations \eqref{ModelEquation}:
	$$\mathbf{u}_{i}^H=\int_{\mathbb{R}^{3}}\left(\begin{array}{c}
		i \\
		iv 
	\end{array}\right) F_i(v) \mathrm{d} v:=\left(\begin{array}{c}
		\rho_i^H \\
		\rho_i^H u^H 
	\end{array}\right):=\left(\begin{array}{c}
		\mathbf{u}_{i,1}^H \\
		\mathbf{u}_{i,2}^H 
	\end{array}\right).  $$
	Then define the high-fidelity solution $u^H$ of the multi-phase kinetic-fluid system \eqref{ModelEquation} by the summation of $u_i^H$, i.e.,
	\begin{eqnarray}\label{def:uh}
		\mathbf{u}^H= \sum_{i=1}^N \mathbf{u}_{i}^H = \sum_{i=1}^N \int_{\mathbb{R}^{3}}\left(\begin{array}{c}
			i \\
			iv 
		\end{array}\right) F_i(v) \mathrm{d} v:=\left(\begin{array}{c}
			\sum_{i=1}^N \rho_i^H \\
			\sum_{i=1}^N \rho_i^H u^H 
		\end{array}\right):=\left(\begin{array}{c}
			\mathbf{u}_{1}^H \\
			\mathbf{u}_{2}^H 
		\end{array}\right).
	\end{eqnarray}
	
	In order to adopt the hypocoercivity theory and study the long-time behavior of the solution to the Vlasov-Fokker-Planck-Navier-Stokes equations \eqref{ModelEquation}, one has to consider the perturbative setting. 
	In this framework, let $\{f_i\}_{i=1}^N$ be the perturbed solution; consider the ansatz
	$$
	F_i = \mu_i + \delta \sqrt{\mu_i} f_i,
	$$
	where parameter $\delta$ is assumed sufficiently small and independent of $\varepsilon$. Denote the steady state by 
	$\mathbf{u}_i^{\text {st }}=(1,0)^T=\left(\mathbf{u}_{i,1}^{s t}, \mathbf{u}_{i,2}^{s t}\right)$, 
	then
	\begin{eqnarray}\label{uH_ust}
		\mathbf{u}_i^H=\mathbf{u}_i^{s t}+\delta \int_{\mathbb{R}^{3}}\left(\begin{array}{c}
			i \\
			iv 
		\end{array}\right) \sqrt{\mu_i} f_i(v) \mathrm{d} v.
	\end{eqnarray}
	
	We let the low-fidelity solution
	$$
	\mathbf{u}^L=\left(\begin{array}{c}
		\rho^L \\
		\rho^L u^L 
	\end{array}\right):=\left(\begin{array}{c}
		\mathbf{u}_1^L \\
		\mathbf{u}_2^L
	\end{array}\right), 
	$$
	obtained from the incompressible Navier-Stokes system \eqref{eq:nu}. This first-order $O(\varepsilon)$ approximation captures variations in the random space of macroscopic quantities in the Vlasov-Fokker-Planck-Navier-Stokes equations with a certain level of accuracy. Since the initial condition of the high-fidelity and low-fidelity models are required to be consistent,
	$$
	\mathbf{u}^L_0=\left(\begin{array}{c}
		\rho_0 \\
		\rho_0 u_0
	\end{array}\right), \quad \rho_0 = \sum_{i=1}^N i n_{i,0}, \  n_{i,0} = \int_{\mathbb{R}^{3}} f_{i,0} \mathrm{d} v.
	$$
	We then consider a linearization of the hydrodynamic limit equations around the state 
	$\left(n_{i,0}, u_0\right)=(1,0)$,
	$$
	n_i^L=1+\delta \tilde{n_i}, \quad u^L=\delta \tilde{u},
	$$
	where $\delta$ is the same quantity as the $\delta$ in \eqref{uH_ust},   since we want the initial data of $\textbf{u}^L$ and $\textbf{u}^H$ to be consistent. 
	One can derive the following equations:
	\begin{eqnarray}\label{equLtilde}
		\left\{\begin{array}{l}
			\partial_t \left(\sum_{i=1}^N i \tilde{n_i} \right)+ \left(\sum_{i=1}^N i\right) \nabla_x \cdot \tilde{u}=0, \\
			\left(1+\kappa \sum_{i=1}^N i \right)\partial_t \tilde{u}+\nabla_x(\tilde{p}+\kappa\sum_{i=1}^N i \tilde{n_i})=\Delta_x \tilde{u},\\
			\nabla_{x} \cdot \tilde{u}=0,
		\end{array}\right.
	\end{eqnarray}
	which is essentially an incompressible acoustic equation \cite{Goudon2004b}.

	\subsection{Notations}
	
	We first define the space and norms that will be used. The Hilbert space of the random variable is given by
	$$
	H\left(\mathbb{R}^d ; \pi \mathrm{d} z\right)=\left\{f \mid I_z \rightarrow \mathbb{R}, \int_{I_z} f^2(z) \pi(z) \mathrm{d} z<\infty\right\}
	$$
	and equipped with the inner product
	$$
	\langle f, g\rangle_\pi=\int_{I_z} f g \pi(z) \mathrm{d} z.
	$$
	
	We introduce the standard multivariate notation. Denote the countable set of "finitely supported" sequences of non-negative integers by
	$$
	\mathcal{F}:=\left\{{\nu}=\left({\nu}_1, {\nu}_2, \cdots\right): {\nu}_j \in \mathbb{N} \text {, and } {\nu}_j \neq 0 \text { for only a finite number of } j\right\} \text {, }
	$$
	with $|{\nu}|:=\sum_{j \geq 1}\left|{\nu}_j\right|$. For $\nu \in \mathcal{F}$ supported in $\{1, \cdots, J\}$, the partial derivative in $z$ is defined by
	$$
	\partial_z^{\nu} =\frac{\partial^{|{\nu}|} }{\partial^{{\nu}_1} z_1 \cdots \partial^{{\nu}_J} z_J} .
	$$
	The $z$-derivative of order $\gamma$ of a function $f$ is denoted by
	$
	f^{\gamma}=\partial_{z}^{\gamma} f.
	$

	For functions $u=u(x), f=f(x, v)$, define the Sobolev norm (with $x$-derivatives)
	$$
	\|u\|_{s}^{2}=\sum_{|\alpha| \leq s}\left\|\partial^{\alpha} u\right\|_{L_{x}^{2}}^{2}, \quad\|f\|_{s}^{2}=\sum_{|\alpha| \leq s}\left\|\partial^{\alpha} f\right\|_{L_{x, v}^{2}}^{2}.
	$$
	In particular, denote by $\|u\|_{0}$  the $L_{x}^{2}$ norm of $u$. 
	
	For functions $u=u(x, z), f=f(x, v, z)$, define the sum of the Sobolev norms
	\begin{equation}\label{Sobolevnorm}
		|u|_{s, r}^{2}=\sum_{|\gamma| \leq r}\left\|u^{\gamma}(\cdot, z)\right\|_{s}^{2}, \quad|f|_{s, r}^{2}=\sum_{|\gamma| \leq r}\left\|f^{\gamma}(\cdot, \cdot, z)\right\|_{s}^{2},	
	\end{equation}
	where $|u|_{s, r}$ and $|f|_{s, r}$ are functions of $z$. Then define the expected value of the total Sobolev norm by
	$$
	\|u\|_{s, r}^{2}=\int|u|_{s, r}^{2} \pi(z) \mathrm{d} z, \quad\|f\|_{s, r}^{2}=\int|f|_{s, r}^{2} \pi(z) \mathrm{d} z,
	$$
	and the sup norm in $z$:
	$$
	\|u\|_{s,r,\infty}=\sup _{z \in I_z}|u|_{s,r} .
	$$
	
	For function $\bar{u}=\bar{u}(z)$, also define the sum of derivatives, the Sobolev norm  and the sup norm by
	$$
	|\bar{u}|_{r}^{2}=\sum_{|\gamma| \leq r}\left|\bar{u}^{\gamma}\right|^{2}, \quad\|\bar{u}\|_{r}^{2}=\int|\bar{u}|_{r}^{2} \pi(z) \mathrm{d} z, \quad\|\bar{u}\|_{r,\infty}=\sup _{z \in I_z} \|\bar{u}\|_{r}.
	$$
	In all these notations, the sub-index $r$ is omitted when $r=0$.

	All the norms or inner products with a single bound (like $|\cdot|,(\cdot, \cdot),[\cdot, \cdot])$, integral in $x, v$ and pointwise in $z$, are functions of $z$. All the norms or inner products with a double bound (like $\|\cdot\|,((\cdot, \cdot)),[[\cdot, \cdot]])$, integral with respect to all variables, are constant numbers.

	The $L^{2}$ inner product of functions defined on $x$-space of $x, v$-space is denoted by $\langle\cdot, \cdot\rangle$, i.e.,
	$$
	\langle f, g\rangle=\int f g \mathrm{d} x, \quad \text { or } \quad\langle f, g\rangle=\iint f g \mathrm{d} v \mathrm{d} x .
	$$
	In case the inputs also depend on $z$, $\langle f, g\rangle$ only integrates in $x$ or $(x, v)$, and the inner product is a function of $z$. For example,
	$$
	\langle f, g\rangle(z)=\int f(x, z) g(x, z) \mathrm{d} x.
	$$

	Next we introduce the inner products related to the hypocoercivity arguments. Define
	$$
	\mathcal{K}_i=\dfrac{\bar{\theta}}{i}\nabla_{v}+\frac{v}{2}, \quad \mathcal{P}_i=\dfrac{\bar{\theta}}{i} v \cdot \nabla_{x}, \quad \mathcal{S}_{j}=\left[\mathcal{K}_{ij}, \mathcal{P}_i\right]=\mathcal{K}_{ij} \mathcal{P}_i-\mathcal{P}_i \mathcal{K}_{ij}=\dfrac{\bar{\theta}^2}{i^2}\partial_{x_{j}}, \quad \mathcal{K}_i^{*}=-\dfrac{\bar{\theta}}{i}\nabla_{v}+\frac{v}{2},
	$$
	where $\mathcal{K}_i^{*}$ is the adjoint operator of $\mathcal{K}_i$, in the sense that $\langle\mathcal{K}_i f_i, g_i\rangle=\left\langle f_i, \mathcal{K}_i^{*} \cdot g_i\right\rangle$, where $f_i$ has one component and $g_i$ has three components.

	For functions $f_i=f_i(x, v), g_i=g_i(x, v)$, define
	$$
	\begin{aligned}
		&(f_i, g_i)=\frac{1}{i^{1/3}}\left(2\langle\mathcal{K}_i f_i, \mathcal{K}_i g\rangle+\varepsilon^2\langle\mathcal{K}_i f_i, \mathcal{S}_i g_i\rangle+\varepsilon^2\langle\mathcal{S}_i f_i, \mathcal{K}_i g_i\rangle+\varepsilon^3\langle\mathcal{S}_i f_i, \mathcal{S}_i g_i\rangle\right), \\
		&[f_i, g_i]=\langle\mathcal{K}_i f_i, \mathcal{K}_i g_i\rangle+\varepsilon^4\langle\mathcal{S}_i f_i, \mathcal{S}_i g_i\rangle+\varepsilon^2\left\langle\mathcal{K}_i^{2} f_i, \mathcal{K}_i^{2} g_i\right\rangle+\varepsilon^4\langle\mathcal{K}_i \mathcal{S}_i f_i, \mathcal{K}_i \mathcal{S}_i g_i\rangle,
	\end{aligned}
	$$
	where we denote $\langle\mathcal{K}_i \mathcal{S}_i f_i, \mathcal{K}_i \mathcal{S}_i g_i\rangle:=\sum_{j, l=1}^{3}\left\langle\mathcal{K}_{ij} \mathcal{S}_{il} f_i, \mathcal{K}_{ij} \mathcal{S}_{il} g_i\right\rangle .$
	
	For functions $f=f(x, v, z), g=g(x, v, z)$, define
	\begin{equation}\label{def:fg}
		\begin{aligned}
			(f, g)_{s, r}=\sum_{|\gamma| \leq r} \sum_{|\alpha| \leq s}\left(\partial^{\alpha} f^{\gamma}(\cdot, \cdot, z), \partial^{\alpha} g^{\gamma}(\cdot, \cdot, z)\right),\\
			[f, g]_{s, r}=\sum_{|\gamma| \leq r} \sum_{|\alpha| \leq s}\left[\partial^{\alpha} f^{\gamma}(\cdot, \cdot, z), \partial^{\alpha} g^{\gamma}(\cdot, \cdot, z)\right],
		\end{aligned}
	\end{equation}
	where $(f, g)_{s, r}$ or $[f, g]_{s, r}$ is a function of $z$.
	
	Then we introduce the inner product in the $(x, v, z)$ space:
	$$
	\langle\langle f, g\rangle\rangle=\int\langle f, g\rangle \pi(z) \mathrm{d} z,\quad
	(( f, g))=\int( f, g)\pi(z) \mathrm{d} z,\quad
	[[ f, g]]=\int[ f, g] \pi(z) \mathrm{d} z,
	$$
	$$
	(( f, g))_{s, r}=\int( f, g)_{s, r}\pi(z) \mathrm{d} z,\quad
	[[ f, g]]_{s, r}=\int[ f, g]_{s, r} \pi(z) \mathrm{d} z.
	$$
	
	\subsection{Main results}	
	
	We make the following assumptions on the random initial data:
	\begin{assumption}\label{assump}
		Assume that each component of the random variable $z:=\left(z_j\right)_{j \geq 1}$ has a compact support.  Let $\left(\psi_j\right)_{j \geq 1}$ be an affine representer of the random initial data $h_{\text {in}}:=(u_{in},\{f_{i,in}\}_{i=1}^N\})$, which by definition means that  \cite{Cohen2015}
		\begin{equation}\label{Assump1}
			h_{\mathrm{in}}(z)=\tilde{h}_0+\sum_{j \geq 1} z_j \psi_j,
		\end{equation}
		where $\tilde{h}_0=\tilde{h}_0(x, v)$ is independent of $z$, and the sequence $\left(\left\|\psi_j\right\|_{L^{\infty}(V)}\right)_{j \geq 1} \in \ell^p$ for $0<p<1$, with $V$ representing the physical space.
	\end{assumption}
	
	The reason we need to assume compact support of random variable $z$ and \eqref{Assump1} is that our analysis is based on  \cite{Cohen2015}, the result of which will be used when we estimate the projection error for greedy algorithm later in this section.
	
	We now state the main result of this section:
	
	\begin{theorem}\label{Thm1}
		Let Assumption \ref{assump} hold and assume $(u, \{f_i\}_{i=1}^N)$ solves \eqref{eq:uf} with initial data verifying \eqref{initialu0} and \eqref{initialu0_2}. 
		Fix a point $z$. Define the energy
		\begin{equation}\label{def:E}
			E(t ; z)=E_{s, r}(t ; z)=\frac{1}{\delta^2}|u|_{s,r}^{2}+\kappa\bar{\theta}\sum_{i=1}^{N}|f_i|_{s, r}^{2}+\frac{1}{\delta^2}|\bar{u}|_{r}^{2},
		\end{equation}
		with integers $s \geq 2$ and $r \geq 0$. 
		Assume there exists a sufficiently small positive number $\eta$  independent of $\delta$ and $\varepsilon$ such that 
		\begin{equation}\label{assumpI}
			\sup_{z\in I_z} E_{s+3,r}(0,z) \leq \eta \text{~~for } s \geq 0, r\geq 0, 
		\end{equation}
		and that $C_{s, r}^{h}=\kappa\bar{\theta}\sum_{i=1}^N\left.(f_i, f_i)_{s, r}\right|_{t=0}$ (defined by \eqref{def:fg})  is finite.
		Then by implementing $K$ high-fidelity simulation runs, the error estimate is given by
		\begin{eqnarray}\label{thm1result}
			\left\|\textbf{u}_j^H(z)-\textbf{u}_j^B(z)\right\|_{s,0} 
			\leq \frac{C_1}{(K / 2+1)^{q / 2}}+\frac{1}{\sqrt{\lambda_0}} \max \left\{\eta', C_{\xi}\right\}\left(1+e^{-\lambda t}\right) \sqrt{K} \delta,
		\end{eqnarray}
		where $C_1, \eta', C_{\xi}$ are  bounded constants independent of $\varepsilon$ and $N$, and depend on the initial data of the perturbation solutions $\eta$ and $\xi$ ($\eta$ and $\xi$ are sufficiently small positive numbers). 
	\end{theorem}

	\begin{remark}
		If additionally one assumes that $\delta=$ $o(1 / \sqrt{K})$, then
		$$
		\left\|\textbf{u}_j^H(z)-\textbf{u}_j^B(z)\right\|_{s,0} \leq \frac{C_1}{(K / 2+1)^{q / 2}}+C_2,
		$$
		where $C_2=o(1)$, and is independent of $K$ and $\varepsilon$.
	\end{remark}
	
	Theorem \ref{Thm1} indicates that the error between the bi- and high-fidelity solutions decays algebraically with respect to the number of high-fidelity runs $K$. The convergence rate $q / 2$ is independent of the dimension of the random space and the regularity of the initial data; it only relates to the $\ell^p$ summability of the affine representer $\left(\psi_j\right)_{j \geq 1}$ in Assumption \ref{assump}.
	
	\begin{remark}
		For the two-phase flow model in fine particle regime where a single species of particles is considered with a given and fixed mass density and size, the scaled PDE systems are given by:
		\begin{equation}\label{eq:two}
			\left\{\begin{aligned}
				&F_{t}+v \cdot \nabla_{x} F=\frac{1}{\varepsilon}\operatorname{div}_{v}\left((v-u) F+\nabla_{v} F\right), \\
				&u_{t}+u \cdot \nabla_{x} u+\nabla_{x} p-\Delta_{x} u=\frac{\kappa}{\varepsilon} \int_{\mathbb{R}^{3}}(v-u) F \mathrm{d} v, \\
				&\nabla_{x} \cdot u=0,
			\end{aligned}\right.
		\end{equation}
		with the initial condition that depends on $z$:
		$$\left.u\right|_{t=0}=u_0, \quad \nabla_x \cdot u_0=0,\left.\quad F\right|_{t=0}=F_{0},$$
		with $u_0 = u(0,x,z), F_0 = F(0,x,v,z)$.
		Let  $\textbf{u}^H$ be the high-fidelity solution, which is defined by the macroscopic moments of density and momentum obtained from the distribution $F$ solved by system \eqref{eq:two}:
		$$\mathbf{u}^H=\int_{\mathbb{R}^{3}}\left(\begin{array}{c}
			1 \\
			v 
		\end{array}\right) f(v) \mathrm{d} v:=\left(\begin{array}{c}
			n^H \\
			n^H u^H 
		\end{array}\right):=\left(\begin{array}{c}
			\textbf{u}_1^H \\
			\textbf{u}_2^H 
		\end{array}\right) .$$
		The incompressible Navier-Stokes system
		\begin{equation}\label{eq:NS}
			\left\{\begin{aligned}
				& \partial_{t} n +\nabla_{x} \cdot(n u)=0, \\
				&\partial_{t}\left(\left(1+ \kappa n\right)u\right)+\operatorname{Div}_{x}\left(\left(1+ \kappa n\right)u \otimes u\right)+ \nabla_{x} \left(p+ \kappa n \right)-\Delta_{x} u=0, \\
				&\nabla_{x} \cdot u=0,
			\end{aligned}\right. 
		\end{equation}
		is chosen as the low-fidelity model, which is a first-order ($O(\varepsilon)$) approximation to the two-phase flow system, and can model the variations in the random space of macroscopic quantities of the two-phase flow system \eqref{eq:two} up to a certain accuracy.
		
		In this framework, let $F$ be the perturbed solution. Consider the ansatz
		\begin{eqnarray}\label{eq:TwoAnsatz}
			F=M+\delta \sqrt{M} f,
		\end{eqnarray}
		where parameter $\delta$ is assumed sufficiently small and independent of $\varepsilon$, and $M$ is the Maxwellian. 
		Plugging \eqref{eq:TwoAnsatz} into \eqref{eq:two}, the following system for the perturbation $\left(u,f\right)$ is achieved:
		\begin{eqnarray}\label{eq:TwoPerturbed}
			\left\{\begin{array}{l}
				\left(f\right)_t+v \cdot \nabla_x f+\frac{1}{\varepsilon} u \cdot\left(\nabla_v-\frac{ v}{2 }\right) f-\frac{1}{ \varepsilon \delta} u \cdot v \sqrt{M}=\frac{1}{\varepsilon}\left(-\frac{|v|^2}{4}+\frac{3}{2}+ \Delta_v\right) f, \\
				u_t+u \cdot \nabla_x u+\nabla_x p-\Delta_x u+\frac{\kappa}{\varepsilon} u +\frac{\kappa \delta}{\varepsilon} u \int_{\mathbb{R}^3} \sqrt{M} f \mathrm{~d} v-\frac{\kappa \delta}{\varepsilon}  \int_{\mathbb{R}^3} v \sqrt{M} f \mathrm{~d} v=0, \\
				\nabla_x \cdot u=0,
			\end{array}\right.
		\end{eqnarray}
		with the initial data
		\begin{eqnarray}
			\left.u\right|_{t=0}=u_0,\left.\quad f\right|_{t=0}=f_{ 0},
		\end{eqnarray}
		\begin{eqnarray}\label{two:initialu0}
			\int_{\mathbb{T}^3} u_0 \mathrm{~d} x+\int_{\mathbb{T}^3} \int_{\mathbb{R}^3}  \delta v \sqrt{M} f_{ 0} \mathrm{~d} v \mathrm{~d} x=0, \quad \nabla_x \cdot u_0=0,
		\end{eqnarray}
		and
		\begin{eqnarray}\label{two:initialu0_2}
			\int_{\mathbb{T}^3} \int_{\mathbb{R}^3} \sqrt{M} f_{ 0} \mathrm{~d} v \mathrm{~d} x=0 .
		\end{eqnarray}
		Denote the steady state by $\mathbf{u}^{\text {st }}=(1,0)^T=\left(\textbf{u}_1^{s t}, \textbf{u}_2^{s t}\right)$, 
		then
		$$
		\mathbf{u}^H=\mathbf{u}^{s t}+\delta \int_{\mathbb{R}^{3}}\left(\begin{array}{c}
			1 \\
			v 
		\end{array}\right) \sqrt{M} f(v) \mathrm{d} v.
		$$
		We let the low-fidelity solution
		$$
		\mathbf{u}^L=\left(\begin{array}{c}
			n^L \\
			n^L u^L 
		\end{array}\right):=\left(\begin{array}{c}
			\textbf{u}_1^L \\
			\textbf{u}_2^L 
		\end{array}\right)
		$$
		obtained from the incompressible Navier-Stokes  system \eqref{eq:NS}. Since the initial condition of the high-fidelity and low-fidelity models are required to be consistent, 
		$$
		\textbf{u}^L_0=\left(\begin{array}{c}
			\int_{\mathbb{R}^{3}} f_0 \mathrm{d} v \\
			\int_{\mathbb{R}^{3}}v f_0  \mathrm{d} v
		\end{array}\right) .
		$$
		We then consider a linearization of the hydrodynamic limit equations around the state 
		$\left(n_0, u_0\right)=(1,0)$
		$$
		n^L=1+\delta \tilde{n}, \quad u^L=\delta \tilde{u}. \quad 
		$$
		One can derive the following equations:
		$$
		\left\{\begin{array}{l}
			\partial_t \tilde{n}+\nabla_x \cdot \tilde{u}=0, \\
			(1+\kappa)\partial_t \tilde{u}+\nabla_x(\tilde{p}+\kappa\tilde{n})-\Delta_{x} \tilde{u}=0.
		\end{array}\right.
		$$
		Similar to multi-phase kinetic-fluid systems, error estimates of the bi-fidelity method for the two-phase kinetic-fluid model can be obtained. We directly give the result in Theorem \ref{ThmTwo} and omit the proof.
	\end{remark}
	
	\begin{theorem}\label{ThmTwo}
		Let Assumption \ref{assump} hold and assume $(u, f)$ solves \eqref{eq:TwoPerturbed} with initial data verifying \eqref{two:initialu0} and \eqref{two:initialu0_2}. 
		Fix a point $z$. Define the energy
		\begin{equation}\label{def:E2}
			E(t ; z)=E_{s, r}(t ; z)=\frac{1}{\delta^2}|u|_{s,r}^{2}+\kappa\bar{\theta}|f|_{s, r}^{2}+\frac{1}{\delta^2}|\bar{u}|_{r}^{2},
		\end{equation}
		with integers $s \geq 2$ and $r \geq 0$. 
		Assume there exists a sufficiently small positive number $\eta$  independent of $\delta$ and $\varepsilon$ such that 
		\begin{equation}\label{assumpI2}
			\sup_{z\in I_z} E_{s+3,r}(0,z) \leq \eta \text{~~for } s \geq 0, r\geq 0. 
		\end{equation}
		Then by implementing $K$ high-fidelity simulation runs, the error estimate is given by
		\begin{eqnarray}\label{thm1result2}
			\left\|\textbf{u}_j^H(z)-\textbf{u}_j^B(z)\right\|_{s,0} 
			\leq \frac{C_1}{(K / 2+1)^{q / 2}}+\frac{1}{\sqrt{\lambda_0}} \max \left\{\eta', C_{\xi}\right\}\left(1+e^{-\lambda t}\right) \sqrt{K} \delta,
		\end{eqnarray}
		where  $C_1, \eta', C_{\xi}$ are  bounded constants independent of $\varepsilon$ and $N$, and depend on the initial data of the perturbation solutions $\eta$ and $\xi$ ($\eta$ and $\xi$ are sufficiently small positive numbers).
	\end{theorem}

	\section{Proof of the error estimates}
	\label{sec:proof}
	\subsection{Prerequisites}	
	
	We first present some lemmas that are crucial for the proof of Theorem \ref{Thm1}. Lemma \ref{thm:energyestimate} gives an energy estimate assuming near-equilibrium initial data. This lemma is proved in Appendix A.1 by an energy estimate on $\partial^{\alpha} f_i^\gamma$. 
	Furthermore, we utilize a standard hypocoercivity argument presented in Lemma \ref{Lem2.2} to strengthen Lemma \ref{thm:energyestimate} and obtain Lemma \ref{thm:hypocoercivity}. Lemma \ref{thm:hypocoercivity} implies that the long-time behavior of the solution is insensitive to random initial data as long as the perturbations $(u_0,\{f_{i, 0}\}_{i=1}^N)$ are small in suitable Sobolev spaces and have vanishing total mass and momentum. Detailed proofs of Lemma \ref{Lem2.2} and Lemma \ref{thm:hypocoercivity} are given in Appendix A.2 and A.3, respectively.
	
	\begin{lemma}\label{thm:energyestimate}
		Assume $(u, \{f_i\}_{i=1}^N)$ solves \eqref{eq:uf} with initial data verifying \eqref{initialu0}. Fix a point $z$, for the energy defined by
		\begin{equation}\label{def:E0}
			E(t ; z)=E_{s, r}(t ; z)=\frac{1}{\delta^2}|u|_{s,r}^{2}+\kappa\bar{\theta}\sum_{i=1}^{N}|f_i|_{s, r}^{2}+\frac{1}{\delta^2}|\bar{u}|_{r}^{2},
		\end{equation}
		with integers $s \geq 2$ and $r \geq 0$, there exists a constant $c_{1}=c_{1}(s, r)>0$, such that $E(0 ; z) \leq c_{1}$ implies that $E(t ; z)$ is non-increasing in $t$.
	\end{lemma}

	\begin{lemma}\label{Lem2.2}
		Assume $(u, \{f_i\}_{i=1}^N)$ solves \eqref{eq:uf} with initial data verifying \eqref{initialu0} and \eqref{initialu0_2}.  Then there exists a constant $c_1'(s, r) \leq c_{1}(s+3, r)$ such that, if one assumes that $s\geq 0, E_{s+3, r}(0) \leq c_{1}^{\prime}(s, r)$ is small enough, then there exists a constant $\lambda_{1}>0$ such that $($at each $z)$
		\begin{equation}\label{ineq:flambda1}
			\partial_{t}(f_i, f_i)_{s, r}+\frac{\lambda_{1}}{\bar{\theta}\epsilon^2} [f_i, f_i]_{s, r} \leq \frac{C(\lambda_{1})}{\bar{\theta}}\left(\frac{1}{\delta^2}|u|_{s, r}^{2}+\frac{1}{\delta^2}\left|\nabla_{x} u\right|_{s, r}^{2}+\dfrac{1}{\epsilon^2}|\mathcal{K}_i f_i|_{s, r}^{2}\right).
		\end{equation}
	\end{lemma}

	\begin{lemma}\label{thm:hypocoercivity}
		Assume $(u, \{f_i\}_{i=1}^N)$ solves \eqref{eq:uf} with initial data verifying \eqref{initialu0} and \eqref{initialu0_2}. There exists a constant $c_{1}^{\prime}(s, r)$ such that, if one assumes $s \geq 0, E_{s+3, r}(0) \leq c_{1}^{\prime}(s, r)$, and that $C_{s, r}^{h}=\kappa\bar{\theta}\sum_{i=1}^N\left.(f_i, f_i)_{s, r}\right|_{t=0}$ (defined by \eqref{def:fg}) 
		is finite, then there exists a constant $\lambda>0$ such that
		$$
		E_{s, r}(t) \leq C\left(E_{s, r}(0)+C_{s, r}^{h}\right) e^{-\lambda t},
		$$
		where $C=C(s, r)$.
	\end{lemma}
	
	Note that the constant $\lambda$ is independent of $\epsilon$ and depends on $\kappa, \bar{\theta}$ and $N$ (the number of particle sizes). All the constants $C$ in this paper are independent of $\epsilon$ and may depend on $\kappa, \bar{\theta}$ and $N$.

	\subsection{Proof of Theorem \ref{Thm1}}
	\label{sec:proof1}
	
	Taking $\|\cdot\|_{s,0}$ norm on both sides of the equality \eqref{eqsThreeParts} for each moment component $(j=1,2)$, one has
	\begin{eqnarray}\label{equjHB}
		\begin{aligned}
			&\left\|\textbf{u}_j^H(z)-\textbf{u}_j^B(z)\right\|_{s,0}=\left\|\textbf{u}_j^H(z)-\sum_{k=1}^K c_k(z) \textbf{u}_j^H\left(z_k\right)\right\|_{s,0}\\ \leq& \underbrace{\left\|\textbf{u}_j^H(z)-\textbf{u}_j^L(z)\right\|_{s,0} }_{\text {Term  } A_1 }
			+\underbrace{\left\|\textbf{u}_j^L(z)-\sum_{k=1}^K c_k(z) \textbf{u}_j^L\left(z_k\right)\right\|_{s,0}}_{\text {Term } A_2 } \\ &+\underbrace{\left\|\sum_{k=1}^K c_k(z)\left(\textbf{u}_j^L\left(z_k\right)-\textbf{u}_j^H\left(z_k\right)\right)\right\|_{s,0}}_{\text {Term  } A_3 }.
		\end{aligned}
	\end{eqnarray}
	The first term $A_1$ is estimated by $\left\|\textbf{u}_j^H-\textbf{u}_j^{s t}\right\|_{s,r}$  and $\left\|\textbf{u}_j^L-\textbf{u}_j^{s t}\right\|_{s,r}$. The second term $A_2$ is actually the projection error of the greedy algorithm when searching the most important points $\gamma_K$ from the low-fidelity solution manifold. To estimate the third term $A_3$,  a bound for the vector $\|\mathbf{c}\|$ with $\|\cdot\|$, the matrix induced $\ell_2$ norm, is obtained and then the Cauchy-Schwarz inequality is applied.

	Under the assumption with initial data verifying \eqref{initialu0} and \eqref{initialu0_2}, based on the hypocoercivity analysis Lemma \ref{thm:hypocoercivity} derived in \cite{JinLin2022}, the following estimate componentwise for $\textbf{u}^H$ is obtained by \eqref{uH_ust} and using the Cauchy-Schwarz inequality for each $j$:
	\begin{eqnarray}\label{step1uHust}
		\left\|\textbf{u}_{j}^H-\textbf{u}_{j}^{s t}\right\|_{s,r} \leq \delta \sum_i i \|f_i\|_{s,r} \leq \delta \eta' e^{-\lambda t}, \quad  j = 1, 2.
	\end{eqnarray}
	
	For the low-fidelity equations (the incompressible Navier-Stokes equations), it is clear that if $\left\|\tilde{u}_j^L(t=0)\right\|_{s,r}$ $\leq \xi$, then at all time $t>0$, 
	\begin{eqnarray}\label{Step2.1uLbound}
		\left\|\tilde{u}_j^L\right\|_{s,r} \leq C_{\xi},
	\end{eqnarray}
	where $\tilde{u}_j^L$ denotes the perturbed part of the low-fidelity solution $\textbf{u}_j^L$. Due to Eqs. \eqref{assumpI} and \eqref{step1uHust}, the bound $\xi$ is also sufficiently small. From Eq. \eqref{Step2.1uLbound}, one has
	\begin{eqnarray}\label{step2.2uLust}
		\left\|\textbf{u}_j^L-\textbf{u}_j^{s t}\right\|_{s,r}=\delta\left\|\tilde{u}_j^L\right\|_{s,r} \leq C_{\xi} \delta .
	\end{eqnarray}
	
	Since the initial data satisfy \eqref{assumpI},
	by Eqs. \eqref{step1uHust} and \eqref{step2.2uLust}, one has
	$$
	\begin{aligned}
		\left\|\textbf{u}_{j}^H-\textbf{u}_{j}^L\right\|_{s,r} &=\left\|\textbf{u}_j^H-\textbf{u}_j^{s t}-\left(\textbf{u}_j^L-\textbf{u}_j^{s t}\right)\right\|_{s,r} \\
		& \leq\left\|\textbf{u}_j^H-\textbf{u}_j^{s t}\right\|_{s,r}+\left\|\textbf{u}_j^L-\textbf{u}_j^{s t}\right\|_{s,r} \\
		& \leq \delta  \eta' e^{-\lambda t} +C_{\xi} \delta ;
	\end{aligned}
	$$
	thus
	\begin{eqnarray}\label{step3uHuL}
		\left\|\textbf{u}_j^H-\textbf{u}_j^L\right\|_{s,0} \leq \max \left\{\eta', C_{\xi}\right\} \delta\left(1+e^{-\lambda t}\right) \leq C^{\prime} \delta .
	\end{eqnarray}
	This indicates that the difference of density and momentum obtained between the high- and low-fidelity solutions are all of $O(\delta)$, that is, 
	the first term $A_1$ is $O(\delta)$. 
	
	We now evaluate the second term $A_2$ in Eq. \eqref{equjHB} by the Kolmogorov $K$-width of a functional manifold, which  characterizes the optimal distance for approximation from a general $K$-dimensional subspace. The subscript $j$ for $j=$ $1,2$ in $\textbf{u}^L(z)$ is omitted.  We denote $d_K\left(\textbf{u}^L\left(I_z\right)\right)$ the Kolmogorov $K$-width of the functional manifold $\textbf{u}^L\left(I_z\right)$, defined by
	$$
	d_K\left(\textbf{u}^L\left(I_z\right)\right)=\inf _{\operatorname{dim}\left(V_K\right)=K} \sup _{v \in \textbf{u}^L\left(I_z\right)} d^L\left(v, V_K\right) .
	$$
	Denote the space $\mathcal{H}=H_x^s$.
	From previous work such as \cite{Cohen2015,Devore2013}, we know that the projection error of the greedy algorithm satisfies
	\begin{eqnarray}\label{step6.1sup}
		\sup _{z \in I_z}\left\|\textbf{u}^L(z)-P_{U^L\left(\gamma_K\right)} \textbf{u}^L(z)\right\|_{s} \leq C \sqrt{d_{K / 2}\left(\textbf{u}^L\left(I_z\right)\right)}, 
	\end{eqnarray}
	where $P_{U^L\left(\gamma_K\right)}$ is the projection operator onto the $K$-dimensional subspace $U^L\left(\gamma_K\right)$, and $C$ is a constant.
	Let the $K$-dimensional subspace $V_K:=\operatorname{Span}\left\{c_\nu: \nu \in \Lambda_K\right\}$.
	Following the approach in \cite{LiuZhu2020}, Section 4.2, one obtains an upper bound for $d_K\left(\textbf{u}^L\left(I_z\right)\right)$ by using the Legendre polynomials on $[-1,1]$. Specifically, one has
	\begin{equation}\label{step6.2dK}
		\begin{aligned}
			d_K\left(\textbf{u}^L\left(I_z\right)\right)_{\mathcal{H}} & \leq \sup _{v \in \textbf{u}^L\left(I_z\right)} \min _{w \in V_K}\|v-w\|_{\mathcal{H}}=\sup _{z \in I_z} \min _{w \in V_K}\left\|\textbf{u}^L(z)-w\right\|_{\mathcal{H}} \\
			& \leq\left\|\textbf{u}^L-\sum_{\nu \in \Lambda_K} w_\nu P_\nu\right\|_{L^{\infty}\left(I_z, \mathcal{H}\right)} \leq \frac{C}{(K+1)^{q}}, \quad q=\frac{1}{p}-1 ,
		\end{aligned}
	\end{equation}
	where $p$ is a constant associated with the affine representer $\left(\psi_j\right)_{j \geq 1}$ in the random initial data specified by Eq. \eqref{Assump1}. In addition, $\sum_{\nu \in \Lambda_K} w_\nu P_\nu$ is the truncated Legendre expansion with $\left(P_k\right)_{k \geq 0}$, the sequence of renormalized Legendre polynomials on $[-1,1]$, and $\Lambda_K$ is the set of indices corresponding to the $K$ largest $\left\|w_\nu\right\|_{\mathcal{H}}$.
	Plugging $K/2$ into Eq. \eqref{step6.2dK} and  by Eq. \eqref{step6.1sup},
	\begin{eqnarray}\label{A2estimate}
		\operatorname{Term} A_2 \leq \frac{C_1}{(K/2+1)^{q/2}}.
	\end{eqnarray}
	
	Now, let us obtain an estimate for $||\mathbf{c}_j||\,(j=1,2)$ to bound the third term $A_3$. 
	Recall the definition of the projection onto $U_j^L\left(\gamma_K\right)$ for each $\textbf{u}_j^L(z)$ $(j=1,2)$ given in Eq. \eqref{defuL},
	$$
	\left(\mathcal{P}_{U_j^L\left(\gamma_K\right)}\left[\textbf{u}_j^L(z)\right]\right)^2=\sum_{k, s=1}^K c_k(z) c_s(z) \textbf{u}_j^L\left(z_k\right) \textbf{u}_j^L\left(z_s\right),
	$$
	thus
	$$
	\int_{\Omega}\left(\mathcal{P}_{U_j^L\left(\gamma_K\right)}\left[\textbf{u}_j^L(z)\right]\right)^2 d x=\sum_{k, s=1}^K c_k(z) c_s(z) \int_{\Omega} \textbf{u}_j^L\left(z_k\right) \textbf{u}_j^L\left(z_s\right) d x:=\mathbf{c}_j^T \mathbf{G}_j^L \mathbf{c}_j \geq \lambda_{j,0}\|\mathbf{c}_j\|^2 ;
	$$
	where $\mathrm{G}_j^L$ is the Gramian matrix of $U_j^L\left(\gamma_K\right)$ defined in Eq. \eqref{Gramianmatrix} (where $\langle\cdot, \cdot\rangle$ applies to $L_x^2$ here) and $\lambda_{j,0}>0$ is its minimum eigenvalue. The last inequality is due to the Gramian matrices being positive-semidefinite and we assume it is positive-definite (otherwise $\mathbf{c}_j$ could not be obtained by solving the system \eqref{GLsystem}). 
	
	Since for all $z$,
	$$
	\int_{\Omega}\left(\mathcal{P}_{U_j^L\left(\gamma_K\right)}\left[\textbf{u}_j^L(z)\right]\right)^2 d x \leq \int_{\Omega}\left[\textbf{u}_j^L(z)\right]^2 d x, \ j=1,2,
	$$
	then
	\begin{eqnarray}\label{step4cestimate}
		\|\mathbf{c}_j\| \leq \frac{1}{\sqrt{\lambda_{0,j}}}\left(\int_{\Omega}\left(\mathcal{P}_{U_j^L\left(\gamma_K\right)}\left[\textbf{u}_j^L(z)\right]\right)^2 d x\right)^{1 / 2} \leq \frac{1}{\sqrt{\lambda_{j,0}}}\left\|\textbf{u}_j^L(z)\right\|_{0},
	\end{eqnarray}
	for all $z$. Thus
	$$
	\|\| \mathbf{c}_j\|\|_{\infty} \leq \frac{1}{\sqrt{\lambda_{j,0}}}\left\|\textbf{u}_j^L(z)\right\|_{0,\infty} .
	$$
	Since $\sqrt{x}$ is a monotone function,
	$$
	\|\| \mathbf{c}_j\|\|_{\infty}=\left(\sum_{k=1}^K\left\|c_k(z)\right\|_{\infty}^2\right)^{1 / 2} .
	$$
	By the regularity of $\textbf{u}_j^L$ in Eq. \eqref{step2.2uLust} and the assumption that the volume of $\Omega$ is bounded, Eq. \eqref{step4cestimate} implies that
	\begin{eqnarray}\label{step4.2cestimate}
		\left(\sum_{k=1}^K\left\|c_k(z)\right\|_{\infty}^2\right)^{1 / 2} \lesssim \frac{1}{\sqrt{\lambda_0}}\left\|\textbf{u}_j^L(z)\right\|_{0,\infty} \leq \frac{1}{\sqrt{\lambda_0}}\left\|\textbf{u}_j^L(z)\right\|_{s,\infty} \lesssim \frac{1}{\sqrt{\lambda_0}}\left(\textbf{u}_j^{s t}+C_{\xi} \delta\right) .
	\end{eqnarray}
	By the Cauchy-Schwarz inequality, using Eqs. \eqref{step3uHuL} and \eqref{step4.2cestimate}, one has
	\begin{equation}\label{A3estimate}
		\begin{aligned}
			\operatorname{Term} A_3 & \leq\left(\sum_{k=1}^K\left\|c_k(z)\right\|^2\right)^{1 / 2}\left(\sum_{k=1}^K\left\|\textbf{u}_j^H\left(z_k\right)-\textbf{u}_j^L\left(z_k\right)\right\|_{s,0}^2\right)^{1 / 2} \\
			& \lesssim\left(\sum_{k=1}^K\left\|c_k(z)\right\|_{\infty}^2\right)^{1 / 2}\left(\sum_{k=1}^K\left\|\textbf{u}_j^H\left(z_k\right)-\textbf{u}_j^L\left(z_k\right)\right\|_{s}^2\right)^{1 / 2} \\
			& \leq \sqrt{N}\left(\sum_{k=1}^K\left\|c_k(z)\right\|_{\infty}^2\right)^{1 / 2}\left(\max _n\left\|\textbf{u}_j^H\left(z_k\right)-\textbf{u}_j^L\left(z_k\right)\right\|_{s}\right) \\
			& \leq \frac{C^{\prime}}{\sqrt{\lambda_0}} \sqrt{K} \delta\left(\textbf{u}_j^{s t}+C_{\xi} \delta\right) \lesssim C \sqrt{K} \delta,
		\end{aligned}
	\end{equation}	
	where $C^{\prime}$ and $C$ are generic constants. The boundness of the random variable $z$ is used in the first inequality.
	
	By Eqs. \eqref{equjHB}, \eqref{step3uHuL}, \eqref{A2estimate} and \eqref{A3estimate}, one obtains Eq. \eqref{thm1result} and prove Theorem 1.\qed

	\section{Other choices of the low-fidelity models}
	\label{sec:other}
	
	In this section, we explore an alternative way of selecting low-fidelity models. One can let the low-fidelity model be the same kinetic-fluid equation as the high-fidelity model, while solving it on a coarser temporal and physical mesh, in particular, using a larger $\Delta t, \Delta x, \Delta v$ than those used for the high-fidelity model under the constraint of the CFL condition. Due to the difficulty in conducting error estimates of a discretized scheme for the full nonlinear kinetic-fluid equation, we will show some numerical experiments to demonstrate our concept.

	\subsection{A bi-fidelity algorithm}
	
	Let $M=1000$ be the number of affordable low-fidelity simulation runs. $K=10$ denotes the number of high-fidelity simulation runs that can be afforded. Let $\gamma_K=\left\{z_1, \cdots, z_K\right\}$ be a set of sample points in $I_z$. Denote the low-fidelity snapshot matrix $\textbf{u}^L(\gamma_K)=\left[\textbf{u}^L\left(z_1\right), \cdots, \textbf{u}^L\left(z_K\right)\right]$ and the corresponding low-fidelity approximation space
	$$
	U^L(\gamma_K)=\operatorname{span}\left\{\textbf{u}^L(\gamma_K)\right\}=\operatorname{span}\left\{\textbf{u}^L\left(z_1\right), \cdots, \textbf{u}^L\left(z_K\right)\right\}.
	$$
	Similarly, the high-fidelity snapshot matrix and the corresponding  high-fidelity approximation as follows:
	$$
	\textbf{u}^H(\gamma_K)=\left[\textbf{u}^H\left(z_1\right), \cdots, \textbf{u}^H\left(z_K\right)\right], \quad U^H(\gamma_K)=\operatorname{span}\left\{\textbf{u}^H(\gamma_K)\right\}.
	$$
	The bi-fidelity algorithm for approximating the high-fidelity solution comprises two stages: offline and online. In the offline stage, we employ the cheap low-fidelity model to explore the parameter space and identify the most important parameter points. In the online stage, we utilize the approximation rule calculated from the low-fidelity model for any given $z$ and apply it to construct the bi-fidelity approximation. The detailed algorithm is summarized in Algorithm 1.
	Most of the steps in this algorithm are straightforward. Providing details for Step 3 (point selection) and Step 6 (bi-fidelity approximation) would be beneficial for understanding Algorithm 1.
	
	\begin{table}
		\begin{tabular}{p{14.5cm}} 
			\hline
			\textbf{Algorithm 1: }Bi-fidelity approximation \\ \hline
			\textbf{Offline:}\\
			1 Select a sample set $\Gamma=\left\{z_1, z_2, \ldots, z_M\right\} \subset I_z$.\\
			2 Run the low-fidelity model $u_l\left(z_j\right)$ for each $z_j \in \Gamma$.\\
			3 Select $K$ "important" points from $\Gamma$ and denote it by $\gamma_K=\left\{z_{i_1}, \cdots, z_{i_K}\right\} \subset \Gamma$. Construct the low-fidelity approximation space $U^L\left(\gamma_K\right)$.\\
			4 Run high-fidelity simulations at each sample point of the selected sample set $\gamma_K$. Construct the high-fidelity approximation space $U^H\left(\gamma_K\right)$.\\
			\textbf{Online:}\\
			5 For any given $z$, get the low-fidelity solution $\mathbf{u}^L(z)$ and compute the low-fidelity coefficients by projection: $\mathbf{u}^L(z) \approx \mathcal{P}_{U^L\left(\gamma_K\right)}\left[\mathbf{u}^L(z)\right]=\sum_{k=1}^K c_k(z) \mathbf{u}^L\left(z_k\right)$\\
			6 Construct the bifidelity approximation by applying the sample approximation rule in the low-fidelity model:
			$\mathbf{u}^B(z)=\sum_{k=1}^K c_k(z) \mathbf{u}^H\left(z_k\right).$ \\ \hline
		\end{tabular}
	\end{table}

	\textbf{Point selection.} To select the subset $\gamma_K$, we adopt the greedy algorithm proposed in \cite{Narayan2014, Zhu2014}. Start with an empty trivial subspace $\gamma_0=\varnothing$ and assume that the first $k-1$ important points $\gamma_{k-1}=\left\{z_{i_1}, \cdots, z_{i_{k-1}}\right\} \subset \Gamma$ have been selected. The next point $z_{i_k} \in \Gamma$ is chosen as the point that maximizes the distance between its corresponding low-fidelity solution and the approximation space $U^L\left(\gamma_{k-1}\right)$, 
	$$
	z_{i_k}=\arg \max _{z \in \Gamma} \operatorname{dist}\left(\textbf{u}^L(z), U^L\left(\gamma_{k-1}\right)\right), \quad \gamma_k=\gamma_{k-1} \cup z_{i_k},
	$$
	where $\operatorname{dist}(v, W)$ is the distance function between $v \in \mathbf{u}^L(\Gamma)$ and subspace $W \subset U^L(\Gamma)$. The greedy procedure is essential for searching a linearly independent basis set in the parameterized low-fidelity solution space. Note that the whole algorithm can be efficiently implemented using standard linear algebra operations. In this scenario, we employ pivoted Cholesky decomposition because one needs not actually form the full $u^L(\gamma_K)$ Gramian. Algorithm 2 gives details for this point selection procedure. See \cite{Narayan2014, Zhu2014} for more technical details.

	\begin{table}
		\begin{tabular}{p{14.5cm}}
			\hline
			\textbf{Algorithm 2: }Pivoting Cholesky decomposition for selection of interpolation nodes \\ \hline
			\textbf{Input:} the snapshots $\mathbf{V}=\left[\mathbf{u}^L\left(z_1\right), \mathbf{u}^L\left(z_2\right), \ldots, \mathbf{u}^L\left(z_M\right)\right]$ from the low-fidelity model, the number of interpolation nodes $K$ \\	
			Initialize the ensemble for each parameter $z_m$:
			$\mathbf{w}(m)=\left(\mathbf{u}_m^L\right)^T \mathbf{u}_m^L, m=1, \ldots, M$\\
			Initialize the nodal selection vector and Cholesky factor: $P=[~], ~\mathbf{L}=\operatorname{zeros}(M, K)$\\
			for $k=1, \ldots, K$ do\\
			\quad Find the next interpolation point (the next pivot): $p = \arg \max _{m \in\{n, \ldots, M\}} w(m)$\\
			\quad Update $\mathbf{P}$ and exchange column $k$ and $p$ in $\mathbf{V}$: $P=[P~ p], ~ \mathbf{V}\left(:,\left[	k~ p\right]\right)=\mathbf{V}\left(:,\left[	p~  k\right]\right)$\\
			\quad Update $\mathbf{L}$:\\
			\quad $\mathbf{r}(t)=\mathbf{V}(:, t)^T \mathbf{V}(:, k)-\sum_{j=1}^{K-1} \mathbf{L}(t, j) \mathbf{L}(k, j), \text { for } t=k+1, \ldots, M$\\
			\quad $\mathbf{L}(k, k)=\sqrt{\mathrm{w}(k)}$\\
			\quad $\mathbf{L}(t, k)=\mathbf{r}(t) / \mathbf{L}(k, k) \text { for } t=k+1, \ldots, M$\\
			\quad $ \mathbf{w}(t)=\mathbf{w}(t)-\mathbf{L}^2(t, k) \text { for } t=k+1, \ldots, M$\\
			end for\\
			Truncate the Cholesky factor: $\mathbf{L}=\mathbf{L}(1: K,:)$\\
			Form the Gramian matrix: $\mathbf{G}^L=\mathbf{L L}^T$\\
			\textbf{Output:} the ordered nodal selection $P$ and the Gramian matrix $\mathbf{G}^L$\\
			\hline
		\end{tabular}
	\end{table}

	\textbf{The Bi-fidelity approximation. }In the offline stage, we have constructed the low- and high-fidelity approximation space, $U^L\left(\gamma_K\right)$ (step 3) and $U^H\left(\gamma_K\right)$ (step 4), respectively. In the online stage, for any given sample point $z \in I_z$, we shall project the corresponding low-fidelity solution $\mathbf{u}^L(z)$ onto the low-fidelity approximation space $U^L\left(\gamma_K\right)$ as follows:
	$$
	\mathbf{u}^L(z) \approx \mathcal{P}_{U^L\left(\gamma_K\right)}\left[\mathbf{u}^L(z)\right]=\sum_{k=1}^K c_k(z) \mathbf{u}^L\left(z_{i_k}\right),
	$$
	where $\mathcal{P}_V$ is the projection operator onto a Hilbert space $V$ and the corresponding projection coefficients $\left\{c_k\right\}$ are computed by the following projection:
	\begin{eqnarray}\label{eq:linear}
		\mathbf{G}^L \mathbf{c}=\mathbf{f}, \quad \mathbf{f}=\left(f_k\right)_{1 \leq k \leq K}, \quad f_k=\left\langle \mathbf{u}^L(z), \mathbf{u}^L\left(z_{i_k}\right)\right\rangle^L,
	\end{eqnarray}
	where $G^L$ is the Gramian matrix of $u^L\left(\gamma_K\right)$, defined by
	$$
	\left(\mathbf{G}^L\right)_{i j}=\left\langle \mathbf{u}^L\left(z_{i_k}\right), \mathbf{u}^L\left(z_{j_k}\right)\right\rangle^L, \quad 1 \leq k, i, j \leq K
	$$
	with $\langle\cdot, \cdot\rangle^L$ the inner product associated with the low-fidelity approximation space $U^L\left(\gamma_K\right)$.
	These low-fidelity coefficients $\left\{c_k\right\}$ serve as surrogates of the corresponding high-fidelity coefficients of $\mathbf{u}^H(z)$. Hence, the desired bi-fidelity approximation of $\mathbf{u}^H(z)$ can be constructed as follows:
	$$
	\mathbf{u}^B(z)=\sum_{k=1}^K c_k(z) \mathbf{u}^H\left(z_{i_k}\right) .
	$$
	
	Note that since the number of low-fidelity basis is typically small ($\mathcal{O}(10)$ in our numerical tests), the computation cost of solving the linear system \eqref{eq:linear} to compute the low-fidelity projection coefficients is negligible. The primary cost of the online stage arises from running a single low-fidelity simulation. If the low-fidelity solver is much cheaper than the high-fidelity solver, this approach can lead to a considerable speedup during the online stage.
	
	\subsection{Numerical tests}
	
	To examine the performance of the method, we shall compute numerical errors in the following way: we choose a fixed set of points $\left\{\hat{z}_i\right\}_{i=1}^M \subset I_z$ that is independent of the point sets $\Gamma$, and evaluate the following error between the bi- and high- fidelity solutions at a final time $t$:
	\begin{eqnarray}\label{eq:err}
		\left\|\textbf{u}^H(t)-\textbf{u}^B(t)\right\|_{L^2\left(D \times I_z\right)} \approx \frac{1}{M} \sum_{i=1}^M\left\|\textbf{u}^H\left(\hat{z}_i, t\right)-\textbf{u}^B\left(\hat{z}_i, t\right)\right\|_{L^2(D)},
	\end{eqnarray}
	where $\|\cdot\|_{L^2(D)}$ is the $L^2$ norm in the physical domain $D=\Omega \times \mathbb{R}^2$. The error can be considered as an approximation to the average $L^2$ error in the whole space of $D \times I_z$.
	
	In our bi-fidelity method, the numerical solutions $\textbf{u}^H$ solved by the finer discretized system \eqref{ModelEquation} are considered as the high-fidelity solutions, with $N_x^H$ grid points in the spatial domain. The numerical solution $\textbf{u}^L$ to Eq. \eqref{ModelEquation}, solved by using coarser physical mesh, are the low-fidelity solutions, with $N_x^L$ ($N_x^L<N_x^H$) grid points and the same velocity discretization as the high-fidelity model. We assume the random initial data in the form of Karhunen-Lo\`eve (KL) expansion.
	Both the high- and low- fidelity models in our simulation are numerically solved by Asymptotic Preserving schemes studied in \cite{JinLin2022AP}.
	
	In all the examples, the spatial domain is chosen to be $[0,1]$ with $N_x^H=128$ grid points for high-fidelity models and $N_x^L=64$ or $32$ grid points for low-fidelity models. The velocity domain is chosen as $\left[-L_v, L_v\right]^2$ with $L_v=8$ and $N_v=32$ grid points in each dimension. Set the final time $t=0.1$. Without loss of generality, the random variable $\mathbf{z}$ is assumed to follow the Gaussian distribution.
	The training set $\Gamma$ is chosen to be $M=1000$ random samples of \textbf{$\mathbf{z}$}. 
	
	Take the volcano like initial data with random inputs using the KL expansion as an example: 
	\begin{eqnarray}\label{ex1}
		\begin{aligned}
			& n_1(0, \mathbf{x},\mathbf{z})\\
			&=\left(0.4+100((x-0.5)^2+(y-0.5)^2)\right)\exp \left(-40(x-0.5)^2-40(y-0.5)^2\right)\left(1+\sigma \sum_{i=1}^N \sqrt{\lambda_i} g_i(\mathbf{x}) \mathbf{z}_i^1\right), \\
			& n_2(0, \mathbf{x},\mathbf{z})\\
			&=\left(0.5+100((x-0.5)^2+(y-0.5)^2)\right)\exp \left(-40(x-0.5)^2-40(y-0.5)^2\right)\left(1+\sigma \sum_{i=1}^N \sqrt{\lambda_i} g_i(\mathbf{x}) \mathbf{z}_i^2\right), \\
			& \mathbf{u}_{p, 1}(0, \mathbf{x})=\mathbf{u}_{p, 2}(0, \mathbf{x})=\left(\begin{array}{c}
				\sin ^2(\pi x) \sin (2 \pi y) \\
				-\sin ^2(\pi y) \sin (2 \pi x)
			\end{array}\right), \\
			& \mathbf{u}(0, \mathbf{x})=\mathbf{u}_{p, 1}(0, \mathbf{x}) ,
		\end{aligned}
	\end{eqnarray}
	where $\sigma=0.1$, $\{\mathbf{z}_i^1\}$ and $\{\mathbf{z}_i^2\}$ are independent random variables with zero mean and unit covariance, and $\left\{\lambda_i, g_i(\mathbf{x})\right\}$ are the eigenvalues and eigenfunctions for
	$$
	\int C(\mathbf{x}, \mathbf{z}) g_i(\mathbf{z}) \mathrm{d} \mathbf{z}=\lambda_i g_i(\mathbf{x}), \quad i=1,2, \ldots
	$$
	where $C(\mathbf{x}, \mathbf{z})=\exp (-|\mathbf{x}-\mathbf{z}|^2 / \ell^2)$ is the covariance function for $\kappa(x, y)$ taking the Gaussian form with a correlation length $\ell$. This is the widely adopted Karhunen-Lo\`eve expansion for random field \cite{Sullivan2015}. 
	Here we employ a short correlation length $\ell=0.08$, compared to the length of the domain $L=1$. 
	The decay of eigenvalues are shown in Fig. \ref{Fig:lambda}. Based on the decay, we employ the first 14 eigenvalues and set $N=9$ so that more than $95 \%$ of the total spectrum is retained. 
	For a rigorous analysis on the properties of numerical Karhunen-Lo\`eve expansion, see \cite{Schwab2006}.	
	
	\begin{figure}[htbp]
		\centering
		\includegraphics[width=6cm]{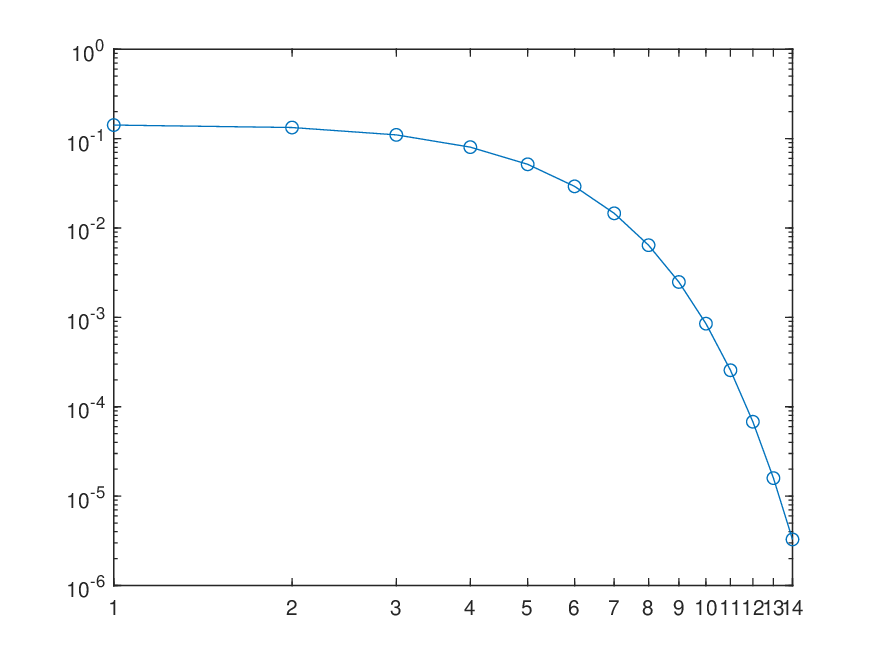}\\
		\caption{Eigenvalues of Karhunen-Lo\`eve expansion.}\label{Fig:lambda}
	\end{figure}
	
	For the initial data \eqref{ex1}, $\mathbf{u}_{p,i} \neq \mathbf{u}\,\, (i=1,2)$ and thus the equilibrium is not assumed. Figs. \ref{E1N32N128fig1}-\ref{E1N32N128fig3} give the solution of the system \eqref{ex1} with $\varepsilon=1$ at $t=0.1$, with $N_x^H=128$ grid points for high-fidelity models and $N_x^L=32$ grid points for low-fidelity models. 
	Fig. \ref{E1N32N128fig1} shows the mean and standard deviation of the bi-fidelity solutions of particle density $n_1$ and particle velocity $u_{p,11}, u_{p,12}$ of the first particle $(i=1)$ by using $r=10$ high-fidelity runs. Here $u_{p,11}, u_{p,11}$ stand for the first and second component of the two-dimensional velocity $\mathbf{u}_{p,1}$, respectively.
	Fig. \ref{E1N32N128fig2} shows the particle density $n_2$ and particle velocity $u_{p,21}, u_{p,22}$ of the second particle $(i=2)$ with $u_{p,21}, u_{p,22}$ standing for the first and second components of $\mathbf{u}_{p,2}$. Fig. \ref{E1N32N128fig3} gives the fluid velocity $\mathbf{u}$ with $u_{1}, u_{2}$ standing for the two components of the two-dimensional bulk velocity $\mathbf{u}$, respectively. Figs. \ref{E1N32N128fig1}-\ref{E1N32N128fig3} show clearly that the mean and standard deviation of the bi-fidelity approximation of $n_1, \mathbf{u}_{p,1}, n_2, \mathbf{u}_{p,2}$ and  $\mathbf{u}$ agree well with the high-fidelity solutions by using only 10 high-fidelity runs. The result suggests that even though the system solved by coarser grids may be inaccurate in the spatial domain, it still can capture the behaviors and characteristics of the multi-phase kinetic-fluid system in the random space.
	
	To compare the bi-fidelity solutions and the high-fidelity solutions, we examine the error of bi-fidelity approximation defined by \eqref{eq:err}, with respect to the number of high-fidelity runs (evaluated over an independent set of $M=1000$ Monte Carlo samples).
	Fast convergence of the mean $L^2$ errors with respect to the number of high-fidelity runs is also observed from Fig. \ref{E1N32N128fig4}.

	\begin{figure}[hp]
		\centering
		\includegraphics[width=14cm]{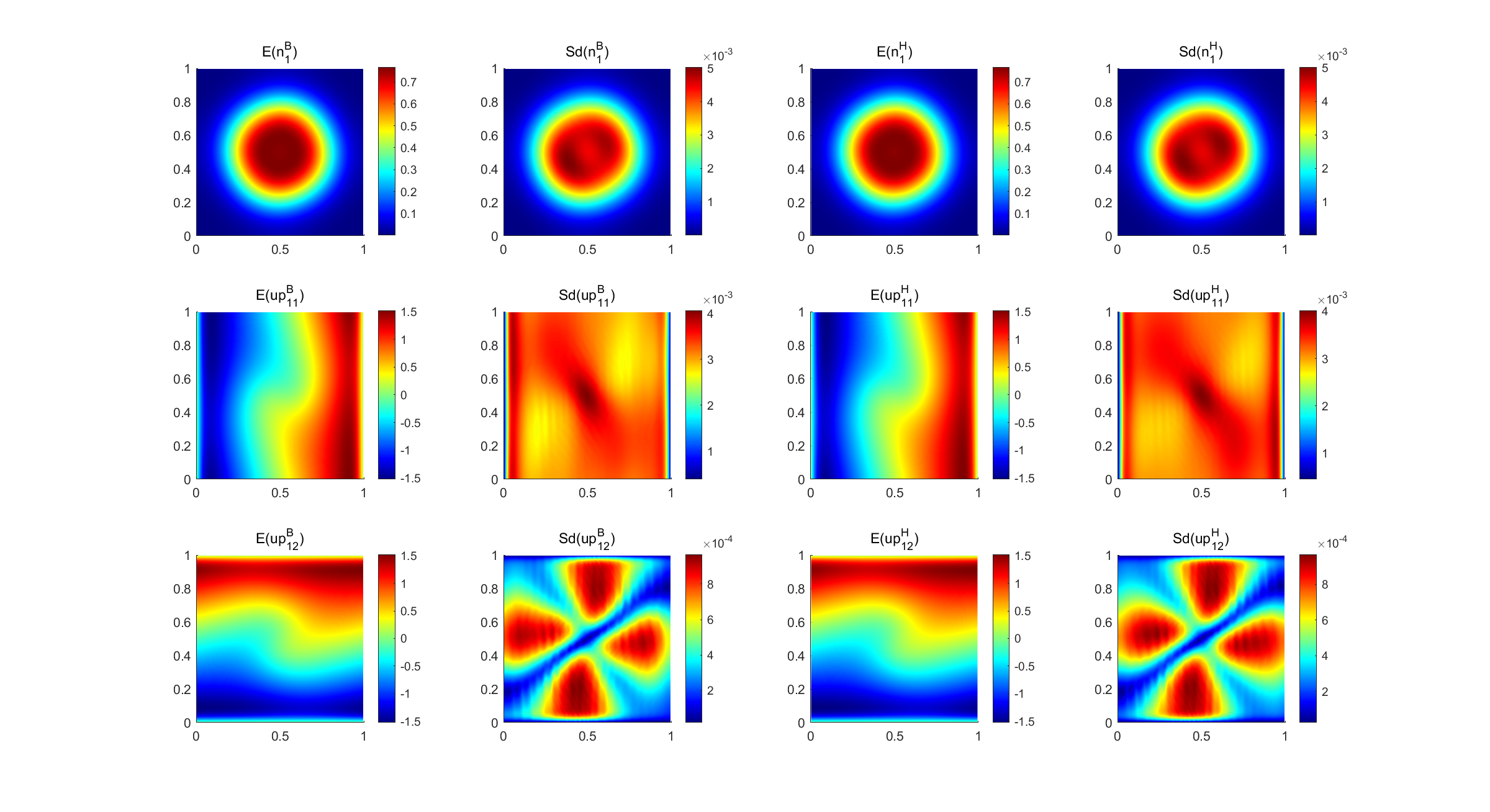}\\
		\caption{Mean and standard deviation of particle density $n_1$ and streamlines of particle velocity $u_{p,11}$, $u_{p,12}$ for bi-fidelity (the left two columns) and high-fidelity (the right two columns) solutions  by using $r=10$ at $t=0.1$ for $\varepsilon=1$ ($N_x^H=128$, $N_x^L=32$).}\label{E1N32N128fig1}
	\end{figure}
	
	\begin{figure}[hp]
		\centering
		\includegraphics[width=14cm]{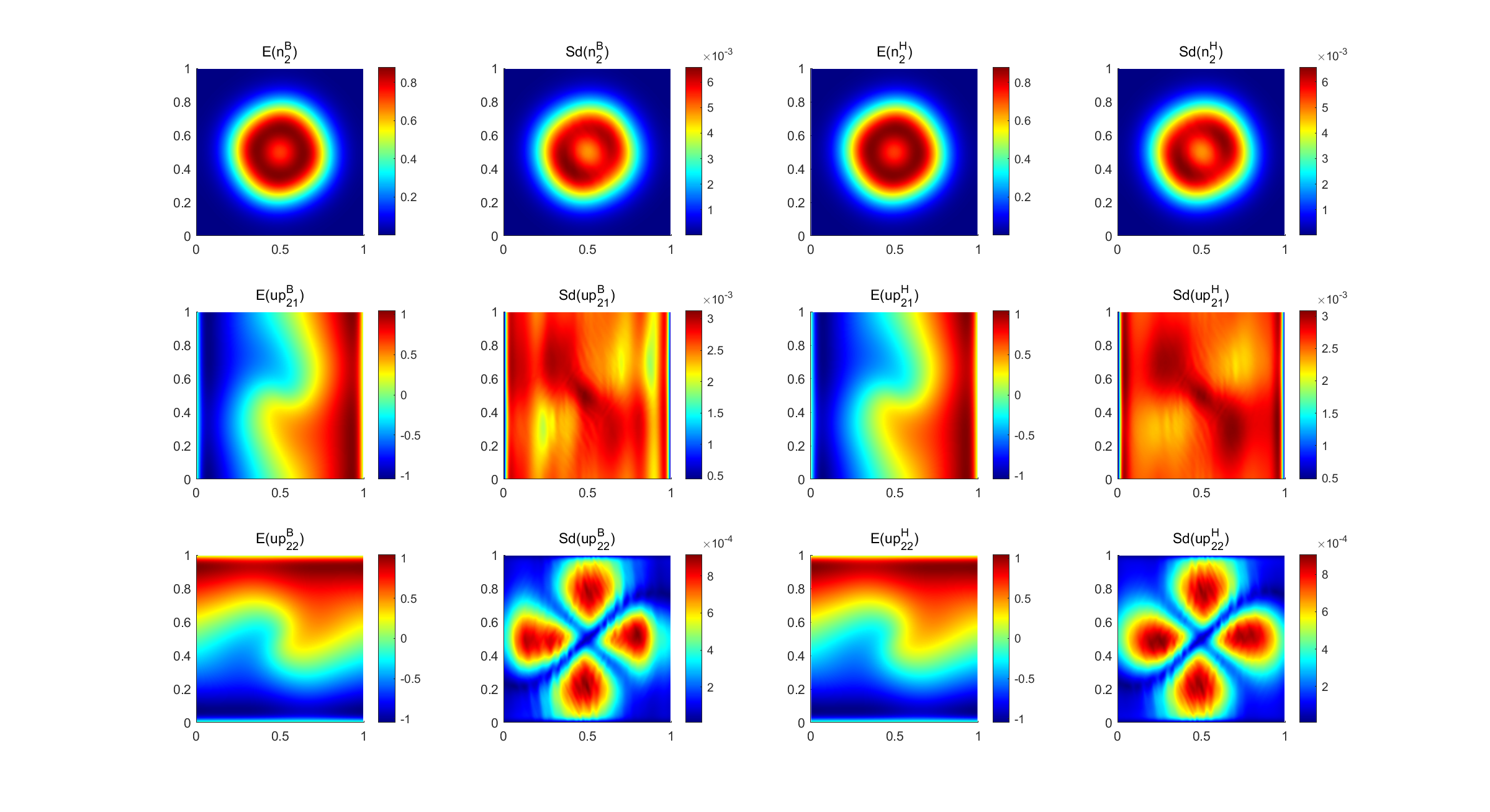}\\
		\caption{Mean and standard deviation of particle density $n_2$ and streamlines of particle velocity $u_{p,21}$, $u_{p,22}$ for bi-fidelity (the left two columns) and high-fidelity (the right two columns) solutions by using $r=10$ at $t=0.1$ for $\varepsilon=1$ ($N_x^H=128$, $N_x^L=32$).}\label{E1N32N128fig2}
	\end{figure}
	
	\begin{figure}[hp]
		\centering
		\includegraphics[width=14cm]{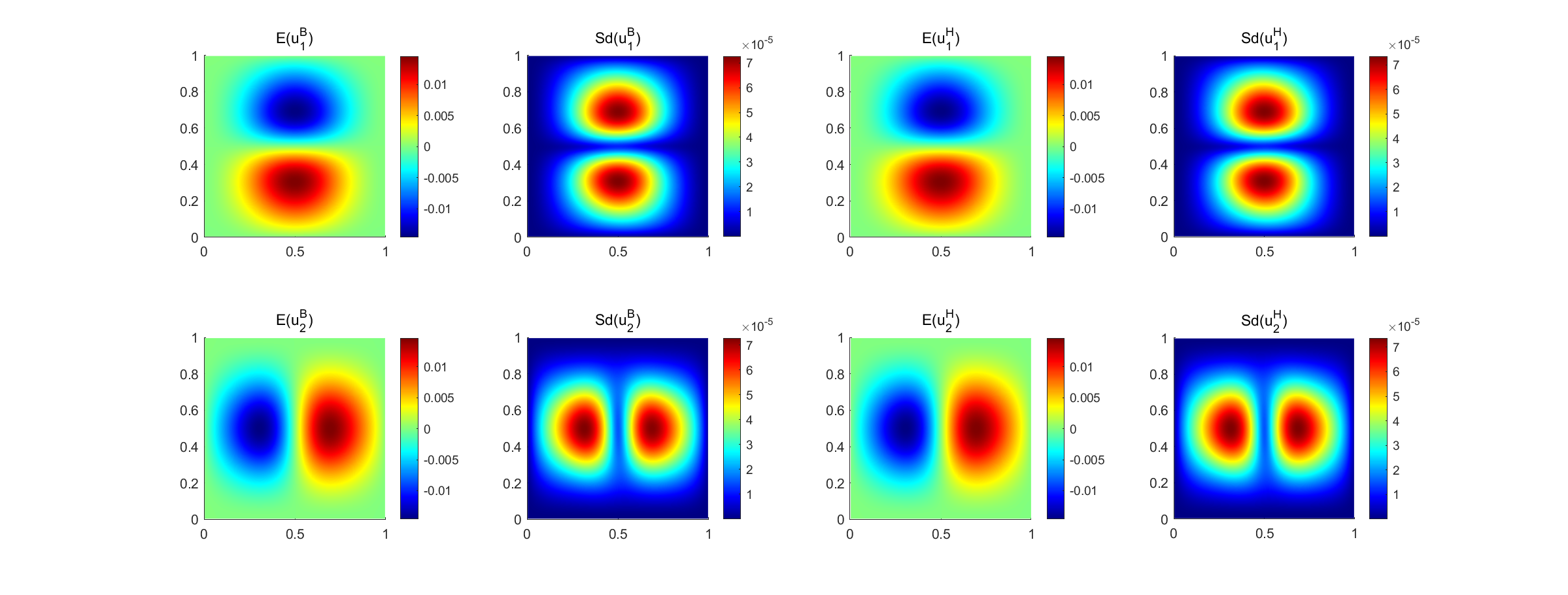}\\
		\caption{Mean and standard deviation of fluid velocity $u_1$, $u_2$ for bi-fidelity (the left two columns) and high-fidelity (the right two columns) solutions by using $r=10$ at $t=0.1$ for $\varepsilon=1$ ($N_x^H=128$, $N_x^L=32$).}\label{E1N32N128fig3}
	\end{figure}
	
	\begin{figure}[hp]
		\centering
		\includegraphics[width=14cm]{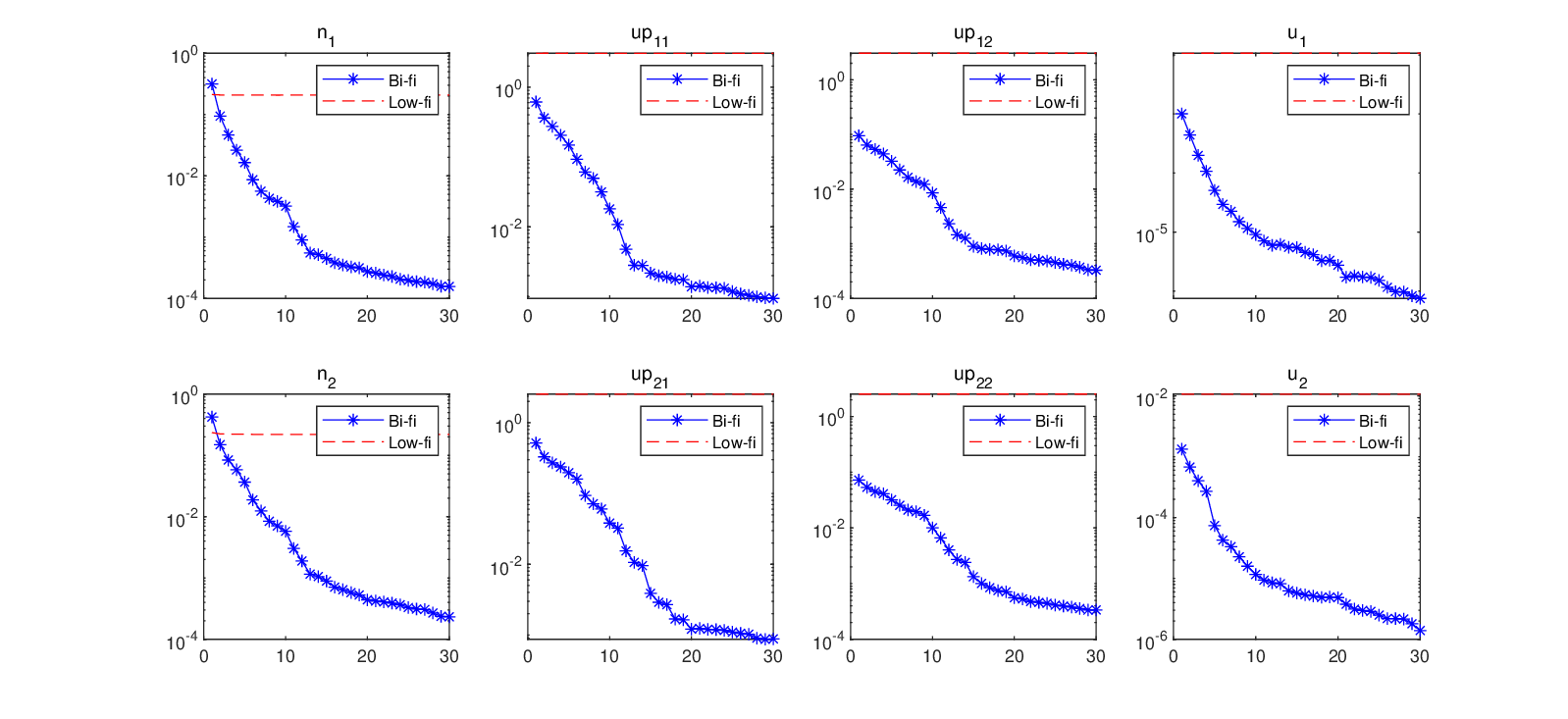}\\
		\caption{The mean $L^2$ errors between the high-fidelity and low- or bi- fidelity approximation  for particle density $n_{1}, n_{2}$ (first column), each component of particle velocity $u_{p,11}, u_{p,12}, u_{p,21}, u_{p,22}$ (second and third column) and fluid velocity $\mathbf{u}$ (fourth column) at $t=0.1$ for $\varepsilon=1$ ($N_x^H=128$, $N_x^L=32$). }\label{E1N32N128fig4}
	\end{figure}

	\begin{figure}[hp]
		\centering
		\includegraphics[width=14cm]{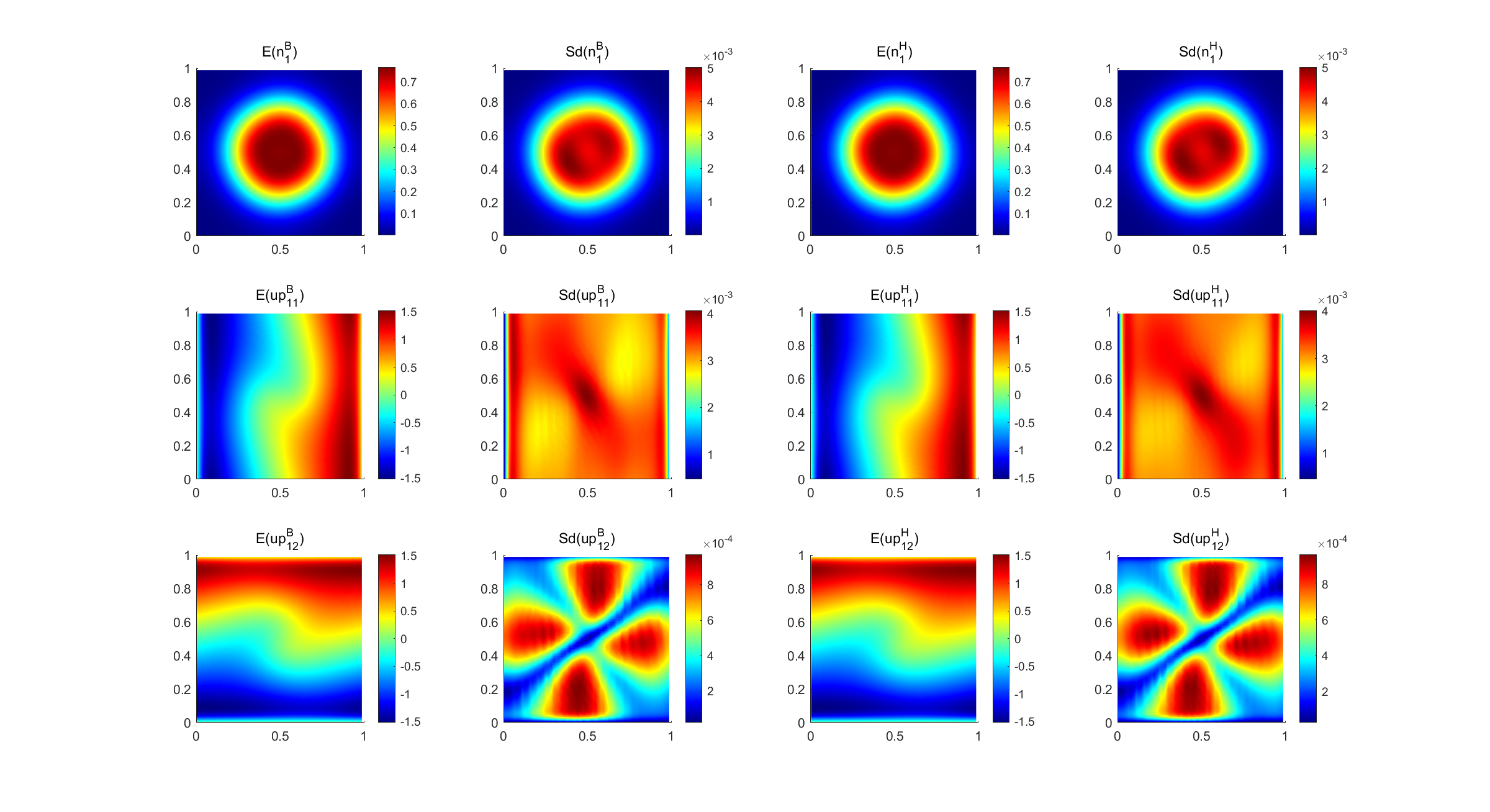}\\
		\caption{Mean and standard deviation of particle density $n_1$ and streamlines of particle velocity $u_{p,11}$, $u_{p,12}$ for bi-fidelity and high-fidelity solutions by using $r=10$ at $t=0.1$ for $\varepsilon=1$ ($N_x^H=128$, $N_x^L=64$).}\label{E1N64N128fig1}
	\end{figure}
	
	\begin{figure}[hp]
		\centering
		\includegraphics[width=14cm]{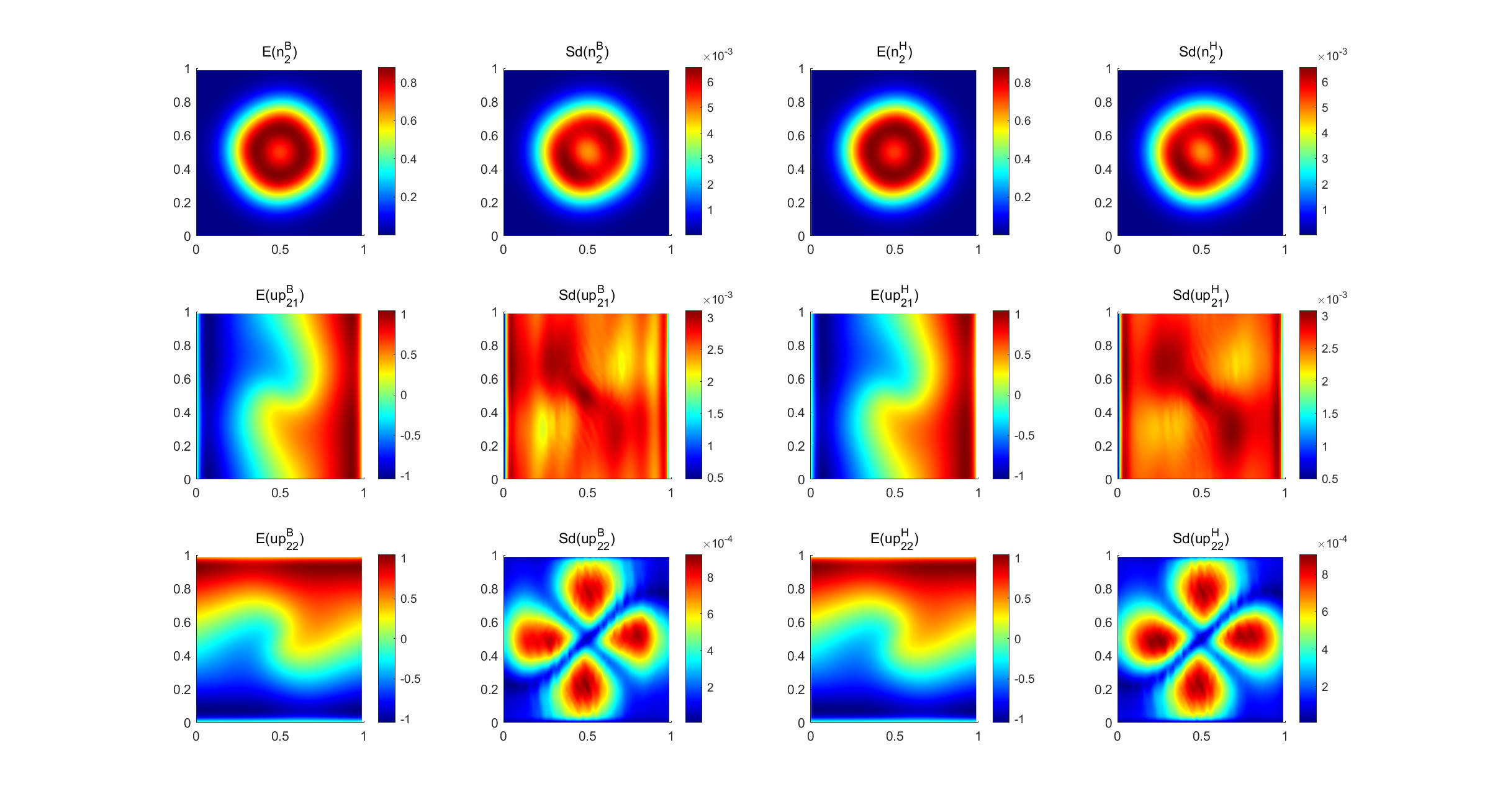}\\
		\caption{Mean and standard deviation of particle density $n_2$ and streamlines of particle velocity $u_{p,21}$, $u_{p,22}$ for bi-fidelity (the left two columns) and high-fidelity (the right two columns) solutions by using $r=10$ at $t=0.1$  for $\varepsilon=1$ ($N_x^H=128$, $N_x^L=64$).}\label{E1N64N128fig2}
	\end{figure}
	
	\begin{figure}[hp]
		\centering
		\includegraphics[width=14cm]{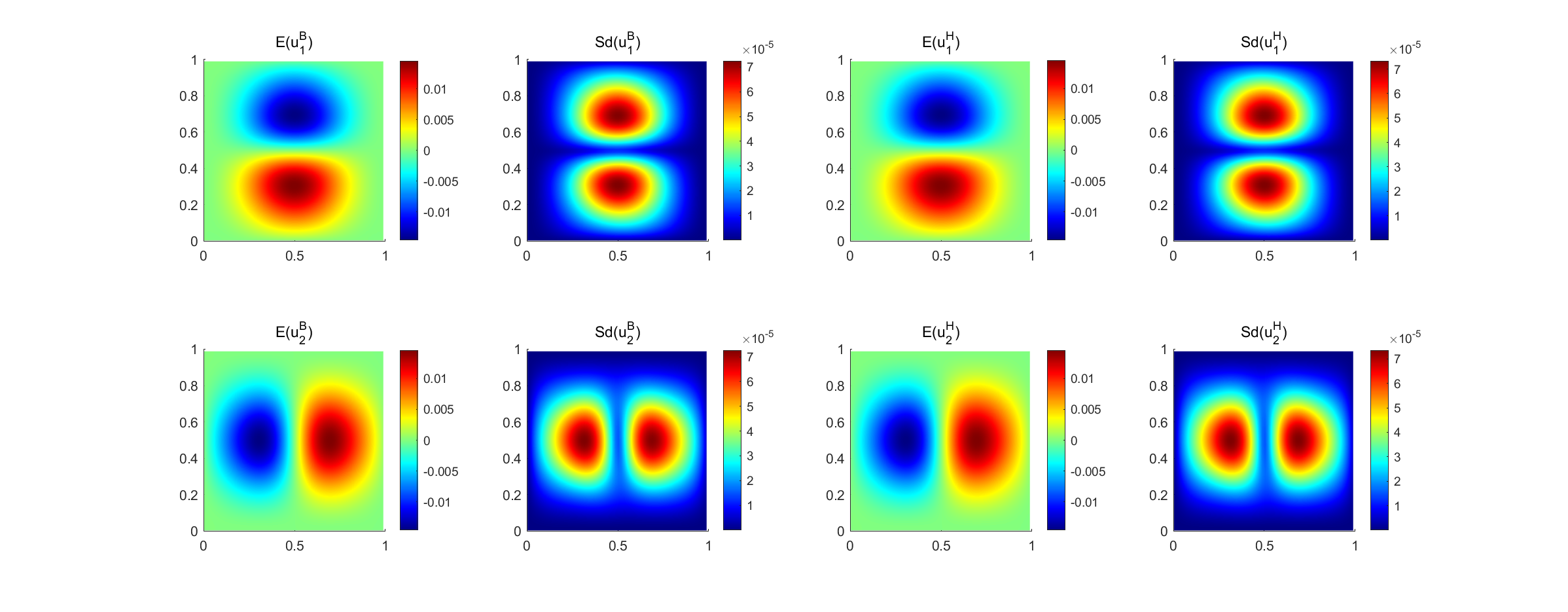}\\
		\caption{Mean and standard deviation of fluid velocity $u_1$, $u_2$ for bi-fidelity and high-fidelity solutions by using $r=10$ at $t=0.1$   for $\varepsilon=1$ ($N_x^H=128$, $N_x^L=64$).}\label{E1N64N128fig3}
	\end{figure}
	
	\begin{figure}[hp]
		\centering
		\includegraphics[width=14cm]{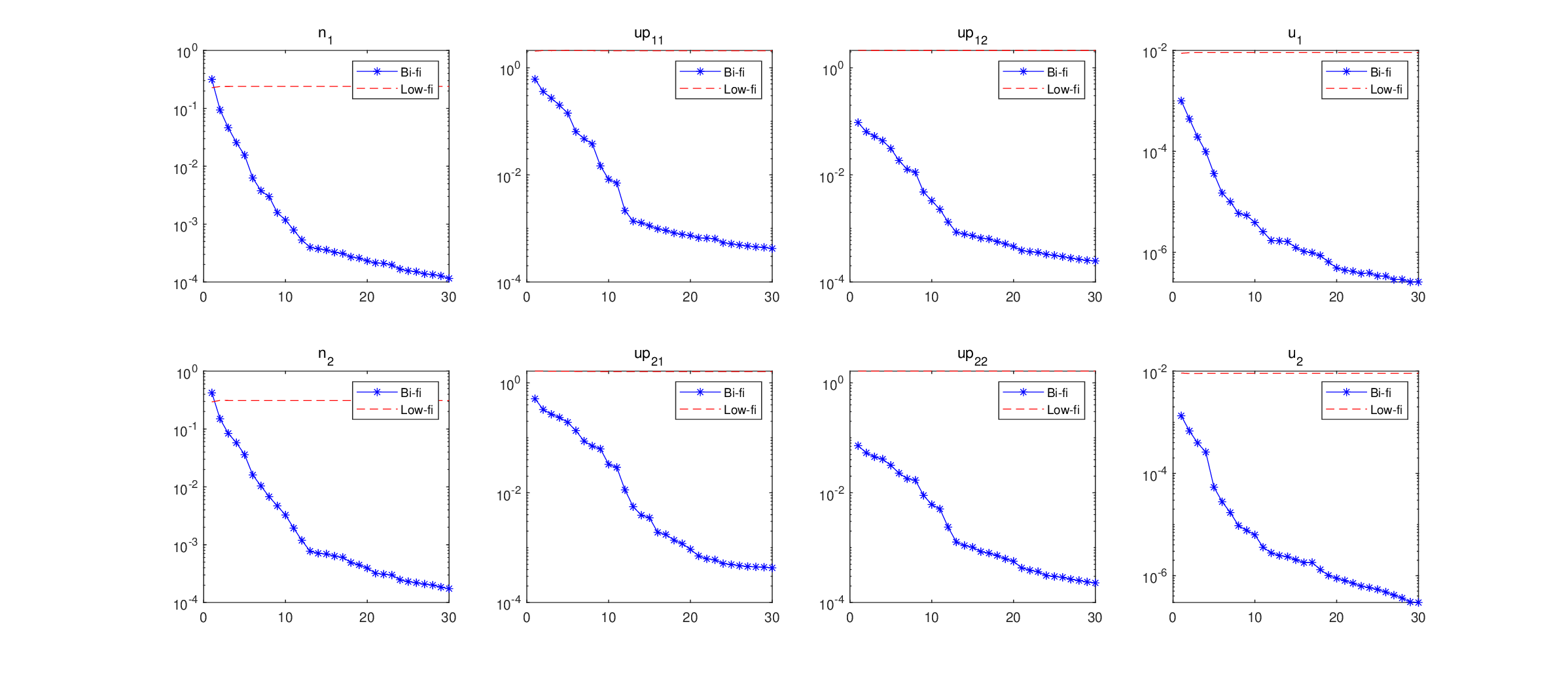}\\
		\caption{The mean $L^2$ errors between the high-fidelity and low- or bi- fidelity approximation  for particle density $n_{1}, n_{2}$ (first column), each component of particle velocity $u_{p,11}, u_{p,12}, u_{p,21}, u_{p,22}$ (second and third column) and fluid velocity $\mathbf{u}$ (fourth column) at $t=0.1$ for $\varepsilon=1$ ($N_x^H=128$, $N_x^L=64$). }\label{E1N64N128fig4}
	\end{figure}

	Figs. \ref{E1N64N128fig1} - \ref{E1N64N128fig3} show the  mean and standard deviation of the bi-fidelity solutions with  $N_x^L=64$ grid points for low-fidelity models with $\varepsilon=1$ at $t=0.1$. Similar to Figs. \ref{E1N32N128fig1} - \ref{E1N32N128fig3},
	velocities of particle and fluid are quite different. 
	because the drag force between the fluid phase and particle phases is not strong. The behavior turns out to be significantly different as $\varepsilon$ decreases. 
	Errors between the high-fidelity and low- or bi- fidelity approximations for particle density $n_1, n_2$, each component of particle velocity $u_{p,11}, u_{p,12}, u_{p,21}, u_{p,22}$ and fluid velocity $u_1, u_2$ are shown in Fig. \ref{E1N64N128fig4}. It is clear that the error decays faster with  bigger K. 
	Compared to Fig. \ref{E1N32N128fig4}, one observes that when $N_x^L$ increases, the error between the high- and bi-fidelity solutions decreases. This is expected because the multi-phases kinetic-fluid system solved by more spatial points can capture more information.

	\begin{figure}[hbp]
		\centering
		\includegraphics[width=14cm]{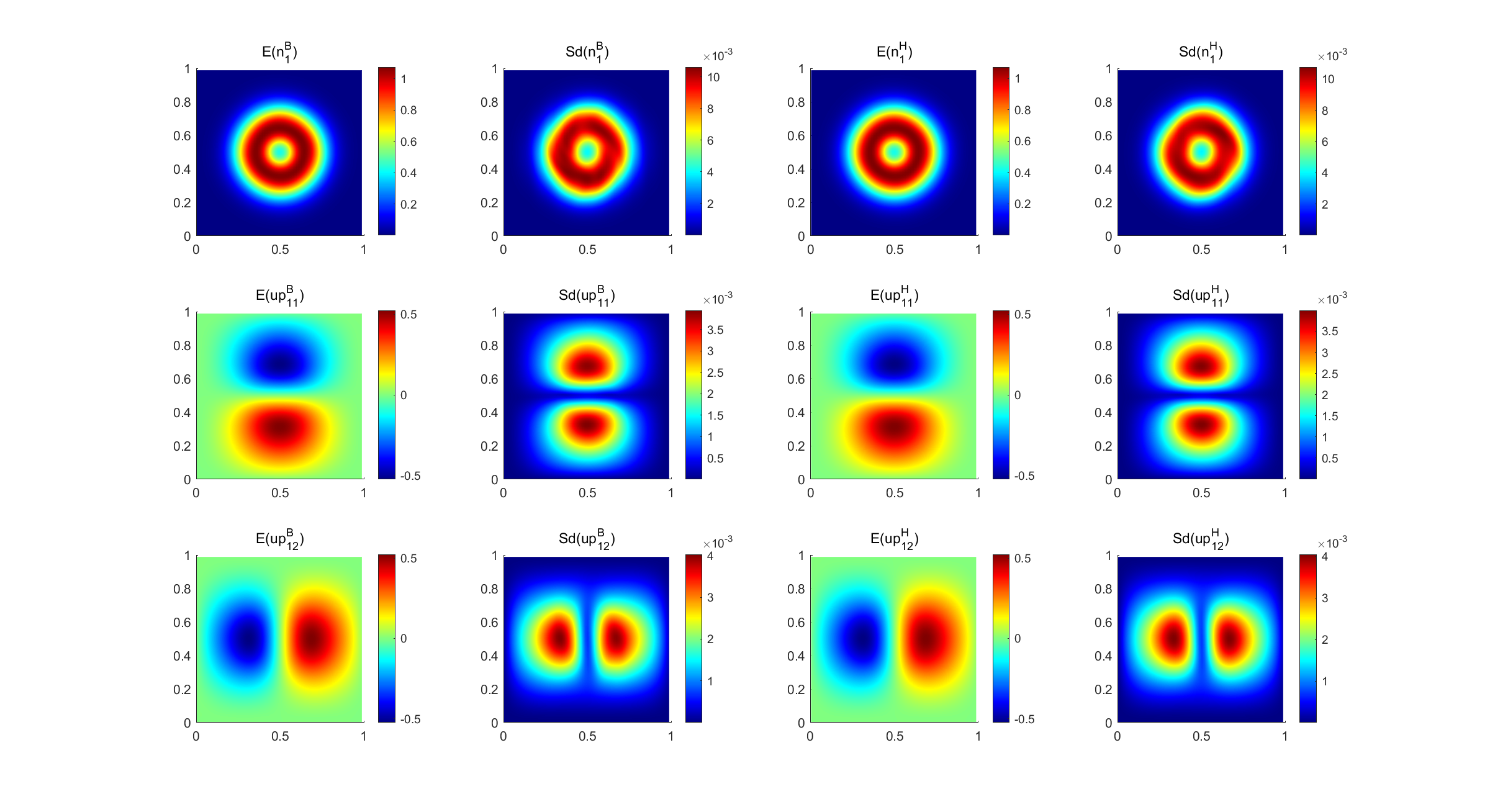}\\
		\caption{Mean and standard deviation of particle density $n_1$ and two components of particle velocity $up_{11}$, $up_{12}$ for bi-fidelity and high-fidelity solutions by using $r=10$ at $t=0.1$  for $\varepsilon=10^{-5}$ ($N_x^H=128$, $N_x^L=32$).}\label{E5N32N128fig1}
	\end{figure}
	
	\begin{figure}[htbp]
		\centering
		\includegraphics[width=14cm]{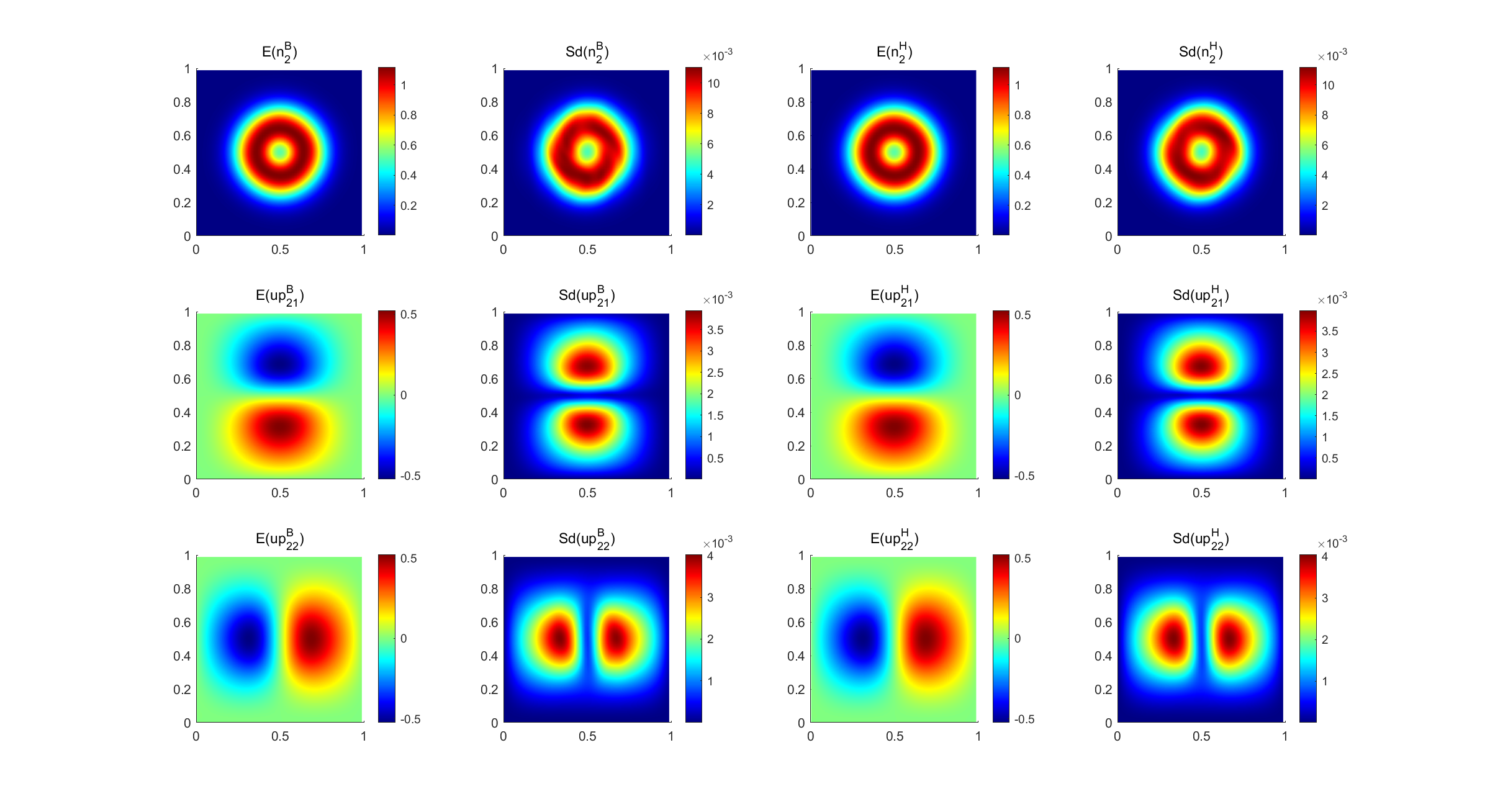}\\
		\caption{Mean and standard deviation of particle density $n_2$ and two components of particle velocity $up_{21}$, $up_{22}$ for bi-fidelity and high-fidelity solutions by using $r=10$ at $t=0.1$ for $\varepsilon=10^{-5}$ ($N_x^H=128$, $N_x^L=32$).}\label{E5N32N128fig2}
	\end{figure}
	
	\begin{figure}[htbp]
		\centering
		\includegraphics[width=14cm]{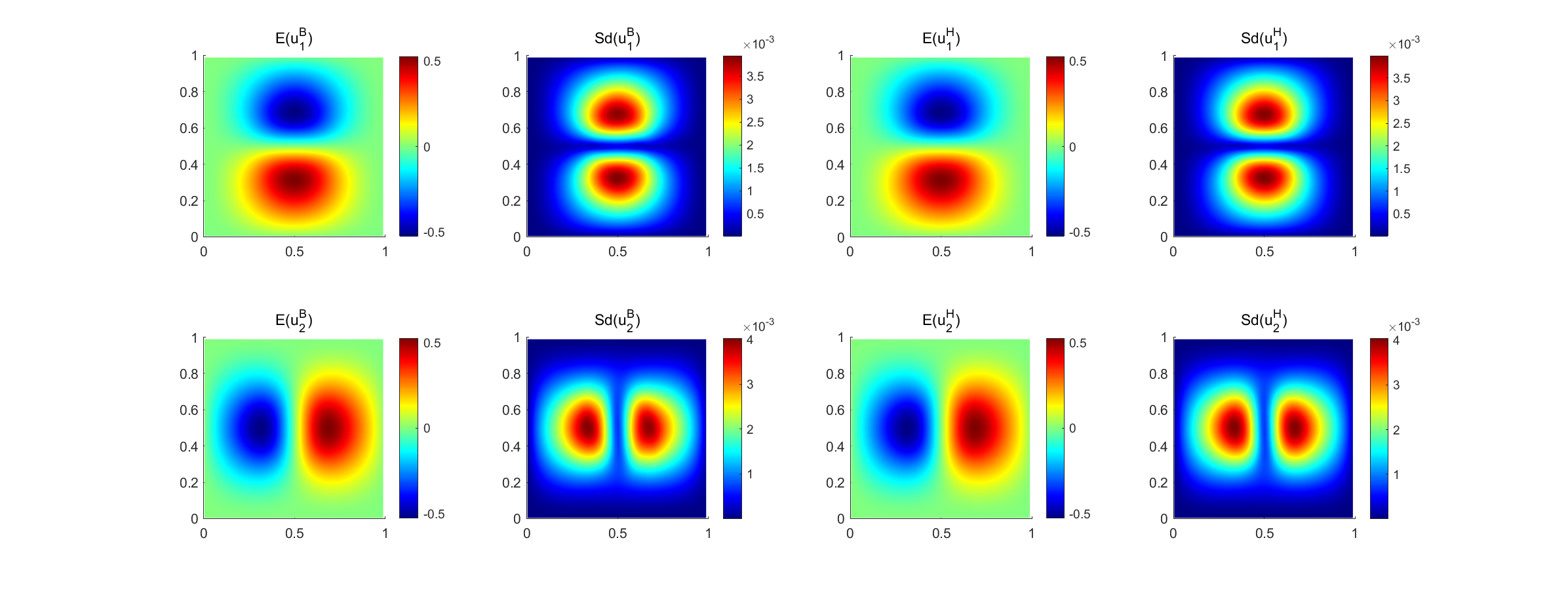}\\
		\caption{Mean and standard deviation of fluid velocity $u_1$, $u_2$ for bi-fidelity and high-fidelity solutions by using $r=10$ at $t=0.1$ for $\varepsilon=10^{-5}$ ($N_x^H=128$, $N_x^L=32$).}\label{E5N32N128fig3}
	\end{figure}
	
	\begin{figure}[htbp]
		\centering
		\includegraphics[width=14cm]{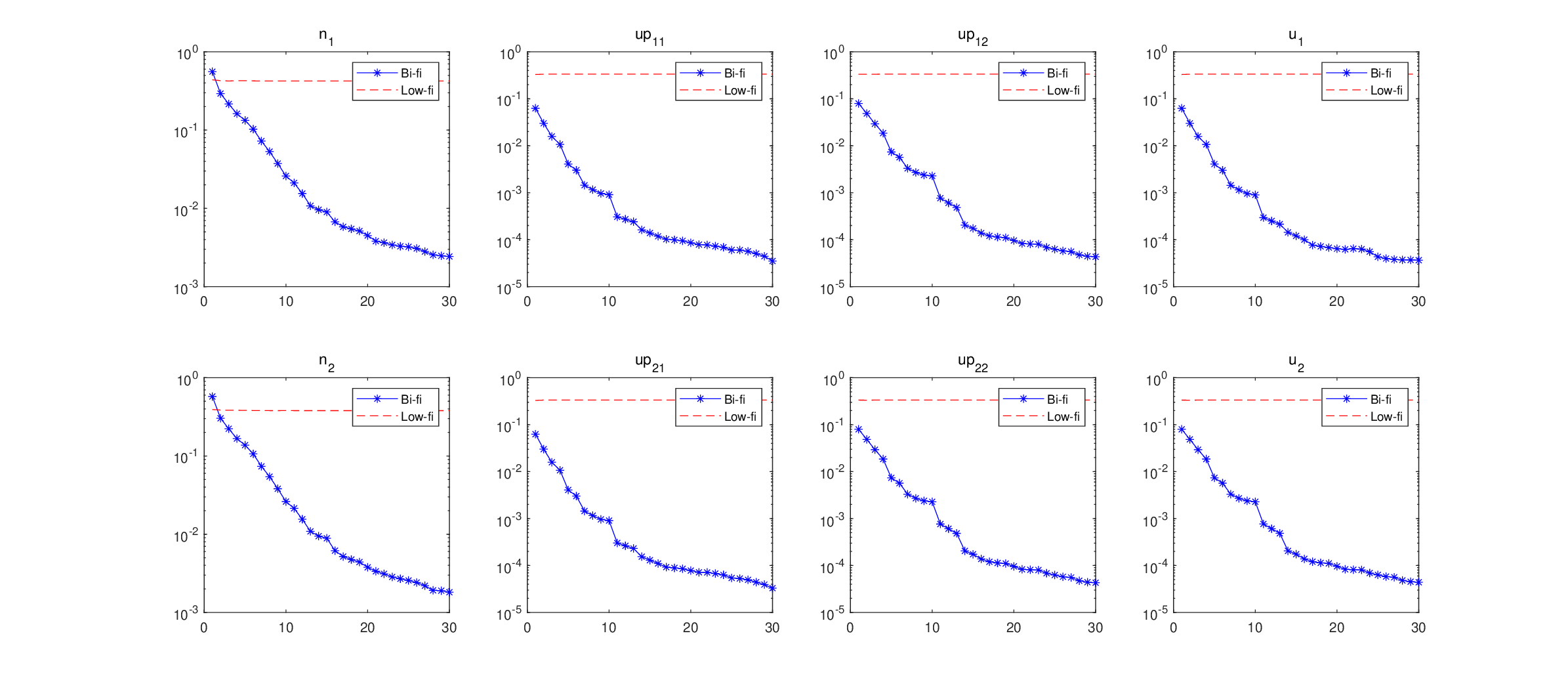}\\
		\caption{The mean $L^2$ errors between the high-fidelity and low- or bi- fidelity approximation for particle density $n_{1}, n_{2}$ (first column), each component of particle velocity $u_{p,11}, u_{p,12}, u_{p,21}, u_{p,22}$ (second and third column) and fluid velocity $\mathbf{u}$ (fourth column) at $t=0.1$ for $\varepsilon=10^{-5}$ ($N_x^H=128$, $N_x^L=32$). }\label{E5N32N128fig4}
	\end{figure}

	\begin{figure}[htbp]
		\centering
		\includegraphics[width=14cm]{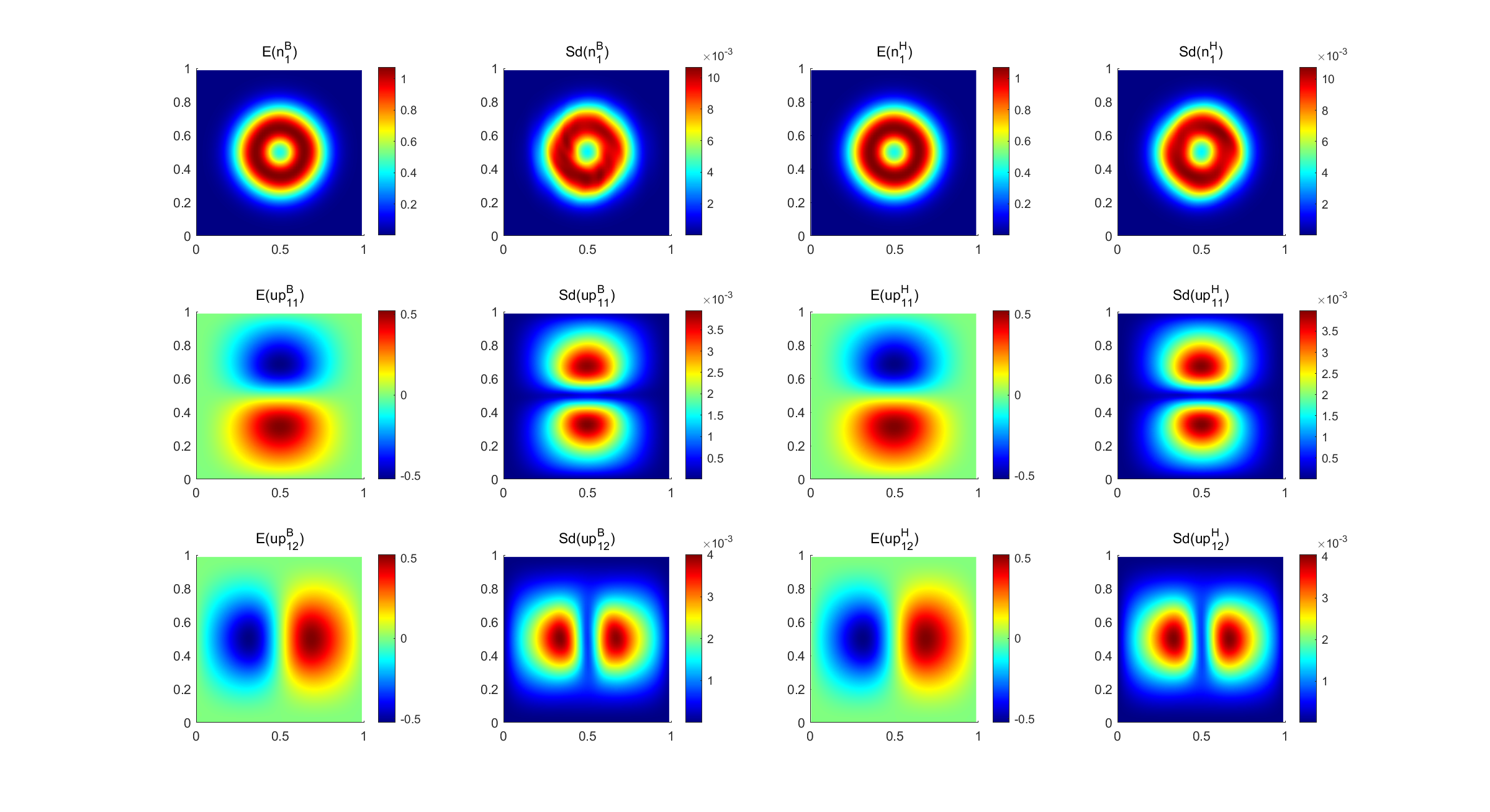}\\
		\caption{Mean and standard deviation of particle density $n_1$ and two components of particle velocity $up_{11}$, $up_{12}$ for bi-fidelity (the left two columns) and high-fidelity (the right two columns) solutions by using $r=10$ at $t=0.1$ for $\varepsilon=10^{-5}$ ($N_x^H=128$, $N_x^L=64$).}\label{E5N64N128fig1}
	\end{figure}
	
	\begin{figure}[htbp]
		\centering
		\includegraphics[width=14cm]{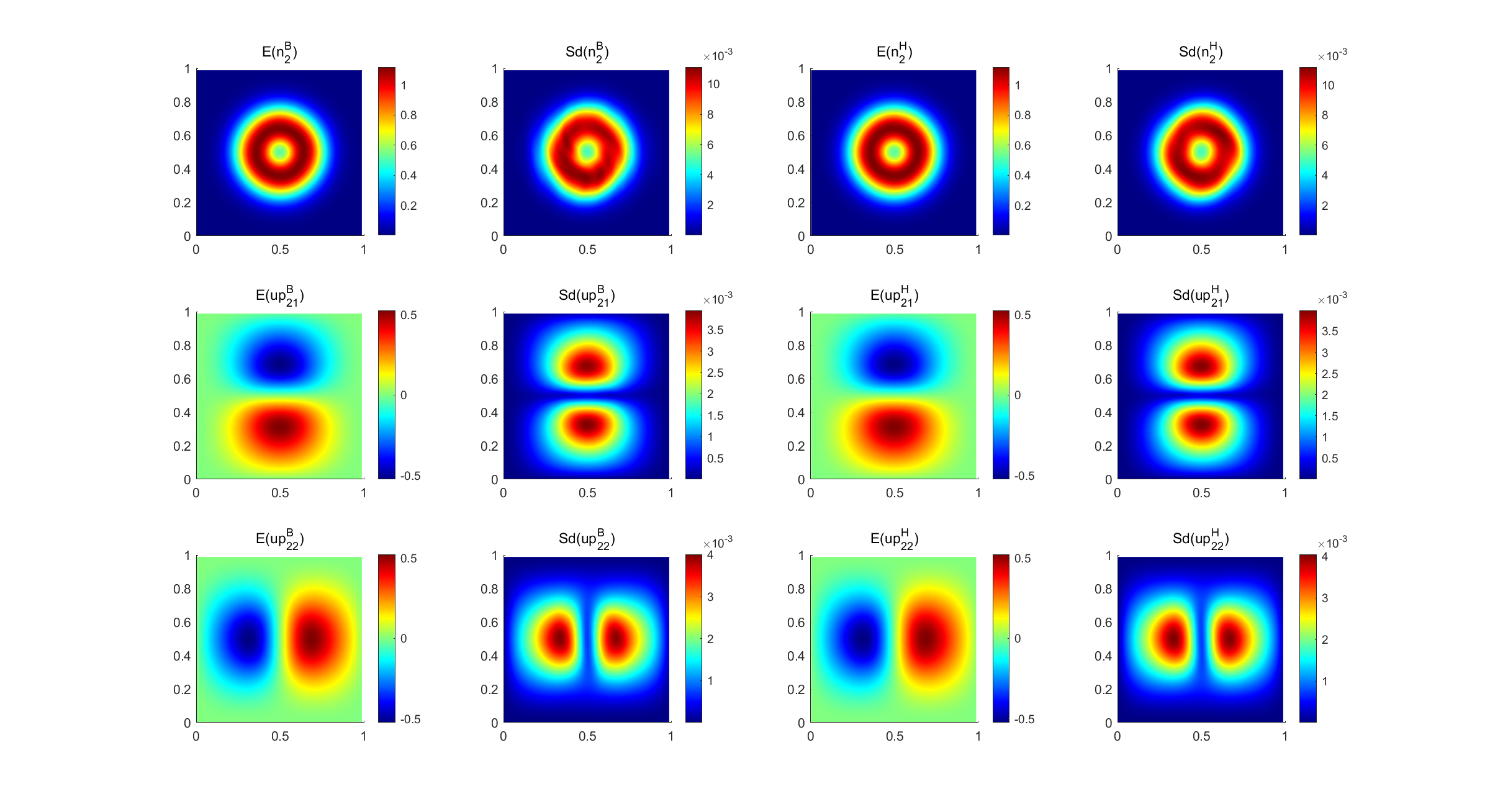}\\
		\caption{Mean and standard deviation of particle density $n_2$ and two components of particle velocity $up_{21}$, $up_{22}$ for bi-fidelity (the left two columns)  and high-fidelity (the right two columns) solutions by using $r=10$ at $t=0.1$ for $\varepsilon=10^{-5}$ ($N_x^H=128$, $N_x^L=64$).}\label{E5N64N128fig2}
	\end{figure}
	
	\begin{figure}[htbp]
		\centering
		\includegraphics[width=14cm]{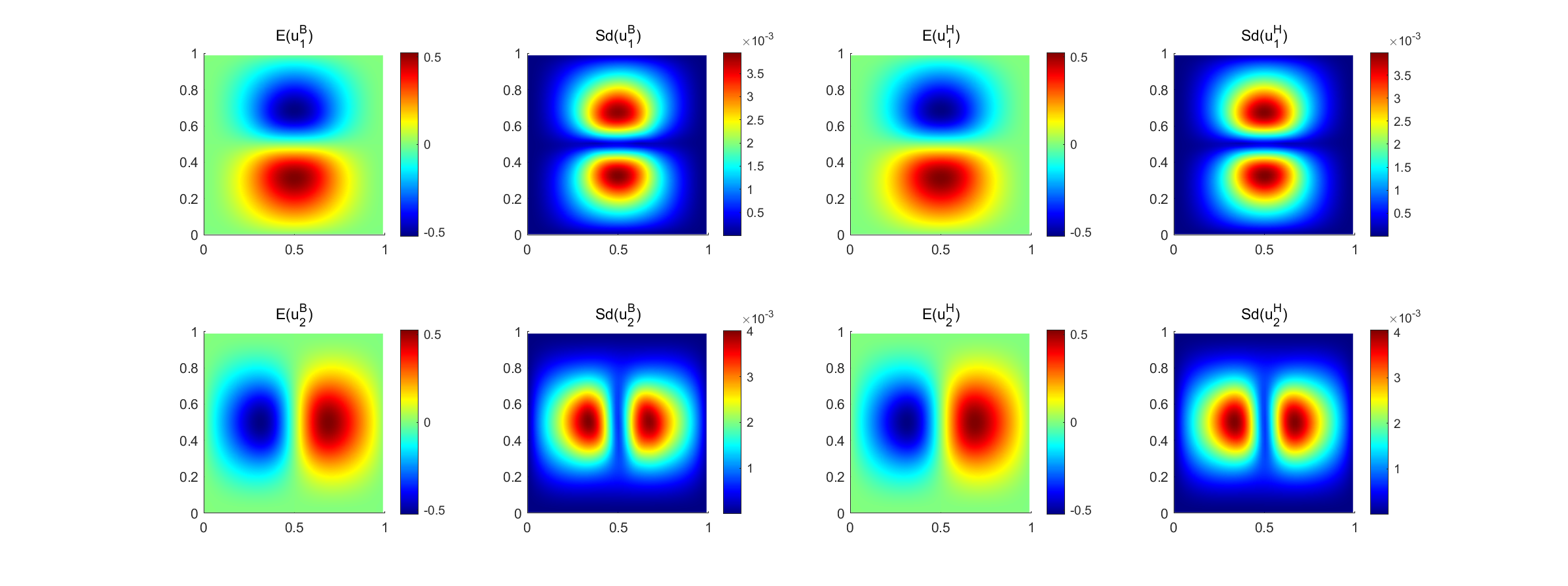}\\
		\caption{Mean and standard deviation of fluid velocity $u_1$, $u_2$ for bi-fidelity (the left two columns) and high-fidelity (the right two columns) solutions by using $r=10$ at $t=0.1$ for $\varepsilon=10^{-5}$ ($N_x^H=128$, $N_x^L=64$).}\label{E5N64N128fig3}
	\end{figure}
	
	\begin{figure}[htbp]
		\centering
		\includegraphics[width=14cm]{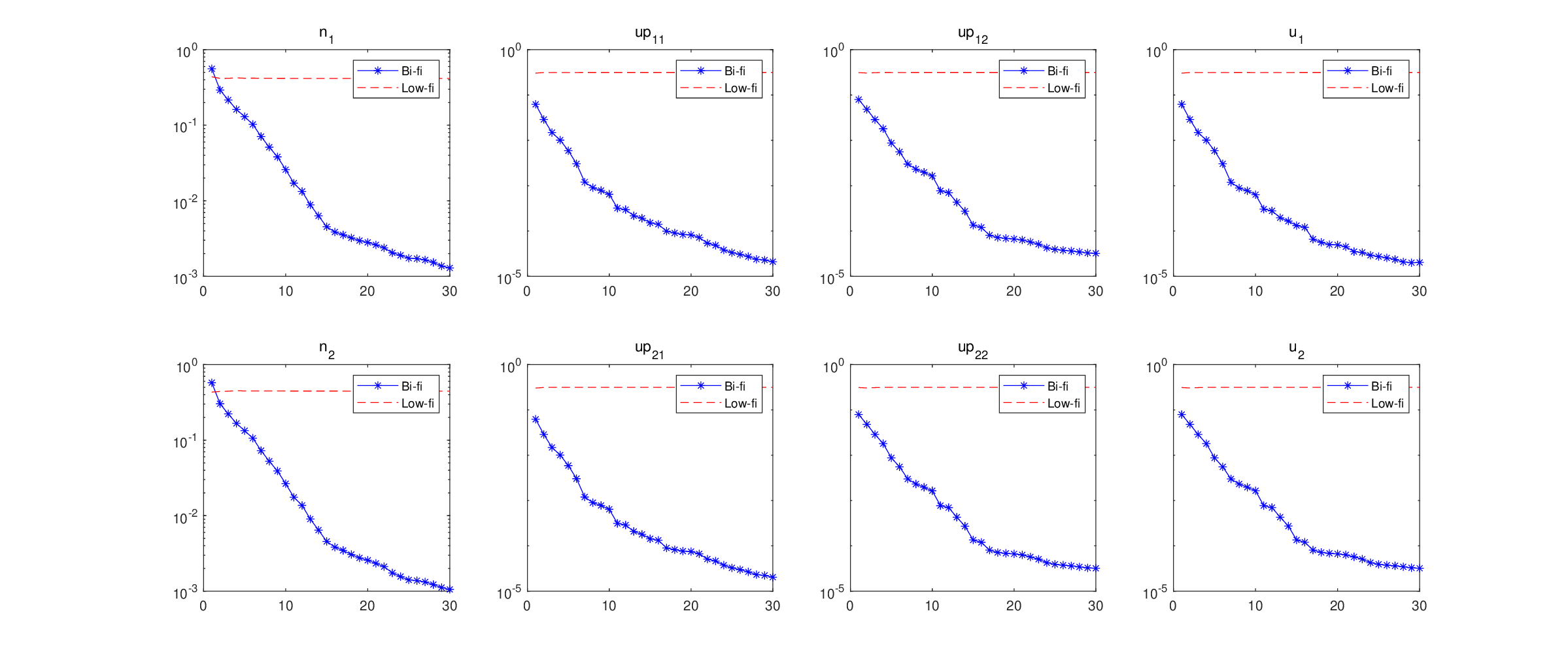}\\
		\caption{The mean $L^2$ errors between the high-fidelity and low- or bi- fidelity approximation for particle density $n_{1}, n_{2}$ (first column), each component of particle velocity $u_{p,11}, u_{p,12}, u_{p,21}, u_{p,22}$ (second and third column) and fluid velocity $\mathbf{u}$ (fourth column) at $t=0.1$ for $\varepsilon=10^{-5}$ ($N_x^H=128$, $N_x^L=64$). }\label{E5N64N128fig4}
	\end{figure}

	For the initial data \eqref{ex1} with $\varepsilon=10^{-5}$, see Figs. \ref{E5N32N128fig1} - \ref{E5N64N128fig4} for numerical results. 
	Figs \ref{E5N32N128fig1} - \ref{E5N32N128fig3} show the mean and standard deviation of the bi-fidelity solutions of particle density $n_1, n_2$ and each component of particle velocity $u_{p,11}, u_{p,12}, u_{p,21}, u_{p,22}$ and fluid velocity $u_1, u_2$ by using the high-fidelity solver 10 times at $t=0.1$, with $N_x^H=128$ grid points for high-fidelity models and $N_x^L=32$ grid points for low-fidelity models.
	Figs. \ref{E5N64N128fig1} - \ref{E5N64N128fig3} show the mean and standard deviation of the bi-fidelity solutions with $N_x^H=128$ grid points for high-fidelity models and $N_x^L=64$ grid points for low-fidelity models. 
	They all match quite well with the high-fidelity solutions. 
	In this case with smaller $\varepsilon$, the particles stop expanding immediately due to the strong drag force. The particles keep the volcano shape well in this period of time.
	A fast convergence of the errors given by \eqref{eq:err} with respect to the number of high-fidelity runs is clearly observed from Fig. \ref{E5N32N128fig4} and Fig. \ref{E5N64N128fig4}.
	In conclusion, a satisfactory accuracy in characterizing the behaviors of the solution in the random space is achieved in both cases: $N^L_x = 64$ and $N^L_x = 32$; and the saving of the computational cost in this test is quite remarkable.

	To further illustrate the performance of our bi-fidelity method, we compare the high-, low- and the corresponding bi-fidelity solutions (with $r=10$) for a particular sample point $z$. Figs \ref{E5N32N128fig6} - \ref{E5N32N128fig8} show the errors between the high-fidelity and low- or bi- fidelity approximation for macroscopic quantities including particle density $n_1$, particle velocity $u_{p,11}, u_{p,12}$ of the first particle $(i=1)$ in Fig. \ref{E5N32N128fig6},  those of the second particle $(i=2)$ in Fig. \ref{E5N32N128fig7} and fluid velocity in Fig. \ref{E5N32N128fig8} with $\varepsilon=10^{-5}$ at $t=0.1$, with $N_x^H = 128$ grid points for high-fidelity models and $N_x^L = 32$ grid points for low-fidelity models.
	With only 10 high-fidelity runs, one observes that the high- and bi-fidelity solutions match really well, whereas the low-fidelity solutions seem to be quite inaccurate at some points in the spatial domain. 
	This example shows the accuracy and efficiency of our bi-fidelity method.

	\begin{figure}[htbp]
		\centering
		\includegraphics[width=12cm]{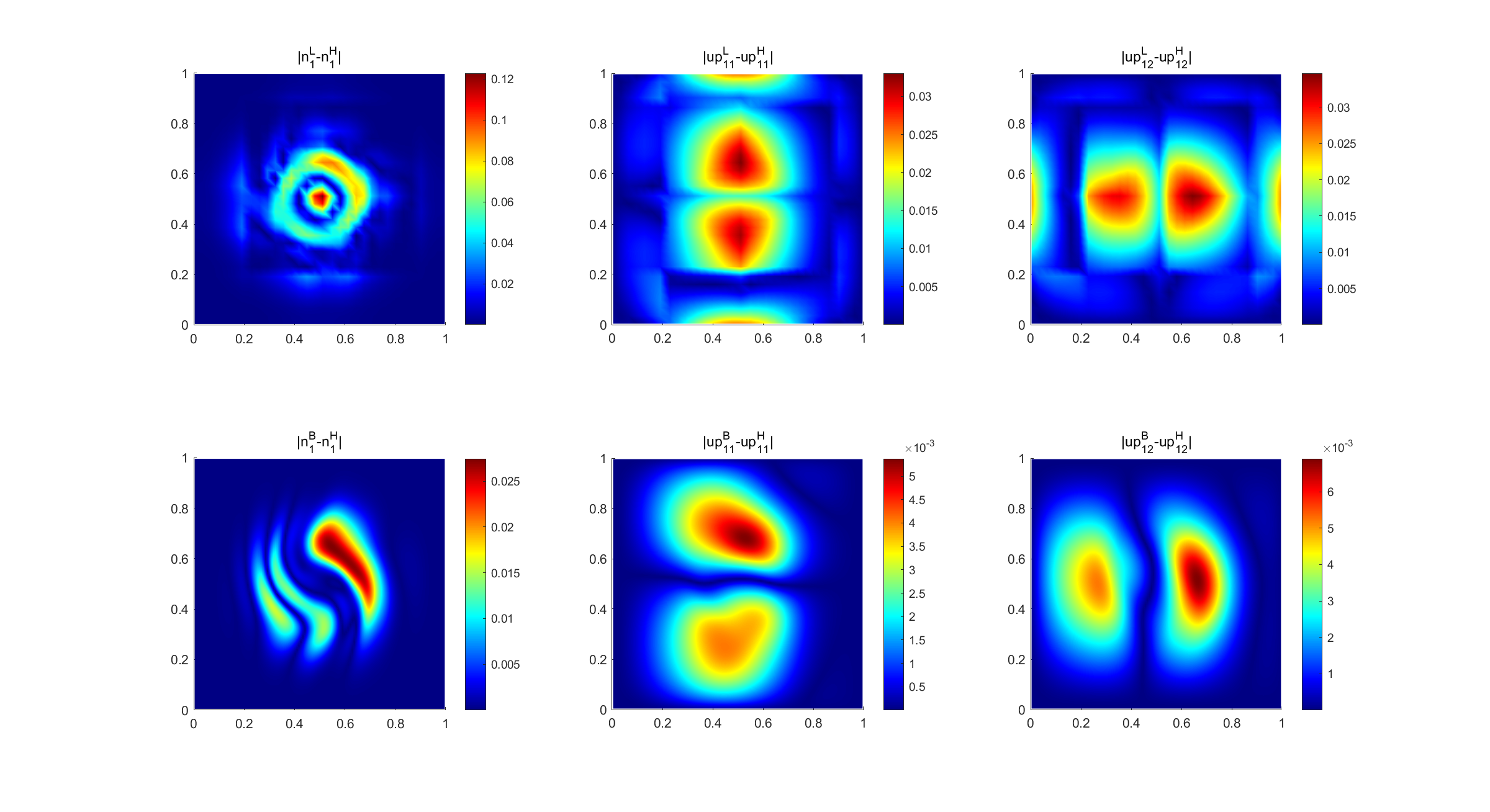}\\
		\caption{Errors for particle density $n_1$ (first column), two components of particle velocity $u_{p,11}, u_{p,12}$ (second and third column) for the first particle ($i=1$) at $t=0.1$ for $\varepsilon=10^{-5}$ ($N_x^H = 128, N_x^L = 32$). The first line is error  between the low- and high-fidelity solution. The second line is error between the bi- and high-fidelity solution.}\label{E5N32N128fig6}
	\end{figure}
	
	\begin{figure}[htbp]
		\centering
		\includegraphics[width=12cm]{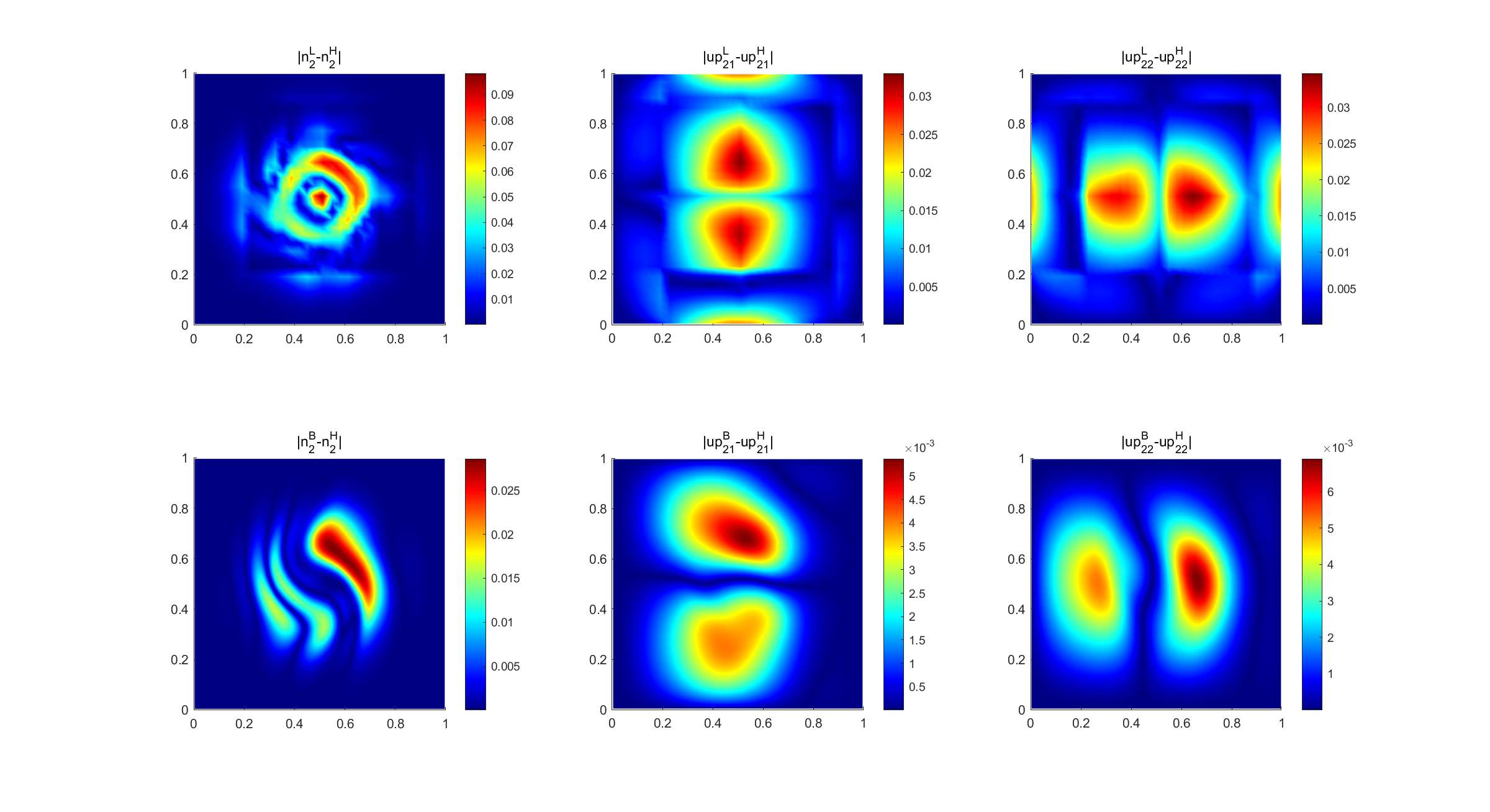}\\
		\caption{Errors for particle density $n_2$(first column), two components of particle velocity $u_{p,21}, u_{p,22}$ (second and third column) for the second particle ($i=2$) at $t=0.1$ for $\varepsilon=10^{-5}$ ($N_x^H = 128, N_x^L = 32$). The first line is error  between the low- and high-fidelity solution. The second line is error between the bi- and high-fidelity solution.} \label{E5N32N128fig7}
	\end{figure}
	
	\begin{figure}[htbp]
		\centering
		\includegraphics[width=10cm]{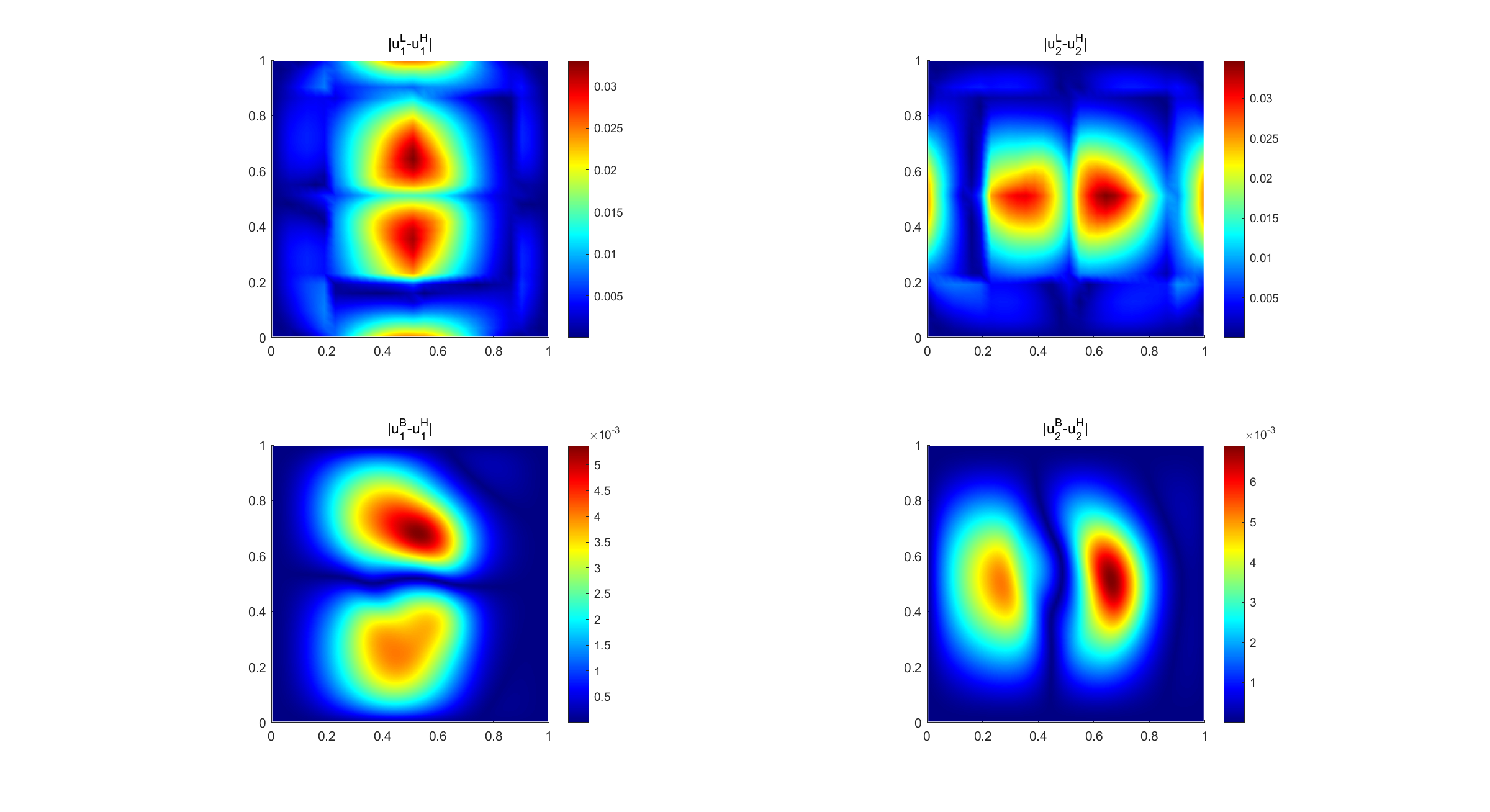}\\
		\caption{Errors for two components of fluid velocity $u_1, u_2$ at $t=0.1$ for $\varepsilon=10^{-5}$ ($N_x^H = 128, N_x^L = 32$). The first line is error  between the low- and high-fidelity solution. The second line is error between the bi- and high-fidelity solution.}\label{E5N32N128fig8}
	\end{figure}

	\begin{figure}[htbp]
		\centering
		\includegraphics[width=12cm]{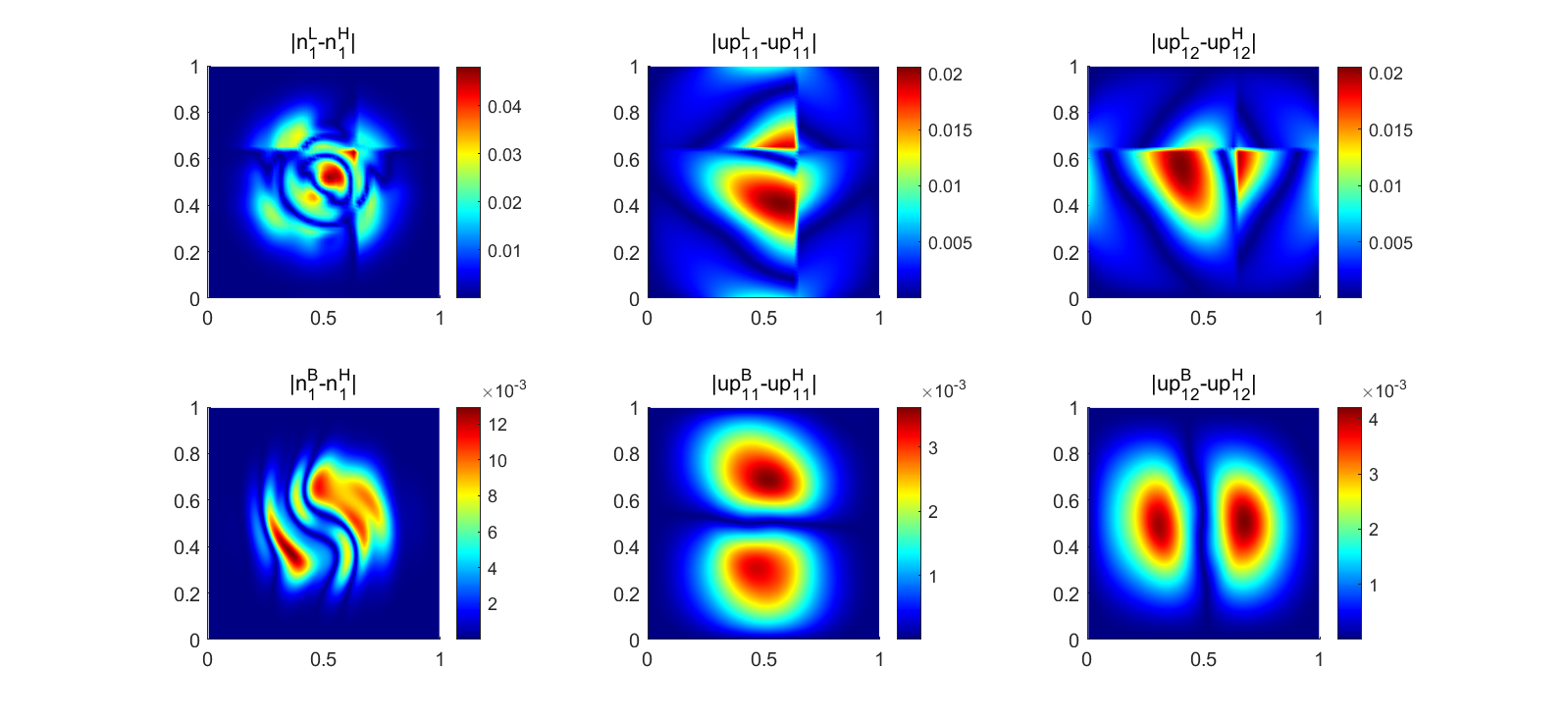}\\
		\caption{Errors for particle density $n_1$ (first column), two components of particle velocity $u_{p,11}, u_{p,12}$ (second and third column) for the first particle ($i=1$) at $t=0.1$ for $\varepsilon=10^{-5}$ ($N_x^H = 128, N_x^L = 64$). The first line is error  between the low- and high-fidelity solution. The second line is error between the bi- and high-fidelity solution.}\label{E5N64N128fig6}
	\end{figure}
	
	\begin{figure}[htbp]
		\centering
		\includegraphics[width=12cm]{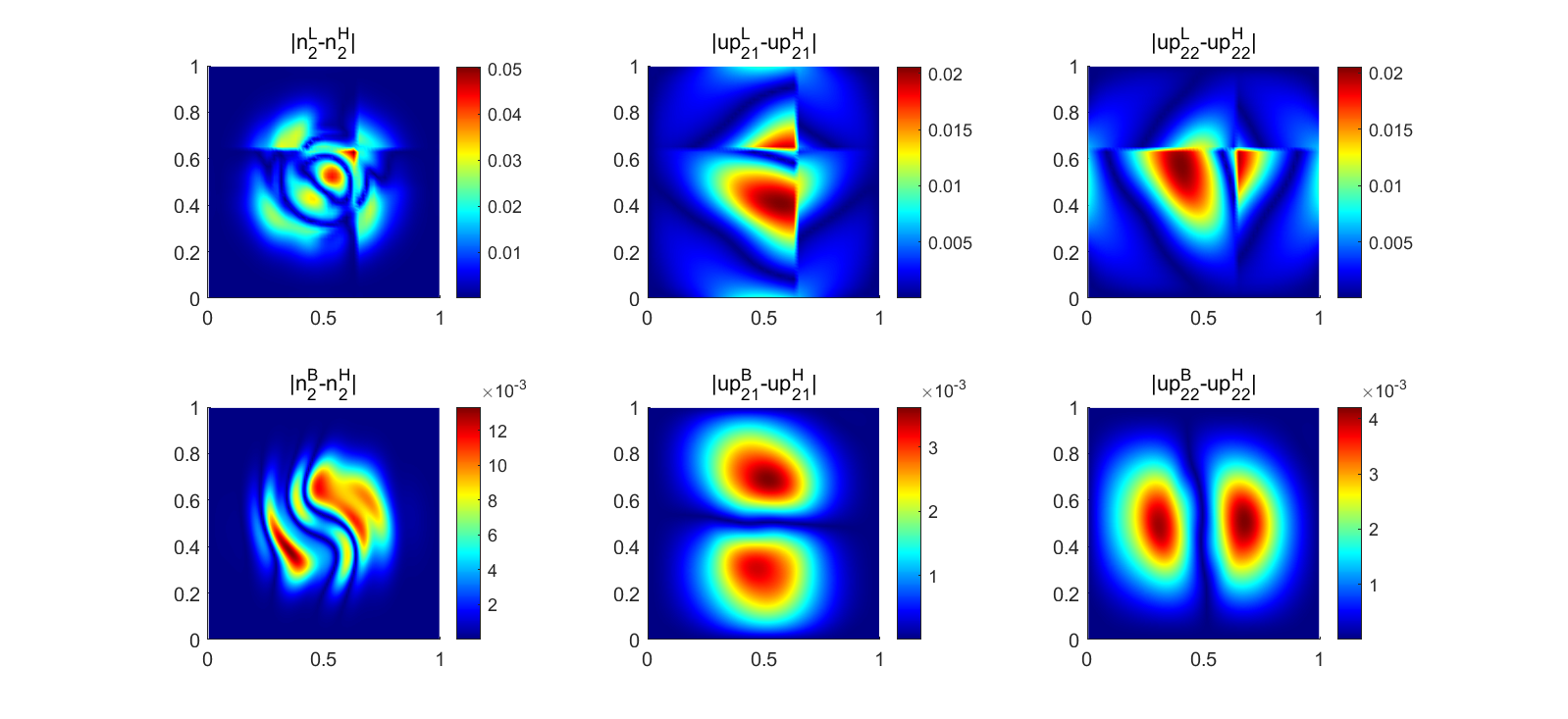}\\
		\caption{Errors for particle density $n_2$ (first column), two components of particle velocity $u_{p,21}, u_{p,22}$ (second and third column) for the second particle ($i=2$) at $t=0.1$ for $\varepsilon=10^{-5}$ ($N_x^H = 128, N_x^L = 64$). The first line is error  between the low- and high-fidelity solution. The second line is error between the bi- and high-fidelity solution.} \label{E5N64N128fig7}
	\end{figure}
	
	\begin{figure}[htbp]
		\centering
		\includegraphics[width=10cm]{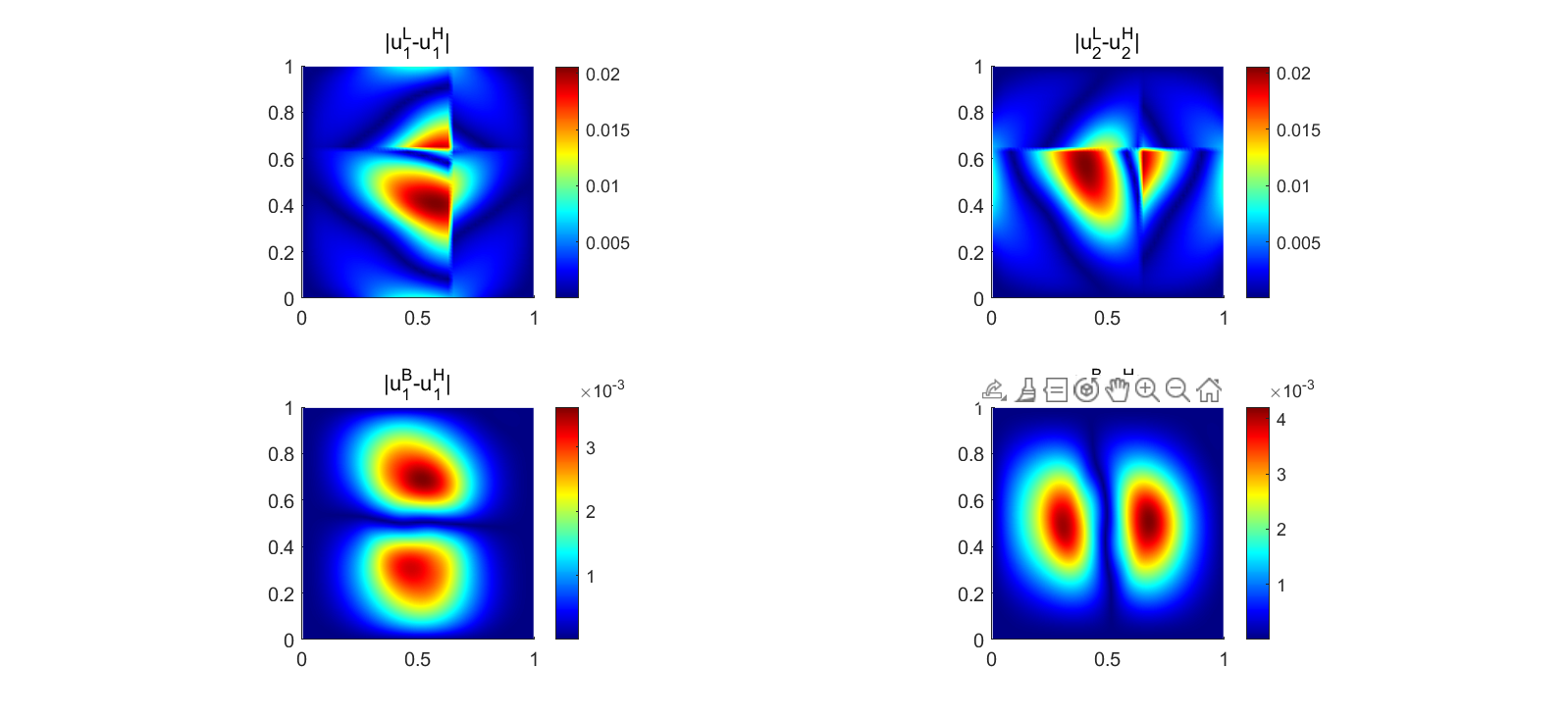}\\
		\caption{Errors for two components of fluid velocity $u_1, u_2$ at $t=0.1$ for $\varepsilon=10^{-5}$ ($N_x^H = 128, N_x^L = 64$). The first line is error  between the low- and high-fidelity solution. The second line is error between the bi- and high-fidelity solution.}\label{E5N64N128fig8}
	\end{figure}

	Figs. \ref{E5N64N128fig6}-\ref{E5N64N128fig8} show the errors between the high-fidelity and low- or bi- fidelity approximation for macroscopic quantities with $N_x^L = 64$ grid points fo low-fidelity models with $\varepsilon=10^{-5}$.
	From Figs. \ref{E5N64N128fig6}-\ref{E5N64N128fig8}, we see a good convergence of errors between the high- and bi-fidelity solutions. 
	One can get the bi-fidelity solutions which are able to capture behavior of the solutions to the multi-phase kinetic-fluid equations in the random space, up to an accuracy $\sim 10^{-3}$; on the other hand, using the low-fidelity model alone can not achieve this result.
	This observation certainly highlights the merits of our bi-fidelity method.

	\section{Conclusion}
	\label{sec:conclusion}
	
	In this paper, error estimates of the bi-fidelity method for solving multi-phase kinetic-fluid equations with random initial inputs are obtained. The result is uniform in the Stokes number. 
	The idea and analysis can be applied to other multiscale kinetic equations with uncertainties, as the regularity analysis in previous work is applicable to various types of linear or nonlinear kinetic equations. This work provides a range of low-fidelity models suitable for the bi-fidelity method, which allows effective handling of challenges posed by high-dimensionality, nonlinearity, coupling and randomness in computation when solving kinetic equations with random parameters.

	\section*{Acknowledgement}
	S. Jin was partially supported by National Key R\&D Program of China (no. 2020YFA0712000), the NSFC grant No. 12031013, and the Fundamental Research Funds for the Central Universities. 
	Y. Lin was partially supported by China Postdoctoral Science Foundation (no. 2021M702142) and National Natural Science Foundation of China (no. 12201404). The computations in this paper were run on the Siyuan-1 cluster supported by the Center for High Performance Computing at Shanghai Jiao Tong University.

	\section* {Appendix A.1 Proof of Lemma \ref{thm:energyestimate}}
	\begin{proof}
		Taking $z$-derivative of order $\gamma$ and $x$-derivative of order $\alpha$ of \eqref{eq:uf}, and taking $z$-derivative of order $\gamma$ of \eqref{eq:ubar}-\eqref{eq:ubar2} gives
		\begin{equation}\label{eq:ufgamma}
			\begin{aligned}
				&\partial_{t} \partial^{\alpha} f_i^{\gamma}+ v \cdot \nabla_{x} \partial^{\alpha} f_i^{\gamma}+\frac{i^{1/3}}{\bar{\theta}\epsilon}\left(\dfrac{\bar{\theta}}{i}\nabla_{v}-\frac{v}{2}\right) \cdot \partial^{\alpha}(u f_i)^{\gamma}
				\underbrace{-\frac{i^{1/3}}{\bar{\theta}\epsilon\delta} \partial^{\alpha} u^{\gamma} \cdot v \sqrt{\mu_i}}\\
				&\begin{aligned}
					=\underbrace{\frac{i^{1/3}}{\bar{\theta}\epsilon}\left(\frac{-|v|^{2}}{4}+\frac{3}{2}\dfrac{\bar{\theta}}{i}+\dfrac{\bar{\theta}^2}{i^2}\Delta_{v}\right) \partial^{\alpha} f_i^{\gamma}},\\
					\partial_{t} \partial^{\alpha} u^{\gamma}+\partial^{\alpha}\left(u \cdot \nabla_{x} u\right)^{\gamma}+\nabla_{x} \partial^{\alpha} p^{\gamma} \underline{-\Delta_{x} \partial^{\alpha} u^{\gamma}} \underbrace{+\dfrac{\kappa}{\epsilon}\sum_{i=1}^{N}i^{1/3}\partial^{\alpha} u^{\gamma}}+\dfrac{\kappa\delta}{\epsilon}\sum_{i=1}^{N}i^{1/3}\int \sqrt{\mu_i} \partial^{\alpha}(u f_i)^{\gamma} \mathrm{d} v\\ \underbrace{-\dfrac{\kappa\delta}{\epsilon}\sum_{i=1}^{N}i^{1/3} \int v \sqrt{\mu_i} \partial^{\alpha} f_i^{\gamma} \mathrm{d} v}=0, \end{aligned}\\
				&\nabla_{x} \cdot \partial^{\alpha} u^{\gamma}=0, \\
				& \partial_{t} \bar{u}^{\gamma} \underline{+\dfrac{\kappa}{\epsilon}\sum_{i=1}^{N}i^{1/3} \bar{u}^{\gamma}}+\dfrac{\kappa\delta}{\epsilon}\frac{1}{|\mathbb{T}|^{3}}\sum_{i=1}^{N}i^{1/3} \iint \sqrt{\mu_i}(u f_i)^{\gamma} \mathrm{d} v \mathrm{d} x =\underline{\frac{\kappa\delta}{\epsilon} \frac{1}{|\mathbb{T}|^{3}} \sum_{i=1}^{N} \int\int  i^{1/3} v \sqrt{\mu_i} f_i^\gamma \mathrm{d} v \mathrm{d} x},\\			
				& -\frac{1}{|\mathbb{T}|^{3}}\sum_{i=1}^{N}\int\int\delta v \sqrt{\mu_i} f_i^\gamma \mathrm{d} v \mathrm{d} x=\bar{u}^\gamma .
			\end{aligned}
		\end{equation}
		Now do $L^{2}$ estimate on each equation above (except the third one), i.e., multiply the first equation by $\kappa\bar{\theta}\partial^{\alpha} f_i^{\gamma}$, integrate in $(v, x)$ and sum over $i$; multiply the second equation by $\dfrac{\partial^{\alpha} u^{\gamma}}{\delta^2}$ and integrate in $x$;  multiply the fourth equation by $\dfrac{\bar{u}^{\gamma}}{\delta^2}$. Finally add the results together and sum over $|\gamma| \leq r,|\alpha| \leq s$. Then one gets the following equation (at each $z$):
		$$
		\frac{1}{2} \partial_{t} E+G+B\leq 0,
		$$
		where the energy $E$ is defined in \eqref{def:E0}. The good terms $G$ are given by
		$$
		G=\underline{G_{1}}+\underbrace{G_{2}}_{|\gamma| \leq s}=\underline{G_{1}}+\dfrac{\kappa}{\epsilon}\sum_{i=1}^{N}i^{1/3}\underbrace{G_{2,i}}_{|\gamma| \leq s}= \frac{1}{\delta^2} G_{1, \gamma}+\dfrac{\kappa}{\epsilon\delta^2}\sum_{i=1}^{N}i^{1/3}\sum_{|\gamma| \leq s} G_{2, i,\gamma},
		$$
		with
		$$
		\begin{aligned}
			G_{1, \gamma} &=\left|\nabla_{x} u^{\gamma}\right|_{s}^{2}+\dfrac{\kappa}{\epsilon}\left(\sum_{i=1}^{N}i^{1/3}-1\right)\left|\bar{u}^{\gamma}\right|^{2} 
			, \\
			G_{2,i, \gamma} &=\left|u^{\gamma} \sqrt{\mu_i}-\delta\frac{\bar{\theta}}{i} \nabla_{v} f_i^{\gamma}- \delta\frac{v}{2} f_i^{\gamma}\right|_{s}^{2},
		\end{aligned}
		$$
		$G_{1}$ and $G_{2}$ come from the underlined terms and the underbraced terms in \eqref{eq:ufgamma} respectively. To verify the $G_{2}$ term, we provide the following calculation:
		$$
		\begin{aligned}
			\left\langle\partial^{\alpha} u^{\gamma}, \partial^{\alpha} u^{\gamma}\right\rangle-&\left\langle \delta \int v \sqrt{\mu_i} \partial^{\alpha} f_i^{\gamma} \mathrm{d} v, \partial^{\alpha} u^{\gamma}\right\rangle
			-\left\langle\delta\partial^{\alpha} u^{\gamma} \cdot v \sqrt{\mu_i}, \partial^{\alpha} f_i^{\gamma}\right\rangle\\
			&-\left\langle\delta^2\left(\frac{-|v|^{2}}{4}+\frac{3}{2}\dfrac{\bar{\theta}}{i}+\dfrac{\bar{\theta}^2}{i^2}\Delta_{v}\right) \partial^{\alpha} f_i^{\gamma}, \partial^{\alpha} f_i^{\gamma}\right\rangle 
			=\left|\partial^{\alpha}\left(u^{\gamma} \sqrt{\mu_i}-\delta\frac{\bar{\theta}}{i} \nabla_{v} f_i^{\gamma}-\delta\frac{v}{2} f_i^{\gamma}\right)\right|_{0}^{2}.
		\end{aligned}
		$$
		The notation $|\cdot|_{0}$ is defined in  \eqref{Sobolevnorm} with $s=r=0$, i.e., taking $L_{x, v}^{2}$ norm for a fixed $z$.
		
		The bad terms $B$ are given by
		$$
		\begin{aligned}
			B&=B_{1}+\dfrac{\kappa}{\epsilon}\sum_{i=1}^{N}i^{1/3}B_{2,i}+\dfrac{\kappa}{\epsilon}\sum_{i=1}^{N}i^{1/3}B_{3,i}\\
			&=\frac{1}{\delta^2}\sum_{|\gamma| \leq r,|\alpha| \leq s} B_{1, \alpha, \gamma}+\dfrac{\kappa}{\epsilon\delta^2}\sum_{i=1}^{N}i^{1/3}\sum_{|\gamma| \leq r,|\alpha| \leq s} B_{2, i,\alpha, \gamma}+\dfrac{\kappa}{\epsilon\delta^2}\sum_{i=1}^{N}i^{1/3}\sum_{|\gamma| \leq r} B_{3, i,\gamma},
		\end{aligned}
		$$
		with
		$$
		\begin{aligned}
			&B_{1, \alpha, \gamma}=\left\langle\partial^{\alpha}\left(u \cdot \nabla_{x} u\right)^{\gamma}, \partial^{\alpha} u^{\gamma}\right\rangle, \\
			&B_{2,i,\alpha, \gamma}=\left\langle\delta\partial^{\alpha}(u f_i)^{\gamma}, \partial^{\alpha}\left[u^{\gamma} \sqrt{\mu_i}-\delta\frac{\bar{\theta}}{i} \nabla_{v} f_i^{\gamma}- \delta\frac{v}{2} f_i^{\gamma}\right]\right\rangle, \\
			&B_{3,i, \gamma}=\frac{1}{|\mathbb{T}|^{3}}\left\langle\delta(u f_i)^{\gamma}, \bar{u}^{\gamma} \sqrt{\mu_i}\right\rangle,
		\end{aligned}
		$$
		coming from the nonlinear terms.
		
		By using Lemma 3.3 in \cite{JinLin2022}, the bad terms are controlled by
		$$
		\begin{aligned}
			&\frac{1}{\delta^2}\left|B_{1, \alpha, \gamma}\right| \leq \frac{C(s,r)}{\tau\delta^2}|u|_{s+1, r}^{2}|u|_{s, r}^{2}+\frac{\tau}{\delta^2}|u|_{s, r}^{2} \leq \frac{C}{\tau} E G_{1}+\tau G_{1}, \\
			&\dfrac{\kappa}{\epsilon\delta^2}\sum_{i=1}^{N}i^{1/3}\left|B_{2, i,\alpha, \gamma}\right| \leq \dfrac{\kappa}{\epsilon}\sum_{i=1}^{N}i^{1/3}\left(\frac{C(s,r)}{\tau\delta^2}|u|_{s, r}^{2}|f_i|_{s, r}^{2}+\frac{\tau}{\delta^2}\left|u^{\gamma} \sqrt{\mu_i}-\delta\frac{\bar{\theta}}{i} \nabla_{v} f_i^{\gamma}-\delta \frac{v}{2} f_i^{\gamma}\right|_{s} \right)\leq \frac{C}{\tau} E G_{1}+\tau G_{2}, \\
			&\dfrac{\kappa}{\epsilon\delta^2}\sum_{i=1}^{N}i^{1/3}\left|B_{3,i, \gamma}\right| \leq \dfrac{\kappa}{\epsilon}\sum_{i=1}^{N}i^{1/3}\left(\frac{C(s,r)}{\tau\delta^2}|u|_{s, r}^{2}|f_i|_{s, r}^{2}+\frac{\tau}{\delta^2}\left|\bar{u}^{\gamma}\right|^{2}\right) \leq \frac{C}{\tau} E G_{1}+\tau G_{1}.
		\end{aligned}
		$$
		
		In conclusion, we have the energy estimate
		\begin{equation}\label{energyestimate}
			\frac{1}{2} \partial_{t} E \leq-(1-\frac{C}{\tau} E-C \tau) G.
		\end{equation}
		Take $\tau=\frac{1}{4 C}$ where $C$ is the constant in \eqref{energyestimate}, and $c_{1}(s, r)=\frac{1}{16C^2}$. 
		Then it follows that $E(t) \leq c_{1}$ for $0 \leq t \leq T^{*}$. By the choice of $\tau$ and $c_{1}$, one has
		$$
		1-C(\tau) E-C \tau \geq 1-\frac{1}{4}-\frac{1}{4}=\frac{1}{2},
		$$
		and thus \eqref{energyestimate} implies
		\begin{equation}\label{ineq:E0}
			\partial_{t} E+G \leq 0,
		\end{equation}
		for $0 \leq t \leq T^{*}$. This prevents $T^{*}$ from being finite. Hence we proved $E(t) \leq c_{1}$ for all $t$ and thus \eqref{ineq:E0} holds for all $t$. Thus $E(t)$ is decreasing in $t$.
		\qed
	\end{proof}
	
	\section* {Appendix A.2 Proof of Lemma  \ref{Lem2.2}}
	\begin{proof}
		One can write the evolution equation of $\partial^{\alpha} f_i^{\gamma}$ as
		\begin{equation}\label{eq:fgamma}
			\begin{aligned}
				\partial_{t} \partial^{\alpha} f_i^{\gamma}+\dfrac{i}{\bar{\theta}}\mathcal{P}_i \partial^{\alpha} f_i^{\gamma}+& \frac{i^{1/3}}{\bar{\theta}\epsilon}\left(\mathcal{K}_i^{*} \cdot \mathcal{K}_i\right) \partial^{\alpha} f_i^{\gamma}= \frac{i^{1/3}}{\bar{\theta}\epsilon\delta} \partial^{\alpha} u^{\gamma} \cdot v \sqrt{\mu_i} \\
				&+\frac{i^{1/3}}{\bar{\theta}\epsilon} \sum_{0 \leq \eta \leq \alpha} \sum_{0 \leq \beta \leq \gamma}\left(\begin{array}{l}
					\gamma \\
					\beta
				\end{array}\right)\left(\begin{array}{l}
					\alpha \\
					\eta
				\end{array}\right) \partial^{\eta} u^{\beta} \cdot \mathcal{K}_i^{*} \partial^{\alpha-\eta} f_i^{\gamma-\beta}.
			\end{aligned}
		\end{equation}
		Take the $(\cdot, \cdot)$ inner product of \eqref{eq:fgamma} with $\partial^{\alpha} f_i^{\gamma}$. 
		
		For the linear terms, by the same argument as the proof of Proposition $4.1$ of \cite{GoudonHe2010}, one gets
		$$
		\begin{aligned}
			i^{1/3}\left(\mathcal{P}_i \partial^{\alpha} f_i^{\gamma}, \partial^{\alpha} f_i^{\gamma}\right)=& 2 \left\langle\mathcal{S}_i \partial^{\alpha} f_i^{\gamma}, \mathcal{K}_i \partial^{\alpha} f_i^{\gamma}\right\rangle+\epsilon^2\left|\mathcal{S}_i \partial^{\alpha} f_i^{\gamma}\right|_{0}^{2} \geq \frac{3}{4}\epsilon^2\left|\mathcal{S}_i \partial^{\alpha} f_i^{\gamma}\right|_{0}^{2}- \frac{4}{\epsilon^2} \left|\mathcal{K}_i \partial^{\alpha} f_i^{\gamma}\right|_{0}^{2}, \\
			\dfrac{i^{1/3}}{\epsilon}\left(\mathcal{K}_i^{*} \cdot \mathcal{K}_i \partial^{\alpha} f_i^{\gamma}, \partial^{\alpha} f_i^{\gamma}\right)=& \dfrac{2}{\epsilon} \left|\mathcal{K}_i \partial^{\alpha} f_i^{\gamma}\right|_{0}^{2}+\dfrac{2}{\epsilon} \left|\mathcal{K}_i^{2} \partial^{\alpha} f_i^{\gamma}\right|_{0}^{2} +\epsilon\left\langle\mathcal{K}_i \partial^{\alpha} f_i^{\gamma}, \mathcal{S}_i \partial^{\alpha} f_i^{\gamma}\right\rangle\\
			&+2\epsilon \left\langle\mathcal{K}_i^{2} \partial^{\alpha} f_i^{\gamma}, \mathcal{S}_i \mathcal{K}_i \partial^{\alpha} f_i^{\gamma}\right\rangle +\epsilon^2\left|\mathcal{S}_i \mathcal{K}_i \partial^{\alpha} f^{\gamma}\right|_{0}^{2}\\
			\geq & \dfrac{3}{2} \left|\mathcal{K}_i \partial^{\alpha} f_i^{\gamma}\right|_{0}^{2}-\frac{\epsilon^2}{2}\left|\mathcal{S}_i \partial^{\alpha} f^{\gamma}\right|_{0}^{2}+\frac{1}{2} \left|\mathcal{K}_i^{2} \partial^{\alpha} f^{\gamma}\right|_{0}^{2}+\frac{\epsilon^2}{3}\left|\mathcal{S}_i \mathcal{K}_i \partial^{\alpha} f^{\gamma}\right|_{0}^{2}, \\
			\dfrac{i^{1/3}}{\epsilon}\left|\left(\partial^{\alpha} u^{\gamma} \cdot v \sqrt{\mu_i}, \partial^{\alpha} f_i^{\gamma}\right)\right|=& \left| \dfrac{2}{\epsilon} \left\langle\mathcal{K}_i\left(\partial^{\alpha} u^{\gamma} \cdot v \sqrt{\mu_i}\right), \mathcal{K}_i \partial^{\alpha} f_i^{\gamma}\right\rangle+\epsilon\left\langle\mathcal{K}_i\left(\partial^{\alpha} u^{\gamma} \cdot v \sqrt{\mu_i}\right), \mathcal{S}_i \partial^{\alpha} f_i^{\gamma}\right\rangle \right. \\
			&\left. +\epsilon\left\langle\mathcal{S}_i\left(\partial^{\alpha} u^{\gamma} \cdot v \sqrt{\mu_i}\right), \mathcal{K}_i \partial^{\alpha} f_i^{\gamma}\right\rangle+\epsilon^2\left\langle\mathcal{S}_i\left(\partial^{\alpha} u^{\gamma} \cdot v \sqrt{\mu_i}\right), \mathcal{S}_i \partial^{\alpha} f_i^{\gamma}\right\rangle \right| \\
			\leq & \tau\left(\dfrac{1}{\epsilon^2}\left|\mathcal{K}_i \partial^{\alpha} f_i^{\gamma}\right|_{0}^{2}+\epsilon^2\left|\mathcal{S}_i \partial^{\alpha} f_i^{\gamma}\right|_{0}^{2}\right)+\frac{C_1}{\tau}\left(|u|_{s, r}^{2}+\left|\nabla_{x} u\right|_{s, r}^{2}\right) .
		\end{aligned}
		$$
		The notation $|\cdot|_{0}$ is defined by \eqref{Sobolevnorm} with $s=r=0$, i.e., taking $L_{x, v}^{2}$ norm for a fixed $z$. 
		
		For the nonlinear term (the summation), in view of the lemma below (Lemma \ref{lemAppendix}),  one gets
		$$
		\begin{aligned}
			& \dfrac{i^{1/3}}{\epsilon}\left|\left(\partial^{\eta} u^{\beta} \cdot \mathcal{K}_i^{*} \partial^{\alpha-\eta} f_i^{\gamma-\beta}, \partial^{\alpha} f_i^{\gamma}\right)\right| \\ =&
			\left| \frac{2}{\epsilon}\left\langle \partial^{\eta} u^{\beta} \cdot K^{*} \partial^{\alpha-\eta} f_i^{\gamma-\beta}, K^{2} \partial^{\alpha} f_i^{\gamma}\right\rangle+\frac{2}{\epsilon}\left\langle \partial^{\eta} u^{\beta} \partial^{\alpha-\eta} f_i^{\gamma-\beta}, K \partial^{\alpha} f_i^{\gamma}\right\rangle\right.\\
			&+\epsilon\left\langle K\left(\partial^{\eta} u^{\beta} \cdot K^{*} \partial^{\alpha-\eta} f_i^{\gamma-\beta}\right), S \partial^{\alpha} f_i^{\gamma}\right\rangle 
			+\epsilon\left\langle S\left(\partial^{\eta} u^{\beta} \cdot K^{*} \partial^{\alpha-\eta} f_i^{\gamma-\beta}\right), K \partial^{\alpha} f_i^{\gamma}\right\rangle\\
			&\left.+\epsilon^2\left\langle S\left(\partial^{\eta} u^{\beta} \cdot K^{*} \partial^{\alpha-\eta} f_i^{\gamma-\beta}\right), S \partial^{\alpha} f_i^{\gamma}\right\rangle \right| \\
			\leq & C_2 \left|\partial^{\eta} u^{\beta}\right|_{3,0}\frac{1}{\epsilon^2}\left(\left[\partial^{\alpha-\eta} f_i^{\gamma-\beta}, \partial^{\alpha-\eta} f_i^{\gamma-\beta}\right]+\left[\partial^{\alpha} f_i^{\gamma}, \partial^{\alpha} f_i^{\gamma}\right]\right) \\
			\leq & C_2  \frac{1}{\epsilon^2} |u|_{s+3, r}[f_i, f_i]_{s, r},
		\end{aligned}
		$$
		where we used the fact that the $x$ and $z$ derivatives commute with the operators $\mathcal{K}_i$ and $\mathcal{S}_i$.

		A combination of the above estimates yields
		$$
		\begin{aligned}
			&\frac{1}{2} \partial_{t}\left(\partial^{\alpha} f_i^{\gamma}, \partial^{\alpha} f_i^{\gamma}\right)+\dfrac{1}{\bar{\theta}}\left( \dfrac{\epsilon^2}{4}\left|\mathcal{S}_i \partial^{\alpha} f_i^{\gamma}\right|_{0}^{2} + \dfrac{1}{2}\left|\mathcal{K}_i^{2} \partial^{\alpha} f_i^{\gamma}\right|_{0}^{2}+\frac{\epsilon^2}{3}\left|\mathcal{S K} \partial^{\alpha} f_i^{\gamma}\right|_{0}^{2}-\left(\frac{4}{\epsilon^2}-\frac{3}{2}\right)\left|\mathcal{K}_i \partial^{\alpha} f_i^{\gamma}\right|_{0}^{2}\right)\\
			&\quad \leq \frac{\tau}{\delta\bar{\theta}}\left(\frac{1}{\epsilon^2}\left|\mathcal{K}_i \partial^{\alpha} f_i^{\gamma}\right|_{0}^{2}+\epsilon^2\left|\mathcal{S}_i \partial^{\alpha} f_i^{\gamma}\right|_{0}^{2}\right)+\frac{C_1}{\tau\delta\bar{\theta}}\left(|u|_{s, r}^{2}+\left|\nabla_{x} u\right|_{s, r}^{2}\right)+ \frac{C_2}{\bar{\theta}\epsilon^2}|u|_{s+3, r}[f_i, f_i]_{s, r}.
		\end{aligned}
		$$
		Summing over $\alpha, \gamma$, and choosing $\tau=\frac{\delta}{8} $ to absorb the term $\left|\mathcal{S}_i \partial^{\alpha} f_i^{\gamma}\right|_{0}^{2}$ on the RHS by the same term on the LHS,  one gets
		$$
		\partial_{t}(f_i, f_i)_{s, r}+\frac{1}{\bar{\theta}\epsilon^2}\left(\frac{1}{8}-MC_2|u|_{s+3, r}\right) [f_i, f_i]_{s, r} \leq \frac{C_3}{\bar{\theta}}\left(\frac{1}{\delta^2}|u|_{s, r}^{2}+\frac{1}{\delta^2}\left|\nabla_{x} u\right|_{s, r}^{2}+\frac{1}{\epsilon^2}|\mathcal{K}_i f_i|_{s, r}^{2}\right),
		$$
		where 
		$M$ is the number of possible pairs $(\alpha, \gamma)$, $C_3 = \max \left\{5, 8M C_1\right\}$.
		Thus if one chooses $c_{1}^{\prime}=\min \left\{c_{1}(s+3, r), \frac{1}{16 MC_2}\right\}$, then by Theorem \ref{thm:energyestimate}, $E_{s+3, r}(t)$ is decreasing and $E_{s+3, r}(t) \leq$ $c_{1}^{\prime}$ for all $t$. Hence $|u|_{s+3, r} \leq E_{s+3, r} \leq c_{1}^{\prime}$ for all $t$ and one gets the conclusion with $\lambda_{1}=\frac{1}{16}$.\qed
	\end{proof}
	
	\setcounter{equation}{0}
	\renewcommand\thelemma{A.\arabic{lemma}}
	
	\begin{lemma}\label{lemAppendix}
		For $f_i$ and $g_i$ orthogonal to $\sqrt{\mu_i}$, 
		$$
		\frac{i^{1/3}}{\epsilon}\left|\left(u \cdot \mathcal{K}_i^{*} f_i, g_i\right)\right| \leq C \frac{1}{\epsilon^2}\|u\|_{H^{3}}([f_i, f_i]+[g_i, g_i]).
		$$
	\end{lemma}
	
	\begin{proof}
		Using the commutator relation $\mathcal{K}_i\left(u \cdot \mathcal{K}_i^* f_i\right)=\mathcal{K}_i^*(u \cdot \mathcal{K}_i f_i)+u f_i$, one has
		\begin{equation}\label{eq:Kfg}
			\begin{aligned}
				&\dfrac{i^{1/3}}{\epsilon}\left(u\cdot K^*f_i,g_i\right)\\=& \frac{2}{\epsilon}\left\langle\mathcal{K}_i\left(u \cdot \mathcal{K}_i^{*} f_i\right), \mathcal{K}_i g_i\right\rangle+\epsilon\left\langle\mathcal{K}_i\left(u \cdot \mathcal{K}_i^{*} f_i\right), \mathcal{S}_i g_i\right\rangle+\epsilon\left\langle\mathcal{S}_i\left(u \cdot \mathcal{K}_i^{*} f_i\right), \mathcal{K}_i g_i\right\rangle+\epsilon^{2}\left\langle\mathcal{S}_i\left(u \cdot \mathcal{K}_i^{*} f_i\right), \mathcal{S}_i g_i\right\rangle \\
				=& \frac{2}{\epsilon}\left\langle u \mathcal{K}_i f_i, \mathcal{K}_i^{2} g_i\right\rangle+\frac{2}{\epsilon}\langle u f_i, \mathcal{K}_i g_i\rangle+\epsilon\langle u \mathcal{K}_i f_i, \mathcal{K}_i \mathcal{S}_i g_i\rangle+\epsilon\langle u f_i, \mathcal{S}_i g_i\rangle +\epsilon\left\langle\mathcal{S}_i(u f_i), \mathcal{K}_i^{2} g_i\right\rangle+\epsilon^{2}\langle\mathcal{S}_i(u f_i), \mathcal{S}_i \mathcal{K}_i g_i\rangle \\
				=& \frac{2}{\epsilon}\left\langle u \mathcal{K}_i f_i, \mathcal{K}_i^{2} g_i\right\rangle+\frac{2}{\epsilon}\langle u f_i, \mathcal{K}_i g_i\rangle+\epsilon\langle u \mathcal{K}_i f_i, \mathcal{K}_i \mathcal{S}_i g_i\rangle+\epsilon\langle u f_i, \mathcal{S}_i g_i\rangle +\epsilon\left\langle(\mathcal{S}_i u) f_i, \mathcal{K}_i^{2} g_i\right\rangle+\epsilon\left\langle u(\mathcal{S}_i f_i), \mathcal{K}_i^{2} g_i\right\rangle \\
				&+\epsilon^{2}\langle(\mathcal{S}_i u) f_i, \mathcal{S}_i \mathcal{K}_i g_i\rangle+\epsilon^{2}\langle u(\mathcal{S}_i f_i), \mathcal{K}_i \mathcal{S}_i g_i\rangle .
			\end{aligned}
		\end{equation}
		
		For each term in \eqref{eq:Kfg}, by using the Cauchy-Schwarz inequality, Lemma 3.4 in \cite{JinLin2022}, and the Sobolev inequality
		$\|u\|_{L^{\infty}}+\left\|\nabla_{x} u\right\|_{L^{\infty}} \leq C\|u\|_{H^{3}},
		$
		one has
		$$
		\begin{aligned}
			\frac{1}{\epsilon}\left\langle u \mathcal{K}_i f_i, \mathcal{K}_i^{2} g_i\right\rangle & \leq  \dfrac{1}{\epsilon}\|u\|_{L^{\infty}}\|\mathcal{K}_i f_i\|_{L^{2}}\|\mathcal{K}_i^2 g_i\|_{L^{2}} \leq C\|u\|_{L^{\infty}}\left(\dfrac{1}{\epsilon^2}\|\mathcal{K}_i f_i\|_{L^{2}}^2+\|\mathcal{K}_i^2 g_i\|_{L^{2}}^2\right), \\
			\dfrac{1}{\epsilon}\langle u f_i, \mathcal{K}_i g_i\rangle & \leq \dfrac{1}{\epsilon}\|u\|_{L^{\infty}}\|f_i\|_{L^{2}}\|\mathcal{K}_i g_i\|_{L^{2}} \leq C\dfrac{1}{\epsilon}\|u\|_{L^{\infty}}\left(\|\mathcal{K}_i f_i\|_{L^{2}}+\epsilon^2\|\mathcal{S}_i f\|_{L^{2}}\right)\|\mathcal{K}_i g_i\|_{L^{2}} \\
			& \leq C\|u\|_{L^{\infty}}\left(\dfrac{1}{\epsilon^2}\|\mathcal{K}_i f_i\|_{L^{2}}^{2}+\|\mathcal{K}_i g\|_{L^{2}}^{2}+\epsilon^2\|\mathcal{S}_i f\|_{L^{2}}^{2}+\|\mathcal{K}_i g_i\|_{L^{2}}^{2}\right), \\
			\epsilon\langle u \mathcal{K}_i f_i, \mathcal{K}_i \mathcal{S}_i g_i\rangle & \leq \epsilon\|u\|_{L^{\infty}}\|\mathcal{K}_i f_i\|_{L^{2}}\|\mathcal{K}_i\mathcal{S}_i g_i\|_{L^{2}} \leq C\|u\|_{L^{\infty}}\left(\|\mathcal{K}_i f_i\|_{L^{2}}^2+\epsilon^2\|\mathcal{K}_i\mathcal{S}_i  g_i\|_{L^{2}}^2\right),\\
			\epsilon\langle u f_i, \mathcal{S}_i g_i\rangle & \leq \epsilon\|u\|_{L^{\infty}}\|f_i\|_{L^{2}}\|\mathcal{S}_i g_i\|_{L^{2}} \leq C \epsilon\|u\|_{L^{\infty}}\left(\|\mathcal{K}_i f_i\|_{L^{2}}+\epsilon^2\|\mathcal{S}_i f_i\|_{L^{2}}\right)\|\mathcal{S}_i g_i\|_{L^{2}} \\
			& \leq C \epsilon\|u\|_{L^{\infty}}\left(\dfrac{1}{\epsilon^2}\|\mathcal{K}_i f_i\|_{L^{2}}^{2}+\epsilon^2\|\mathcal{S}_i g_i\|_{L^{2}}^{2}+\epsilon^2\|\mathcal{S}_i f_i\|_{L^{2}}^{2}+\epsilon^2\|\mathcal{S}_i g_i\|_{L^{2}}^{2}\right).
		\end{aligned}
		$$
		The latter four terms in \eqref{eq:Kfg} can be approximated similarly and one obtains
		$$
		\frac{i^{1/3}}{\epsilon}\left|\left(u \cdot \mathcal{K}_i^{*} f_i, g_i\right)\right| \leq C \frac{1}{\epsilon^2}\|u\|_{H^{3}}([f_i, f_i]+[g_i, g_i]).
		$$\qed
	\end{proof}
	
	\section* {Appendix A.3 Proof of Lemma \ref{thm:hypocoercivity}}
	\begin{proof}
		Multiplying \eqref{ineq:flambda1} by $\lambda_{4}\kappa\bar{\theta}$ ($\lambda_{4}>0$ is a constant to be chosen later), summing over $i$, and adding to \eqref{ineq:E0} yields
		$$
		\partial_{t} \tilde{E}+\tilde{G} \leq \lambda_{4} \tilde{B},
		$$
		where
		$$
		\begin{aligned}
			&\tilde{E}=E+\lambda_{4}\kappa\bar{\theta}\sum_{i=1}^{N}(f_i, f_i)_{s, r}, \quad \tilde{G}=G+\lambda_{4} \lambda_{1}\frac{\kappa}{\epsilon^2} \sum_{i=1}^{N}[f_i, f_i]_{s, r},  \\
			&\tilde{B}=C\left(\lambda_{1}\right)\kappa\sum_{i=1}^{N}\left(\frac{1}{\delta^2}|u|_{s, r}^{2}+\frac{1}{\delta^2}\left|\nabla_{x} u\right|_{s, r}^{2}+\frac{1}{\epsilon^{2}}|\mathcal{K}_i f_i|_{s, r}^{2}\right) . 
		\end{aligned}
		$$
		It is clear that 
		$$\tilde{B} \leq C\left(G+\kappa\sum_{i=1}^{N}\frac{1}{\epsilon^{2}}|\mathcal{K}_i f_i|_{s, r}^{2}\right)
		\leq C \tilde{G}.$$ Thus by choosing $\lambda_{4}=\min \left\{\frac{1}{2 C}, 1\right\},\, C$ being the previous constant, one gets
		\begin{equation}\label{ineq:Etilde0}
			\partial_{t} \tilde{E}+\frac{1}{2} \tilde{G} \leq 0.
		\end{equation}
		It yields from Lemma 3.4 in \cite{JinLin2022} and definition that
		$$
		|f_i|_{s, r}^{2} \leq C\left(|\mathcal{K}_i f_i|_{s, r}^{2}+\epsilon^2|\mathcal{S}_i f_i|_{s, r}^{2}\right)
		$$
		and
		$$
		(f_i, f_i)_{s, r} \leq C\left(|\mathcal{K}_i f_i|_{s, r}^{2}+\epsilon^2|\mathcal{S}_i f_i|_{s, r}^{2}\right) \leq C \frac{1}{\epsilon^{2}}[f_i, f_i]_{s, r}.
		$$
		Thus
		\begin{equation}\label{ineq:EtildeG}
			\tilde{E} \leq C\left(G+\kappa\bar{\theta}\sum_{i=1}^{N}|f_i|_{s, r}^{2}\right)+\lambda_{4}\kappa\bar{\theta}\sum_{i=1}^{N}(f_i, f_i)_{s, r} \leq C\left(G+\sum_{i=1}^{N}\left(|\mathcal{K}_i f_i|_{s, r}^{2}+\epsilon^2|\mathcal{S}_i f_i|_{s, r}^{2}\right)\right) \leq C \tilde{G}.
		\end{equation}
		This together with \eqref{ineq:Etilde0} implies
		$$
		\tilde{E}(t) \leq \tilde{E}(0) e^{-\lambda t},
		$$
		where $\lambda=\frac{1}{2 C}$, $C$ being the constant in \eqref{ineq:EtildeG}.
		Finally, the proof of Theorem \ref{thm:hypocoercivity} is completed by noticing that
		$$
		E(t) \leq \tilde{E}(t) \leq \tilde{E}(0) e^{-\lambda t} \leq\left(E(0)+C^{h}\right) e^{-\lambda t}.
		$$ \qed
	\end{proof}

	\bibliography{mybibfile}

\end{document}